\documentclass[12pt,a4paper]{amsart}
\usepackage{amsfonts}
\usepackage{amsthm}
\usepackage{amsmath}
\usepackage{amscd}
\usepackage[latin2]{inputenc}
\usepackage{t1enc}
\usepackage[mathscr]{eucal}
\usepackage{indentfirst}
\usepackage{caption}
\usepackage{graphicx}
\usepackage{subfig}
\usepackage{float}
\usepackage{graphics}
\usepackage{pict2e}
\usepackage{epic}
\numberwithin{equation}{section}
\usepackage[margin=2.8cm, top=2.8cm,heightrounded]{geometry}
\usepackage{epstopdf} 
\usepackage{enumitem}

\usepackage{hyperref}
\hypersetup{
    colorlinks=true,
    linkcolor=blue,
    filecolor=magenta,      
    urlcolor=cyan,
    citecolor=ForestGreen,
    pdftitle={Surfaces in Montesinos Knots},
    pdfpagemode=FullScreen,
    }

\usepackage[dvipsnames]{xcolor}

\theoremstyle{plain}
\newtheorem{Th}{Theorem}[section]
\newtheorem{Lemma}[Th]{Lemma}
\newtheorem{Cor}[Th]{Corollary}
\newtheorem{Prop}[Th]{Proposition}

 \theoremstyle{definition}
\newtheorem{Def}[Th]{Definition}

\newcommand{\Z}{\mathbb{Z}}

\begin{document}

\title[Number of Closed Essential Surfaces in Montesinos Knots]{The Number of Closed Essential Surfaces in Montesinos Knots with Four Rational Tangles}

\author[B. Basilio]{Brannon Basilio}

\address{University of Illinois at Urbana-Champaign \\ Department of Mathematics \\
Champaign IL 61820} 

\email{basilio3@illinois.edu}



\begin{abstract}
	In the complement of a hyperbolic Montesinos knot with 4 rational tangles, we investigate the number of closed, connected, essential, orientable surfaces of a fixed genus $g$, up to isotopy. We show that there are exactly 12 genus 2 surfaces and $8\phi(g - 1)$ surfaces of genus greater than 2, where $\phi(g - 1)$ is the Euler totient function of $g - 1$. Observe that this count is independent of the number of crossings of the knot. Moreover, this class of knots form an infinite class of hyperbolic 3-manifolds and the result applies to all such knot complements.
\end{abstract}

\maketitle

\tableofcontents

%
%
\section{Motivation}
Closed essential surfaces are a widely used tool for studying the geometry and topology of 3-manifolds. In short, these are surfaces, $S$, without boundary that inject into the fundamental group of the 3-manifold $M$, i.e. the map $\pi_1 S \hookrightarrow \pi_1M$ is injective, and $S$ is not boundary parallel. The exact definitions and conventions used in this section can be found in Section 2. The existence of distinct closed essential surfaces is not guaranteed, for example \cite{Hatcher1985IncompressibleSI} shows there is only one closed incompressible surface in 2-bridge knot complements. However, Oertel \cite{oertel1984} showed that for Montesinos knots with $q_i \geq 3$ for each $i$ and more than four rational tangles, the complement contains closed essential surfaces of every genus greater than 2. Moreover, by Corollary~2.3 of \cite{JACO1984195}, irreducible and atoroidal manifolds contain a finite number of essential surfaces of a fixed topological type. In this paper we focus on hyperbolic Montesinos knots in $S^3$, so that the complement is irreducible and atoroidal by \cite{bams1183548782}, giving finitely many closed essential surfaces.

A natural problem is to investigate how many closed essential surfaces there are in a 3-manifold and in particular, which ones are connected. We present the main theorem of this paper which gives the exact number of closed, connected, essential, orientable surfaces in a class of Montesinos knot complements.

\begin{Th} \label{main} 
	The number of closed, connected, essential, orientable surfaces of genus $g$ in the complement of a hyperbolic Montesinos knot $K = K(p_1/q_1,\ldots,p_k/q_k)$ in $S^3$ with each $q_i\geq 3$ and $k=4$, is 
\[
	\begin{cases}
		12, & \text{if } g = 2\\
		8\phi(g - 1), & \text{if } g \geq 3 
	\end{cases}
\] where $\phi(g - 1)$ is the Euler totient function of $g - 1$.
\end{Th}

This class of knots include alternating Montesinos knots with arbitrary many crossings. In comparison to this result, Kahn and Markovic \cite{MR2916295} show that the number of essential immersed surfaces in a closed hyperbolic 3-manifold is bounded by $g^{2g}$. Dunfield, Garoufalidis, and Rubinstein \cite{dunfield2020counting} showed how to count closed essential surfaces in 3-manifolds, relying on work of Oertel \cite{oertel1984} and Tollefson \cite{ojm1200786486}, along with Ehrhart's method (see e.g. \cite{JADeLoera}) of counting lattice points in polytopes. Lee \cite{lee2021essential} proved the following conjecture of \cite{dunfield2020counting}: the knot $K13n586$ has exactly $\phi(g - 1)$ closed, connected, orientable, essential surfaces of genus $g$ up to isotopy, where $\phi(g - 1)$ is the Euler totient function. Hass, Thompson, and Tsvietkova \cite{MR3658176} found a polynomial bound, in terms of the number of crossings of the link, for the number of closed incompressible surfaces in an alternating link complement. Note that the bound found in \cite{MR3658176} is exponential in terms of the genus of the surface and dependent of the number of crossings of the knot. In contrast, the count in Theorem~\ref{main} is in terms of the Euler totient function of the genus of the surface and independent of the number of crossings of the knot. 

Sections~\ref{Background} and \ref{Convention} introduce background and conventions, respectively, for this paper. Section~\ref{Idea} gives the main ideas and tools used for the proof of Theorem~\ref{main}. Section~\ref{proof} gives definitions and an example of how to construct closed essential surfaces from punctured partitioned chord diagrams (see Section~\ref{proof}, Definition~\ref{DefPPCD}). Section \ref{DescriptionOfOrbifold} recalls a description, given in \cite{oertel1984}, of the complement of a Montesinos knot viewed as an orbifold fibration. Section~\ref{g2case} gives an upper bound on the number of for genus 2 surfaces and Section~\ref{generalcase} gives an upper bound for genus greater than 2. Section~\ref{DistinctnessOfPPCDs} shows that each of these surfaces are distinct, i.e. non-isotopic, giving Theorem~\ref{main}.

%
%
\section{Background} \label{Background}
We now review Montesinos knots and their complements. A \textit{rational tangle of slope $p/q \in \mathbb{Q}$} is constructed by taking lines of slope $p/q$ on the square ``pillowcase'' starting at the four corners (see Figure~\ref{Rational_Tangle}). A \textit{Montesinos link} $K=K(p_1/q_1,\ldots,p_k/q_k)$ is a link obtained by taking $k$ rational tangles $p_1/q_1,\ldots p_k/q_k$ arranged counterclockwise (equivalently arranged in a line) with two parallel arcs connecting each pair of consecutive tangles (see Figure~\ref{Montesinos_Knot}). One can construct a 4-punctured sphere $S_C$ in a Montesinos link complement by first projecting $K$ to the plane sphere as in Figure~\ref{Montesinos_Knot} (left), then taking a simple closed curve $c$ which intersects $K$ transversely at four points on the arcs connecting rational tangle regions (see Figure~\ref{S_Cs}), and then capping $c$ with two disks, one above the projection plane sphere and one below. As will be shown (Proposition~\ref{NonisotopicSCs}), Figure~\ref{S_Cs} shows the only two possible incompressible $S_C$'s, up to isotopy, in Montesinos link complements with four rational tangles.

\begin{figure}[h!]
\centering
	\begin{minipage}[b]{0.39\linewidth}
		\centering
		\includegraphics[viewport = -20 140 675 675, scale = 0.21, clip]{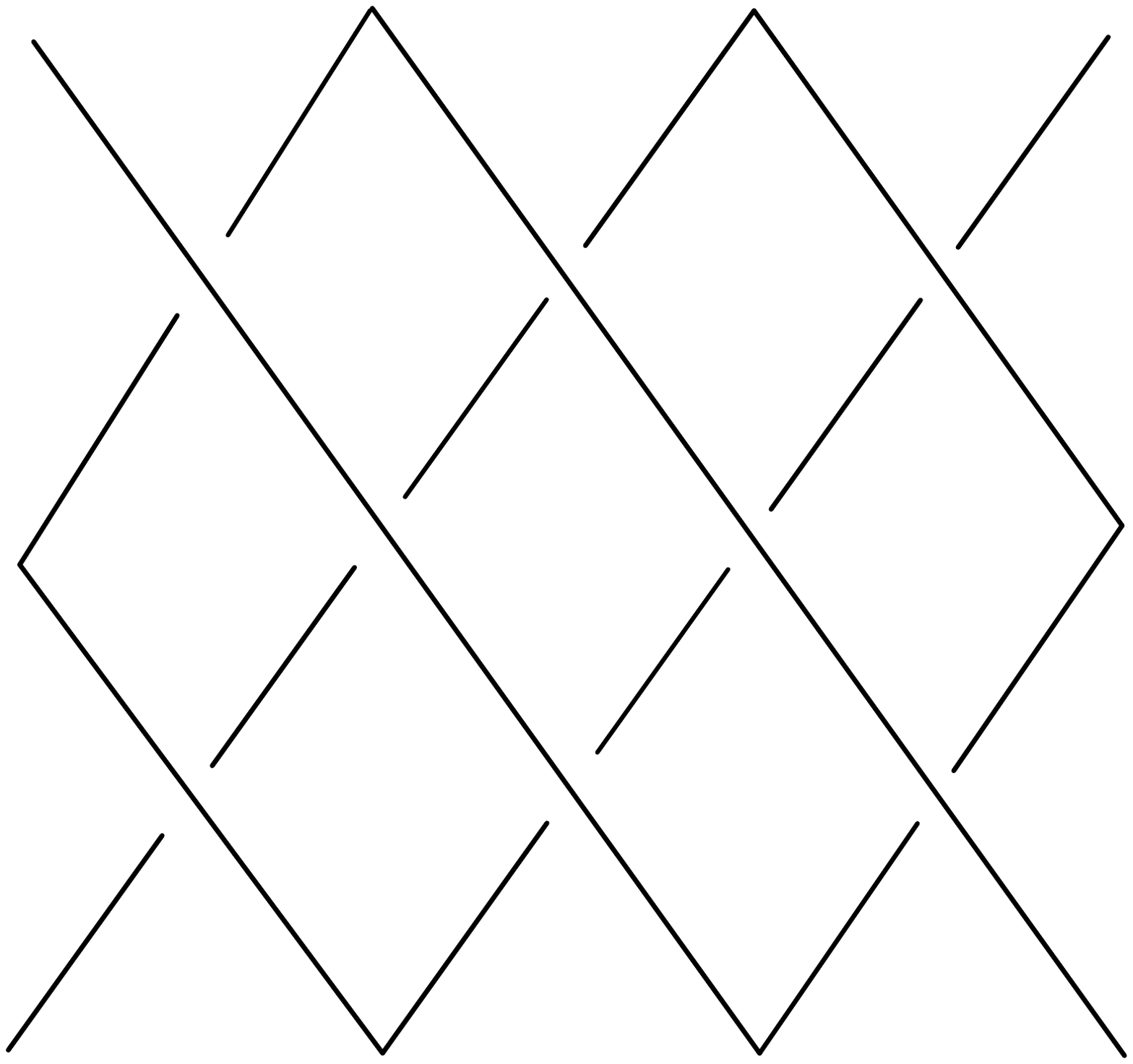}
		\caption{Rational tangle with slope $\frac{2}{3}$.}
		\label{Rational_Tangle}
	\end{minipage}
	\begin{minipage}[b]{0.6\linewidth}
		\begin{minipage}[b]{0.47\linewidth}
		\centering
		\includegraphics[viewport = -15 180 650 650, scale = 0.21, clip]{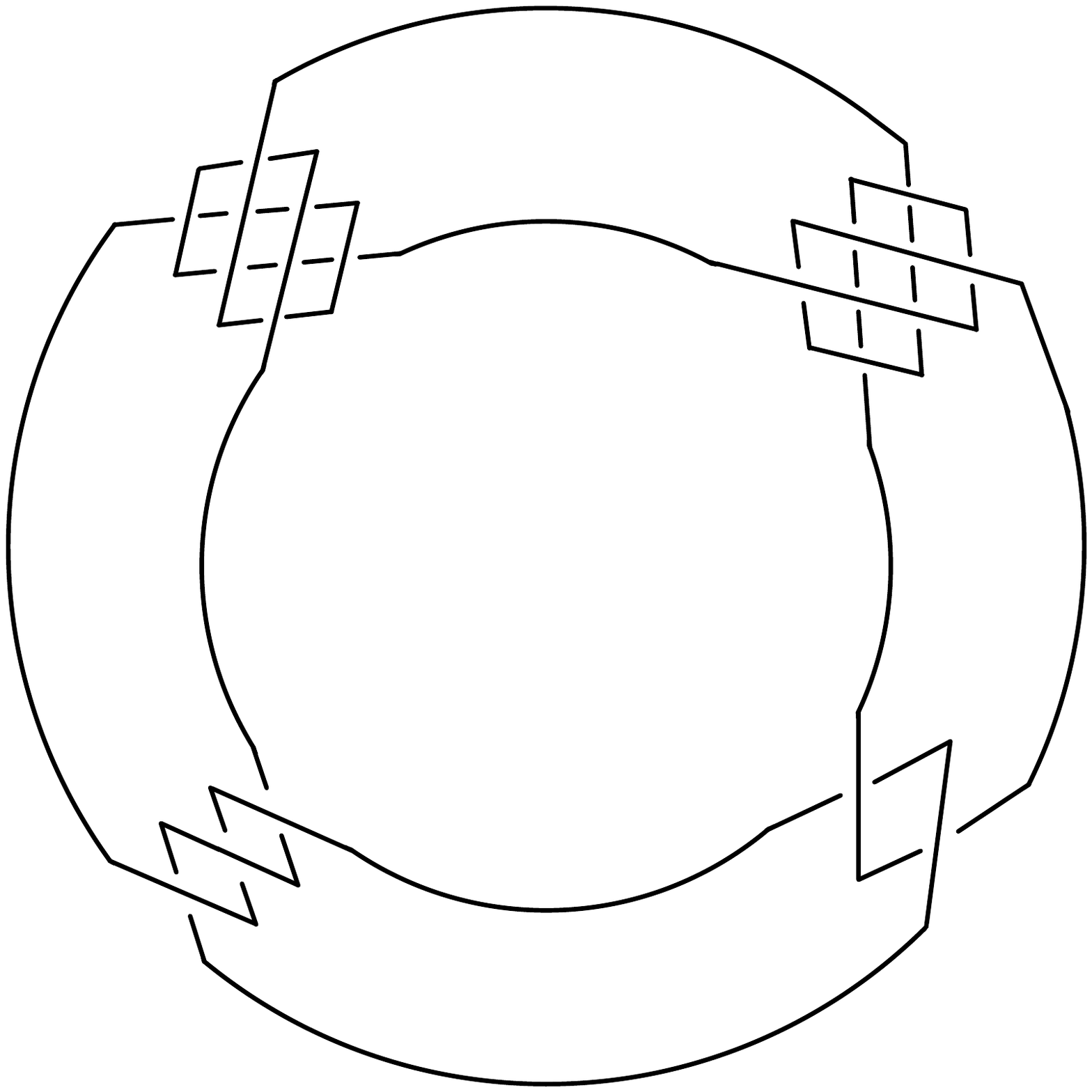}
		\end{minipage}
		\begin{minipage}[b]{0.43\linewidth}
		\centering
		\includegraphics[viewport = 20 290 750 750, scale = 0.225, clip]{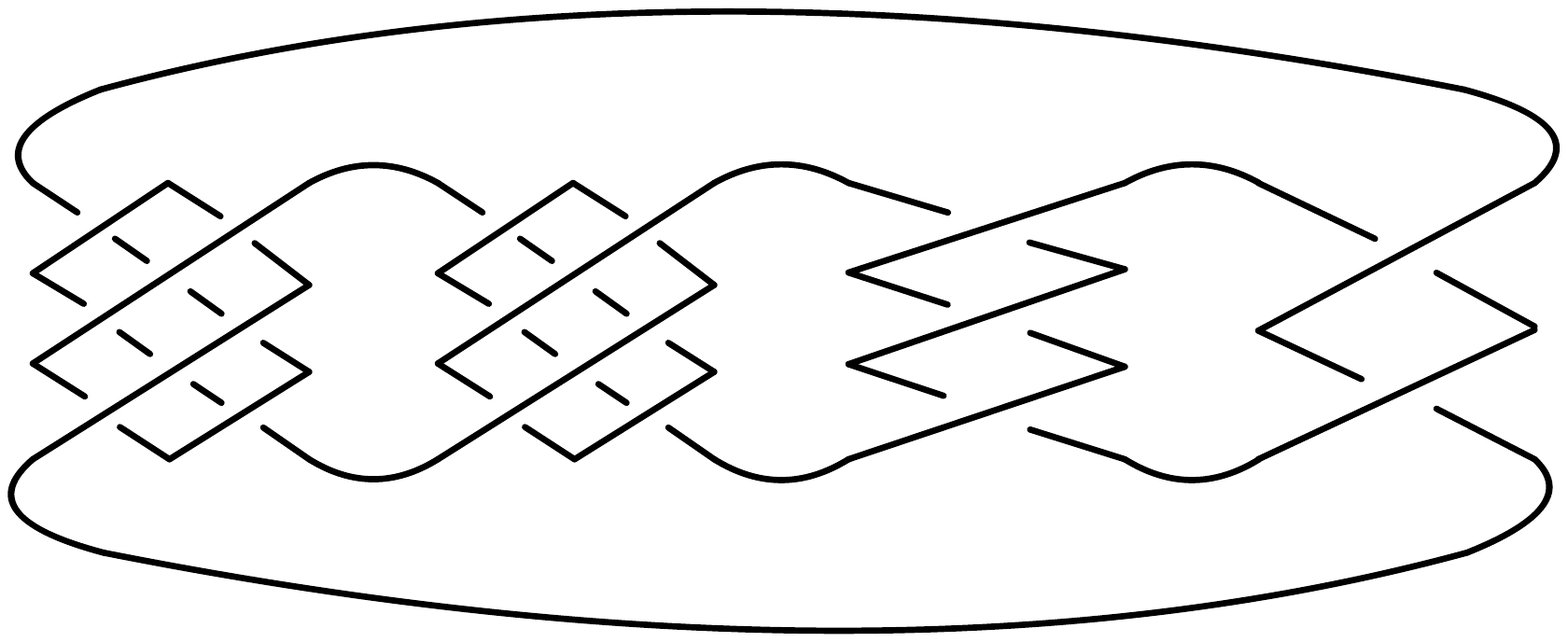}
		\end{minipage}
		\caption{Montesinos knot $K(\frac{2}{3}, \frac{1}{3}, \frac{1}{2}, \frac{2}{3})$ expressed in two equivalent ways.}
		\label{Montesinos_Knot}
	\end{minipage}
\end{figure}

\begin{figure}[h!]
\centering
	\includegraphics[viewport = 0 120 750 750, scale = 0.25, clip]{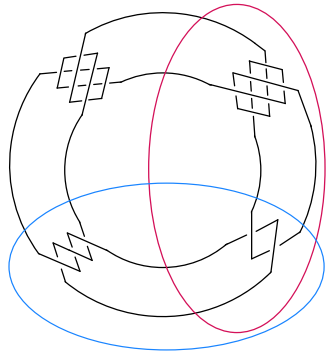}
	\caption{The only two possible 4-punctured spheres $S_C$'s represented by simple closed curves on the projection sphere.}
	\label{S_Cs}
\end{figure}

A \textit{compressing disk} for an embedded surface $S$ in a manifold $M$ is an embedded disk $D \subset M$ such that $D \cap S = \partial D$ and $D$ does not bound a disk in $S$. A \textit{closed incompressible surface} is a compact surface without boundary which does not have a compressing disk. Moreover, a \textit{closed essential surface} is a closed incompressible surface that is not parallel to the boundary. Note that the boundary of a knot complement is homeomorphic to a torus. A \textit{peripheral tubing operation} is the process of attaching an annulus between two punctures of $S_C$'s (see Figure~\ref{PeriphTubing}). A \textit{Seifert tangle} $(B, p_1/q_1, \ldots, p_k/q_k)$ consists of a ball $B$ containing two embedded arcs and possibly some embedded closed curves. These arcs and curves decompose as rational tangles $p_1/q_1,\ldots,p_k/q_k$ where attached each successive rational tangle are parallel strands, except for the first and last rational tangle (see Figure~\ref{SeifertTangle}).

\begin{figure}[h!]
\centering
	\begin{minipage}[b]{0.44\linewidth}
		\includegraphics[viewport = 20 250 700 700, scale = 0.32, clip]{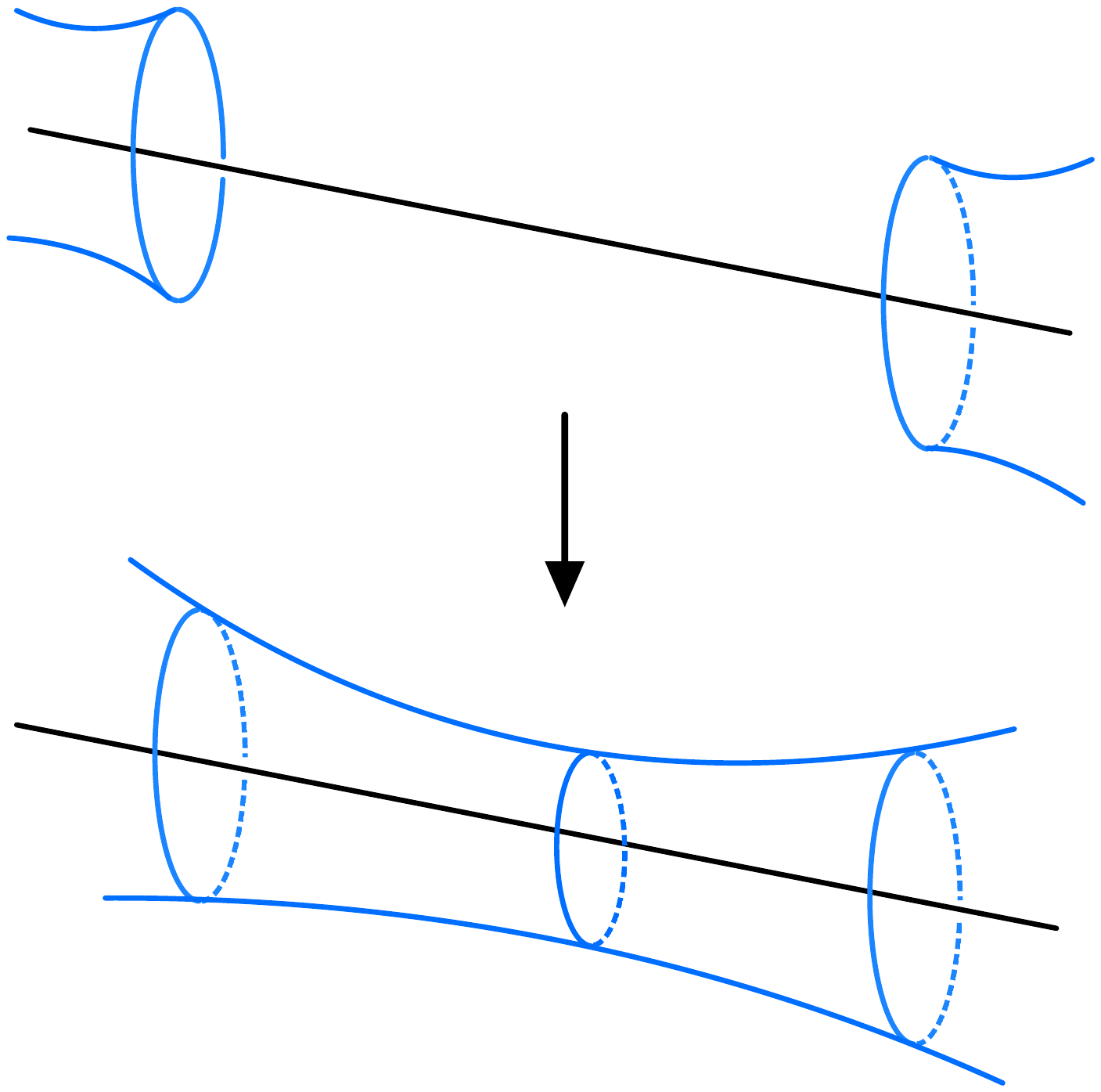}
		\caption{Peripheral tubing operation.}
		\label{PeriphTubing}
	\end{minipage}
	\begin{minipage}[b]{0.44\linewidth}
		\centering
		\includegraphics[viewport = 45 310 600 660, scale = 0.32, clip]{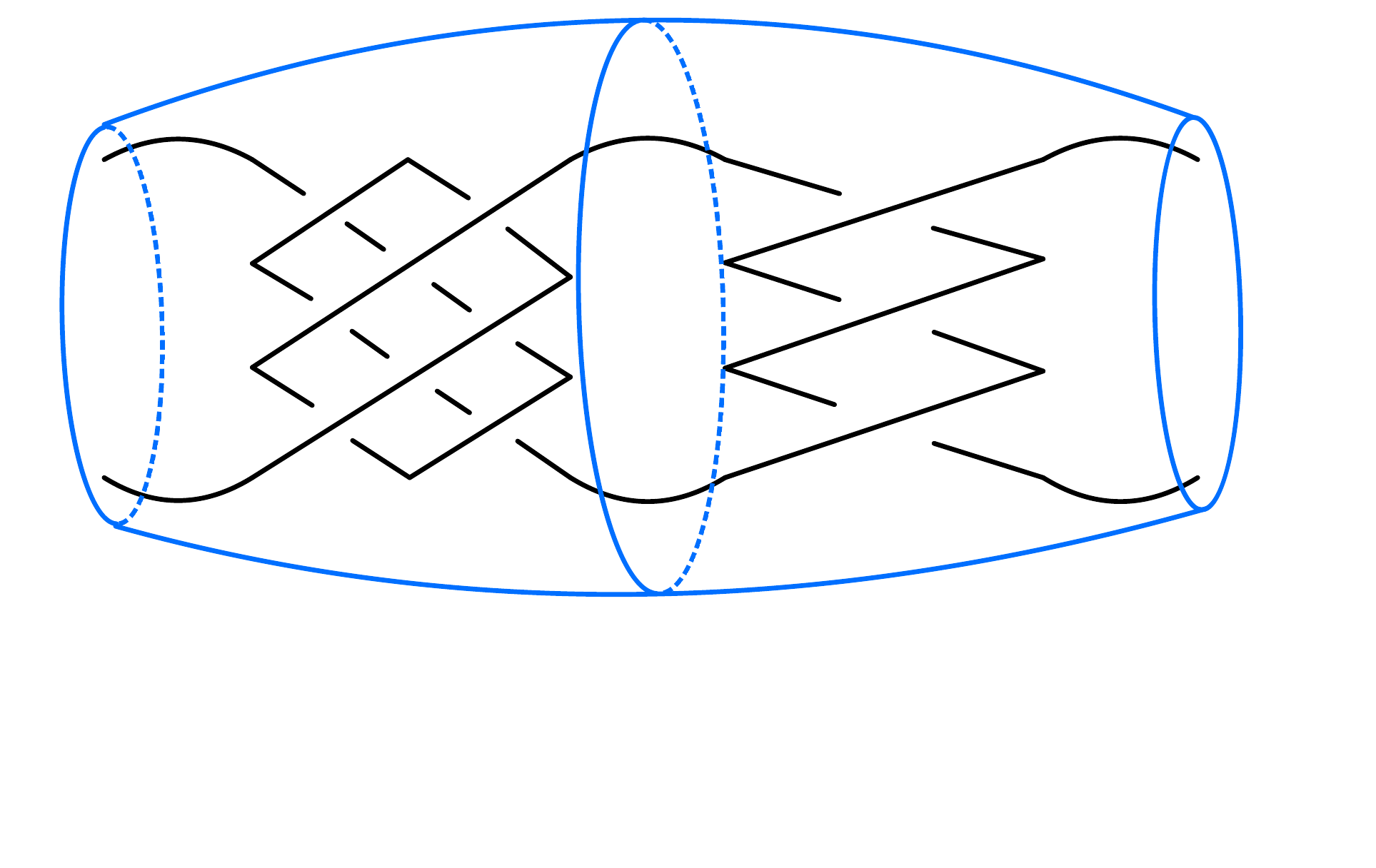}
		\caption{Seifert tangle $(B, \frac{2}{3}, \frac{1}{3})$.}
		\label{SeifertTangle}
	\end{minipage}
\end{figure}

Now, we turn to the definitions needed in order to count closed essential surfaces in Montesinos knot complements. A \textit{chord diagram} is an oriented circle, bounding a disk, containing the following: 

	\begin{enumerate}[label = \roman*), font = \itshape]
	\item finitely many \textit{base points} on the boundary circle,
	\item straight lines in between two base points, called \textit{chords}, contained in the interior of the disk with minimal intersection with other chords,
	\item any two chords have distinct base points and all base points belong to a unique chord.
	\end{enumerate}
See Figure~\ref{Chord_Diagram} for an example of a chord diagram. The \textit{length of a chord}, with base points say $b_1$ and $b_2$, in a chord diagram is the minimum number of base points between $b_1$ and $b_2$ augmented by 1. See for example the chord labelled $C$ in Figure~\ref{Chord_Diagram}.

\begin{figure}[h!]
\centering
\includegraphics[viewport = 0 150 650 650, scale = 0.3, clip]{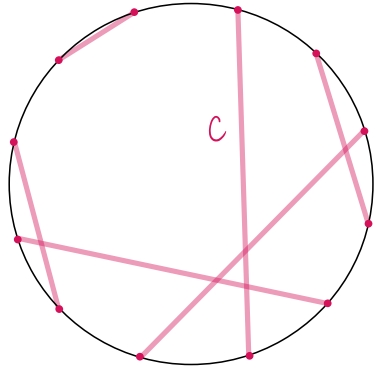}
\caption{An example of a chord diagram, with chord $C$ having length 5.}
\label{Chord_Diagram}
\end{figure}

We state three results of Oertel \cite{oertel1984} which are utilized throughout this paper:

\begin{Th}[Oertel \cite{oertel1984}, Theorem 1] \label{Oertel1}
	If $\sum_{i=1}^k p_i/q_i \neq 0$, in particular if $K$ is a knot, then every closed essential surface in $S^3 - K$ is isotopic to a surface obtained from a finite collection of disjoint incompressible spheres $S_C$ by a sequence of peripheral tubing operations. When $\sum_{i=1}^k p_i/q_i = 0$, there is, in addition, just one other isotopy class of closed essential surfaces in $S^3 - K$; surfaces in this class have Euler characteristic $l(2-k+\sum_{i=1}^k 1/q_i = 0)$ where $l=$l.c.m$(q_1,\ldots,q_k)$.
\end{Th}

\begin{Th}[Oertel \cite{oertel1984}, Theorem 2] \label{Oertel2}
	If $q_i \geq 3$ for each $i$, then a surface obtained from disjoint incompressible $S_C$'s by a sequence of tubing operations is incompressible if and only if each tube passes through at least one rational tangle.
\end{Th}

\begin{Cor}[Oertel \cite{oertel1984}, Corollary 2.14]\label{Oertel3} A vertical 4-punctured sphere in $O_K$, the fibered orbifold with underlying space $S^3$ and $\mathbb{Z}_2$-singular set $K=K(p_1/q_1,\ldots,p_k/q_k),$ is incompressible as an orbifold if and only if it bounds a Seifert tangle on each side which is not a rational tangle.
\end{Cor}

%
%
\section{Convention} \label{Convention}
In order to utilize Theorem~\ref{Oertel2}, throughout this paper, every Montesinos knot $K$ will have $q_i \geq 3$ for each $i$. We prove which Montesinos knots are hyperbolic.

\begin{Prop} A Montesinos knot $K = K(p_1/q_1,\ldots,p_k/q_k)$ is hyperbolic if $q_i \geq 3$ for each $i$ and $k \geq 4$.
\end{Prop}
\begin{proof}
	We show that the knot is neither a satellite knot nor a torus knot so by Thurston's hyperbolization theorem for Haken manifolds, the knot complement is hyperbolic. First, since $S^3$ is irreducible, then so is the knot complement as every 2-sphere in $S^3$ bounds a ball to both sides and hence every 2-sphere in $S^3-K$ bounds a ball to one side. By Theorem~\ref{Oertel1}, every essential surface $S$ is a tubing of 4-punctured spheres. Since every tubing of 4-punctured spheres has $\chi(S) < 0$, then the knot complement is geometrically atoroidal so $K$ is not a satellite knot.
	
	We now show that $K$ is not a torus knot. By assumption we have $q_i \geq 3$ for each $i$ and $k \geq 4$, hence Corollary 3 of \cite{oertel1984} implies that the knot complement of $K$ contains closed essential surfaces of every genus greater than 2. However, by Waldhausen (see e.g. \cite{Lyon1971}), torus knots contain only one closed incompressible surface which is boundary parallel. Thus, $K$ cannot be a torus knot. So, we may apply Thurston's hyperbolization theorem to conclude that $K$ is hyperbolic.
\end{proof}

In light of the proposition, all Montesinos knots in this paper are hyperbolic. Moreover, for simplicity, we also assume that every Montesinos knot has $k=4$ rational tangles. Hyperbolicity ensures finiteness of the number of surfaces of genus $g$ in the complement of $K$ in $S^3$ up to isotopy, by Corollary 2.3 of \cite{JACO1984195}. We will also only count closed, connected, essential, orientable surfaces, up to isotopy, embedded in the complement.

%
%
\section{Idea of the Proof}\label{Idea}

We use Theorem~\ref{Oertel1} to prove that it allows us to define a punctured partitioned chord diagram. That is, punctured partitioned chord diagrams record the tubing information: which punctures of the $S_C$'s are tubed and which rational tangles each tube goes through. We then show that each tubing description of 4-punctured $S_C$'s, as in Theorem~\ref{Oertel1}, of a closed essential surface in the complement corresponds to a unique punctured partitioned chord diagram. Noting that each chord diagram can correspond to multiple surfaces, we show that each punctured partitioned chord diagram corresponds to at most 2 closed essential surfaces. Taking the dual tree of these chord diagrams, which is unique, we use these to show the types of allowable punctured partitioned chord diagrams. Namely, every punctured partitioned chord diagram coming from a closed essential genus $g$ surface must have a chord of maximum possible length $4g - 5$ and exactly two chords of length 1. Moreover, punctured partitioned chord diagrams are uniquely determined by the location of this chord of maximum possible length. Thus, the number of surfaces of genus $g$ in the complement is at most the number of possible locations of the chord of maximum length multiplied by 2. In particular, we show that there are at most $8\phi(g - 1)$ punctured partitioned chord diagrams which correspond to a tubing description of a closed, connected, essential, orientable surface of genus greater than 2. For surfaces of genus 2, there are at most 12 punctured partitioned chord diagrams which correspond to such surfaces. Lastly, we prove that all of these diagrams correspond to distinct, non-isotopic surfaces using results of \cite{FLOYD1984117}, \cite{Oertel1984385} and \cite{MR224099}.

%
%
\section{Tools and Definitions Used for Counting Surfaces in Montesinos Knots}\label{proof}

We introduce the main tool, i.e.~punctured partitioned chord diagrams, used throughout the proof and prove some necessary properties about punctured partitioned chord diagrams. Namely, we establish existence first then, after a few more properties are established, show uniqueness of punctured partitioned chord diagrams for a tubing description of a closed, connected, essential, orientable essential surface. Example~\ref{Example} shows how to obtain a tubing of a pair of punctures of incompressible $S_C$'s from a punctured partitioned chord diagram.

\begin{Def} \label{DefPPCD}
	A \textit{punctured partitioned chord diagram} is a chord diagram containing 
	\begin{enumerate}
		\item chords between pairs of base points (represented by arcs) which are disjoint,
		\item four marked regions partitioning the boundary circle (and in turn the disk),
		\item each base point is contained in a unique marked region,
		\item each marked region contains the same number of base points,
		\item a puncture disjoint from the chords, in the disk that the circle bounds,
		\item and no chord is isotopic into a single marked region.
	\end{enumerate}
\end{Def}

\begin{figure}[h!]
\centering
\includegraphics[viewport = 0 150 650 650, scale = 0.3, clip]{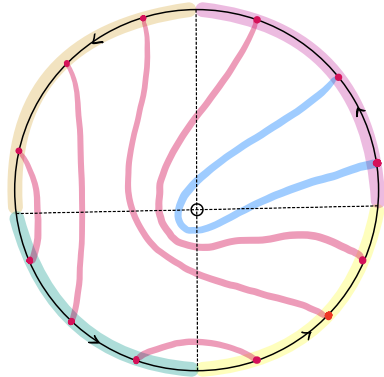}
\caption{An example of a punctured partitioned chord diagram containing a chord, shown in blue, with maximal possible length $11$ (see Definition~\ref{lengthofchord}).}
\label{PPCD}
\end{figure}

\textit{Remark}. By taking an isotopy of the puncture, arcs can be isotoped, relative to its base points, to be chords. We will be counting these diagrams up to isotopy.

\begin{Def}\label{lengthofchord}The \textit{length of a chord} with base points say $b_1$ and $b_2$ in a punctured partitioned chord diagram is defined by first isotoping the chord, keeping the chord contained in the punctured disk, onto the boundary circle of the punctured chord diagram and counting the minimum number of base points between $b_1$ and $b_2$ augmented by 1 (e.g. see blue chord in Figure~\ref{PPCD}).
\end{Def}

\textit{Remark}. The isotopy in Definition~\ref{lengthofchord} cannot pass through the puncture, hence the length of a chord is well-defined.

\begin{Def}\label{distancebetweenbasepoints}
	The \textit{distance between base points} $b_1$ and $b_2$ in a punctured partitioned chord diagram is defined to be the minimum number of base points between $b_1$ and $b_2$, exclusive.
\end{Def}

\begin{Def}\label{distancebetweenbasepoints}
	Two chords $c'$ and $c''$ are said to be \textit{parallel} if the base points of $c'$ are equidistant apart from the pair of base points of $c''$.
\end{Def}

\begin{Def}\label{dualtree}
	A \textit{dual tree} of a punctured partitioned chord diagram is the dual planar graph (i.e. a graph embedded in $\mathbb{R}^2$) defined as follows: 
	\begin{enumerate}
		\item vertices of the tree correspond to complementary regions in the chord diagram bounded by chords and arcs of the bounding circle,
		\item edges of the tree occur when the closure of complementary regions, corresponding to vertices, have nonempty intersection,
		\item the complementary region which contains the puncture will be marked as the root of the tree.
	\end{enumerate}
\end{Def}

\textit{Remark}. The dual graph is indeed a tree since each chord partitions the diagram into exactly two subdiagrams, so removing any dual edge disconnects the dual graph. Hence, the dual graph is minimally connected and so is a tree. Note that in the construction of the dual tree, no choices were made. So every chord diagram has a unique dual tree.

%
%
\subsubsection{Example of Constructing Surface from Punctured Partitioned Chord Diagram}\label{Example}
First, we discuss an important type of chord in a punctured partitioned chord diagram: the chord of maximum possible length. Suppose we have a punctured partitioned chord diagram with $g - 1$ base points in each region. Then, the maximum length of any chord is $4g - 5$. 

Observe that if there is such a chord in the diagram, it is unique since any other such chord necessarily intersects it. Now, to see that the maximum possible length of a chord is $4g - 5$, we use the Euler characteristic of surfaces and compute the number of 4-punctured spheres. Observe that each 4-punctured sphere has Euler characteristic equal to $-2$. Letting $n$ be the number of 4-punctured spheres of a surface $S$ with genus $g$, $m$ be the number of tubes $t$, and $k$ be the number of punctures, we have $$\chi(S) = 2-2g = n \cdot \chi(S_C) + m \cdot \chi(t) - k \cdot \chi(S^1) = -2n.$$ Hence, the number of 4-punctured spheres of a closed surface $S$ with genus $g$ is $g - 1$. 

We use the example in Figure~\ref{PPCD} to show how to obtain a closed essential surface from a punctured partitioned chord diagram. Note that this surface may not be unique. However, we may obtain a closed essential surface from a punctured partitioned chord diagram and a choice of incompressible 4-punctured spheres $S_C$'s. We first choose one of the two possible incompressible $S_C$'s in the knot complement, both of which have four punctures on four distinct arcs. The number of base points in a single region gives the total number of $S_C$'s.  The four arcs of the partition give the four arcs which contain the punctures of the $S_C$'s, well defined up to cyclic permutation of the regions. Chords between a pair of base points of the diagram give a tubing of the corresponding punctures of the $S_C$'s in the following way. Isotope the chord, relative to its base points, to the boundary of the disk, observing that the chord cannot be isotoped through the puncture. Apply the same process to every chord in the diagram. Observe that chords that pass from one region to another represent passing through the rational tangle(s) between the arcs corresponding to those regions. If two isotoped chords intersect, then the lengths of such intersecting chords (see Definition~\ref{lengthofchord}) indicate which tube will be innermost with respect to a meridian curve of $K$, where the longer length chord will be innermost. That is, after isotoping all chords to the boundary circle, then for pairs of chords which overlap, take the chord with largest length to correspond to the tube which is innermost between the pair of tubes corresponding to the chords. For example, the blue chord in Figure~\ref{PPCD} has maximum possible length and hence the corresponding tube indicated in blue in Figure~\ref{Example_Surface} (right), will be innermost amongst all tubes. After tubing between all punctures, following the chords of the diagram, we obtain an embedded closed essential surface in the complement of $K$.

\begin{figure}[h!]
\centering
	\begin{minipage}[b]{0.4\linewidth}
		\centering
		\includegraphics[viewport = 0 150 650 700, scale = 0.25, clip]{PPCD.pdf}
		\label{PPCD1}
	\end{minipage}
	\begin{minipage}[b]{0.4\linewidth}
		\centering
		\includegraphics[viewport = 0 150 650 650, scale = 0.25, clip]{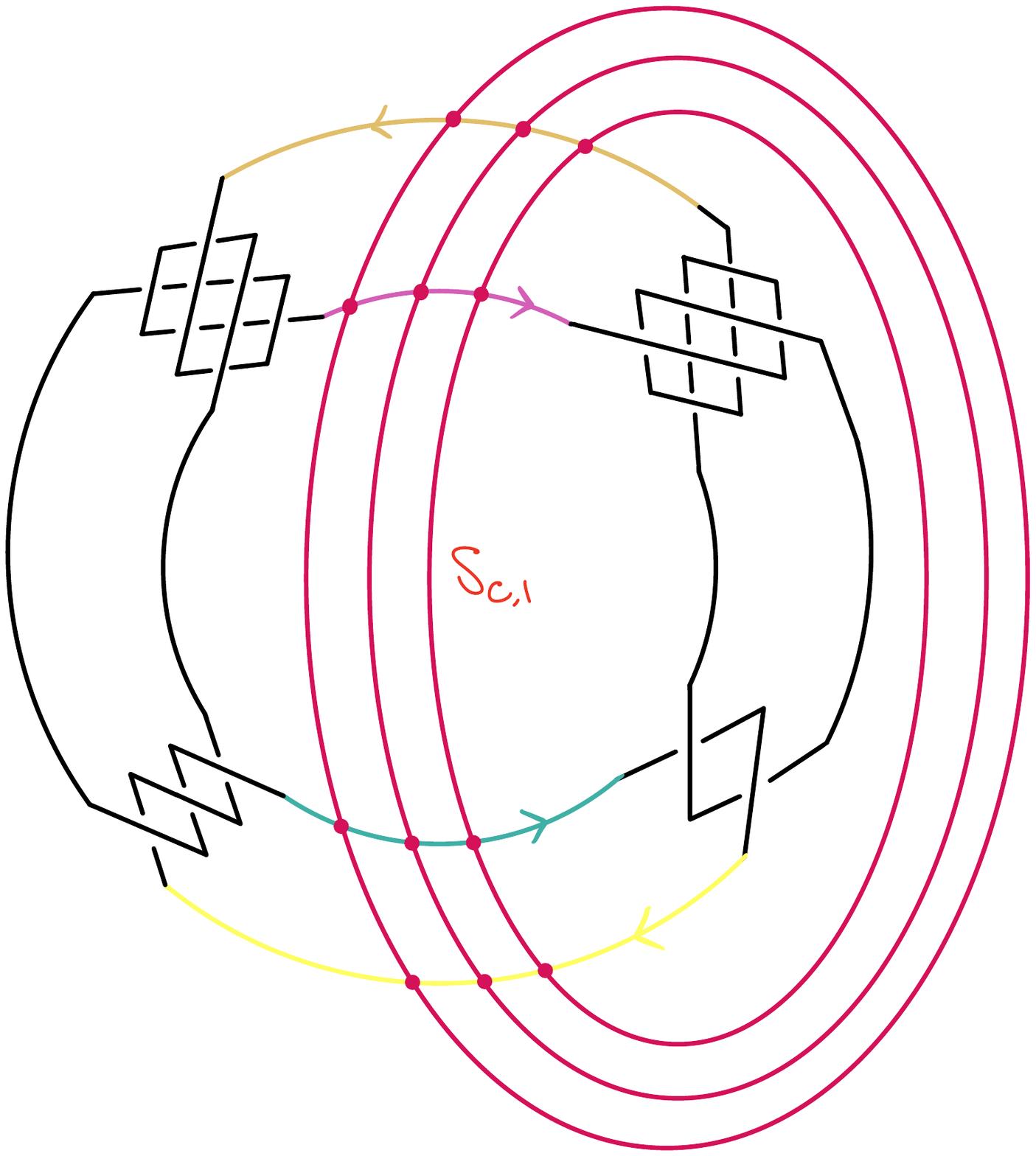}
	\end{minipage}
	\begin{minipage}[b]{0.4\linewidth}
		\centering
		\includegraphics[viewport = 0 150 650 650, scale = 0.25, clip]{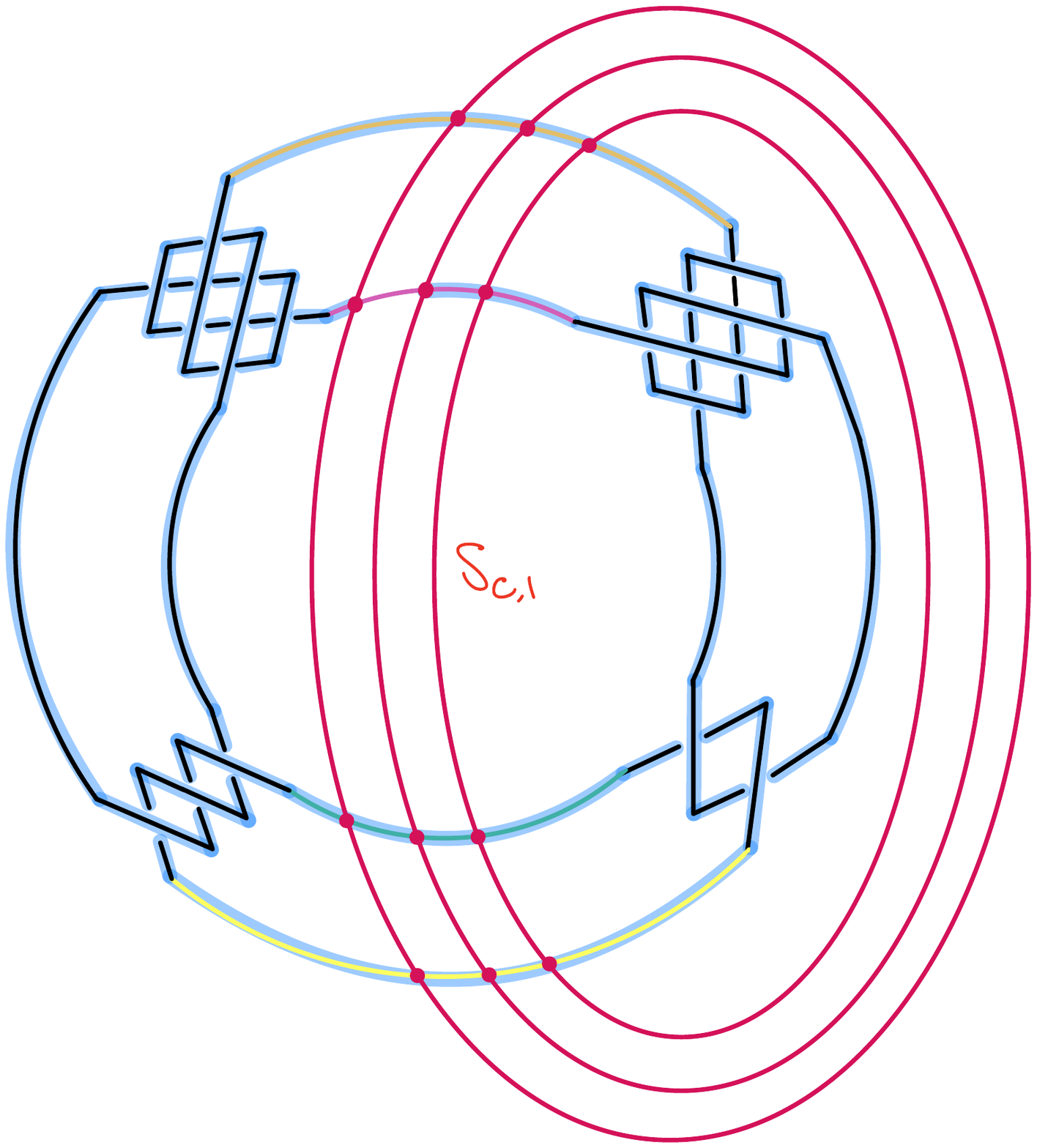}
	\end{minipage}
	\begin{minipage}[b]{0.4\linewidth}
		\centering
		\includegraphics[viewport = 0 150 650 650, scale = 0.25, clip]{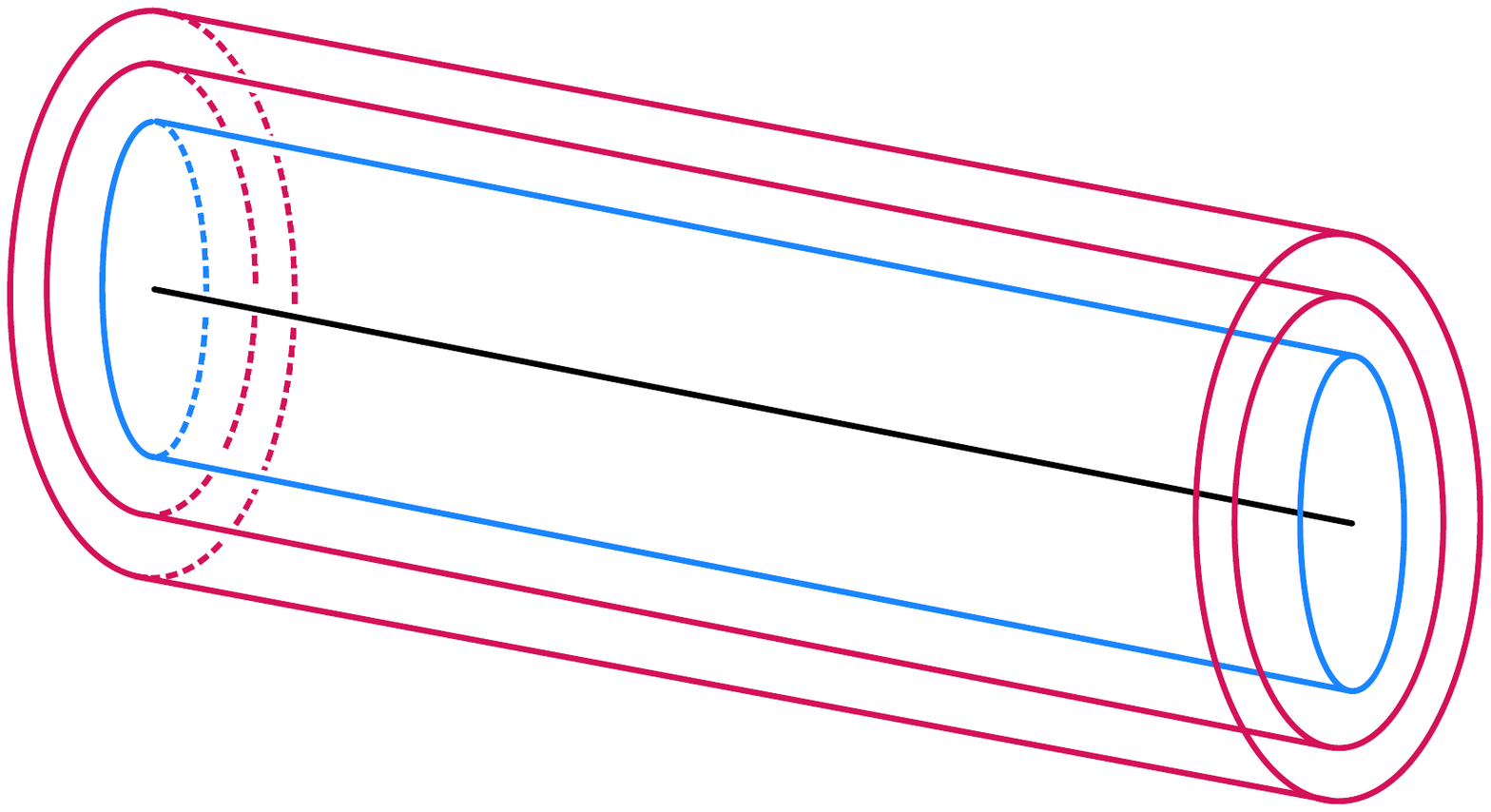}
	\end{minipage}
	\caption{The tubing corresponding to the blue chord of the punctured partitioned chord diagram in Figure~\ref{PPCD}.}
	\label{Example_Surface}
\end{figure}

%
%
\subsubsection{Geometric Interpretation of Punctured Partitioned Chord Diagrams}\label{RegNbhdPPCD}
We discuss a geometric way to interpret punctured partitioned chord diagrams. A collared neighborhood of the knot $K$ in the knot complement $M = S^3 - K$ is a tubular neighborhood of $K$ which does not include $K$. Namely, it is homeomorphic to $K \times (D^2 - \{0\})$ where $K \times \{0\}$ is sent to $K$ in $M$. This collared neighborhood is homeomorphic to a solid torus minus the core. In other words, a collared neighborhood of $K$ is homeomorphic to $S^1 \times (D^2 - \{0\})$. When we restrict to the arcs of $K$ that contain the punctures of the 4-punctured spheres $S_C$, then we are only concerned about four particular arcs. Moreover, by Theorem~\ref{Oertel2}, we only need to record when a tube passes through a rational tangle. Hence, we may omit the parts of the collared neighborhood which pass through a rational tangle and then attach the regular neighborhoods of these four arcs together following the associated rational tangles between them. The four arcs of the knots give rise to the four regions of the diagram, and regions are adjacent exactly when arcs are adjacent when traversing the knot. When a chord in the diagram passes from one region to another, this corresponds to a tube passing through at least one rational tangle. We now have a space homeomorphic to a $S^1 \times (D^2 - \{0\})$ with four marked annuli consisting of meridians. Finally, we collapse $D^2 - \{0\}$ of the collared neighborhood to $(0,1]$, since every point in $(0,1]$ will have exactly one $S^1$ fiber over it. Thus, we get an annulus partitioned into four pieces. Moreover, the inner boundary component  represents $K$ and thus becomes the puncture of the punctured partitioned chord diagram.\\

We review definitions of integer pairings as in \cite{AHW2002}, which will be used in the latter part of Section~\ref{generalcase}. We follow the descriptions and definitions cited in \cite{lee2021essential}.

\begin{Def}\label{pairing}
	Let $[1, N] \subseteq \mathbb{N}$ be the set of integers $\{1, 2, \ldots, N\}$. A bijection $g : [a , b] \rightarrow [c, d]$ defined on subsets $[a, b], [c, d] \subseteq [1, N]$ is called a \textit{pairing}.
	
	A pairing is \textit{orientation preserving} if it is increasing and \textit{orientation reversing} if it is decreasing. 
\end{Def}

\begin{Def}\label{pairing}
	Let $[1, N] \subseteq \mathbb{N}$ be the set of integers $\{1, 2, \ldots, N\}$ and suppose we have a collection of pairings $\{g_i\}, 1 \leq i \leq k$. The $g_i$'s generate a pseudogroup on $[1, N]$ and any two integers are said to be in the same \textit{orbit} if some pairing in the pseudogroup sends one to the other.
\end{Def}

%
\section{Description of Montesinos Knots in Terms of Orbifolds} \label{DescriptionOfOrbifold}

We will prove in Proposition~\ref{S_CForm} that incompressible, peripheral incompressible 4-punctured spheres $S_C$'s have the form as in Figure~\ref{S_Cs} and that those are the only such $S_C$'s in the complement of $K$ in $S^3$. First, we recall a description, given in \cite{oertel1984}, of a fibered orbifold $O_K$ with underlying space $S^3$ and $\Z_2$ singular set $K = K(p_1/q_1,\ldots,p_k/q_k)$. For a rational tangle of slope $p/q$, we describe a different construction in terms of orbifold fibrations. Namely, let $(T, p/q)$ be a solid Seifert fiber torus with exceptional fiber of order $(q,-p)$. Intuitively, this means that the fibers in a neighborhood of the exceptional fiber twist about, i.e. corkscrew around, the exceptional fiber $p$ times and traverse along the exceptional fiber $q$ times. In order to uniquely determine $p$, we must fix a fundamental domain for $\partial T$. To do this, we fix two simple closed curves $x$ and $y$ on $\partial T$: take $x$ to intersect each fiber transversely exactly once and $y$ a fiber of $T$. Thus, lifting the boundary $\partial D$ of a meridian disk $D$ of $T$ to the universal cover of $\partial T$, we get a line of slope $p/q$ relative to the axis given by the lift of $x$ and the axis given by the lift of $y$. We may obtain a rational tangle $(L, p/q)$ from this construction by taking the quotient of a $180^{\circ}$ rotation of $T$ about a diameter of the exceptional fiber of $(T, p/q)$, where the rotation is fiber preserving (see Figure~\ref{OrbiFibration}, right). Label the arcs of intersection of the axis of rotation and $T$ as $\kappa_0$ and $\kappa_1$. After a homeomorphism of $(L, p/q)$, we may obtain the space shown on the right in Figure~\ref{OrbiFibration} where $\kappa_0$ and $\kappa_1$ will be the arcs of the rational tangle. The space $(L, p/q)$ is the 3-ball with marked arcs $\kappa_0, \kappa_1$ and a coordinate system consisting of an $x$-~and $y$-axis. The fibering of $(T, p/q)$ descends to a fibration of $(L,p/q)$ in the following way. Generic fibers, disjoint from the axis of rotation, of $(T, p/q)$ become parallel copies of the $y$-axis. The generic fibers of $(T, p/q)$ that intersect the axis of rotation descend to exceptional fibers as arcs with one endpoint on $\kappa_0$ and the other endpoint on $\kappa_1$. For example, the generic fiber on $\partial T$ shown in Figure~\ref{OrbiFibration}, left, descends to a parallel copy of the $y$-axis in $\partial L$ and the two generic fibers on $\partial T$ that intersect the axis of rotation descends to the two exceptional fibers (arcs) on $\partial L$ having one endpoint on $\kappa_0$ and the other on $\kappa_1$ (see Figure~\ref{OrbiFibration}, upper right).

To construct the 3-orbifold $O_K$ of $K$, deform every rational tangle $(L_i, p_i/q_i)$ in such a way that it is lens-shaped with its axis at the edge, as shown in upper figure of Figure~\ref{2Orbifold}. Then, identify the left face of $L_i$ with the right face of $L_{i + 1}$ (subscripts mod $k$) such that fibers, punctures, and axes are identified and half of one $x$-axis is identified with half of the next $x$-axis (see Figure~\ref{2Orbifold}, upper). The 3-orbifold $O_K$ is fibered over the 2-orbifold as in Figure~\ref{2Orbifold}, lower.

\begin{figure}[h!]
\centering
	\begin{minipage}[b]{0.3\linewidth}
		\centering
		\includegraphics[viewport = 30 170 650 725, scale = 0.25, clip]{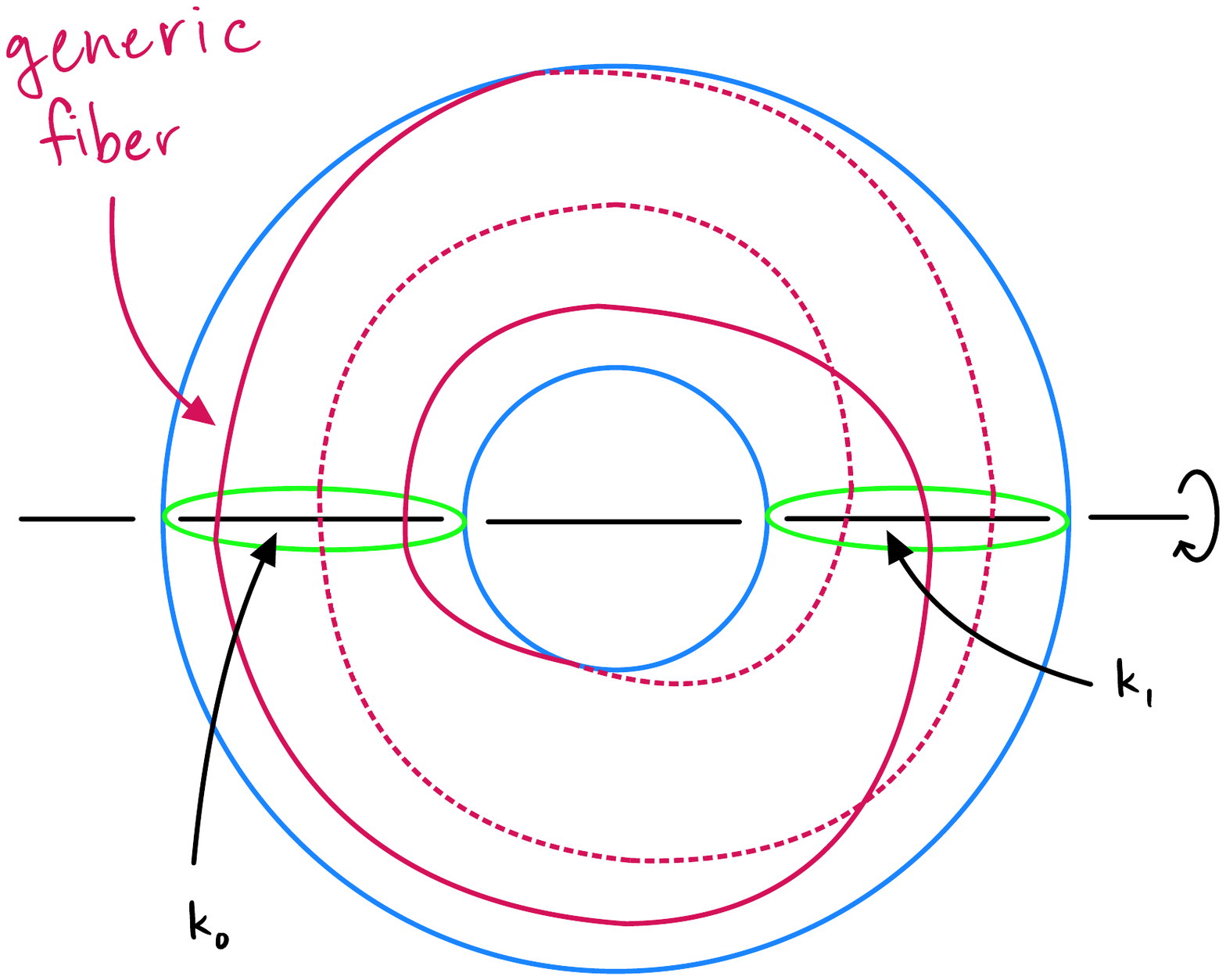}
	\end{minipage}
	\begin{minipage}[b]{0.375\linewidth}
		\centering
		\includegraphics[viewport = 0 150 650 650, scale = 0.25, clip]{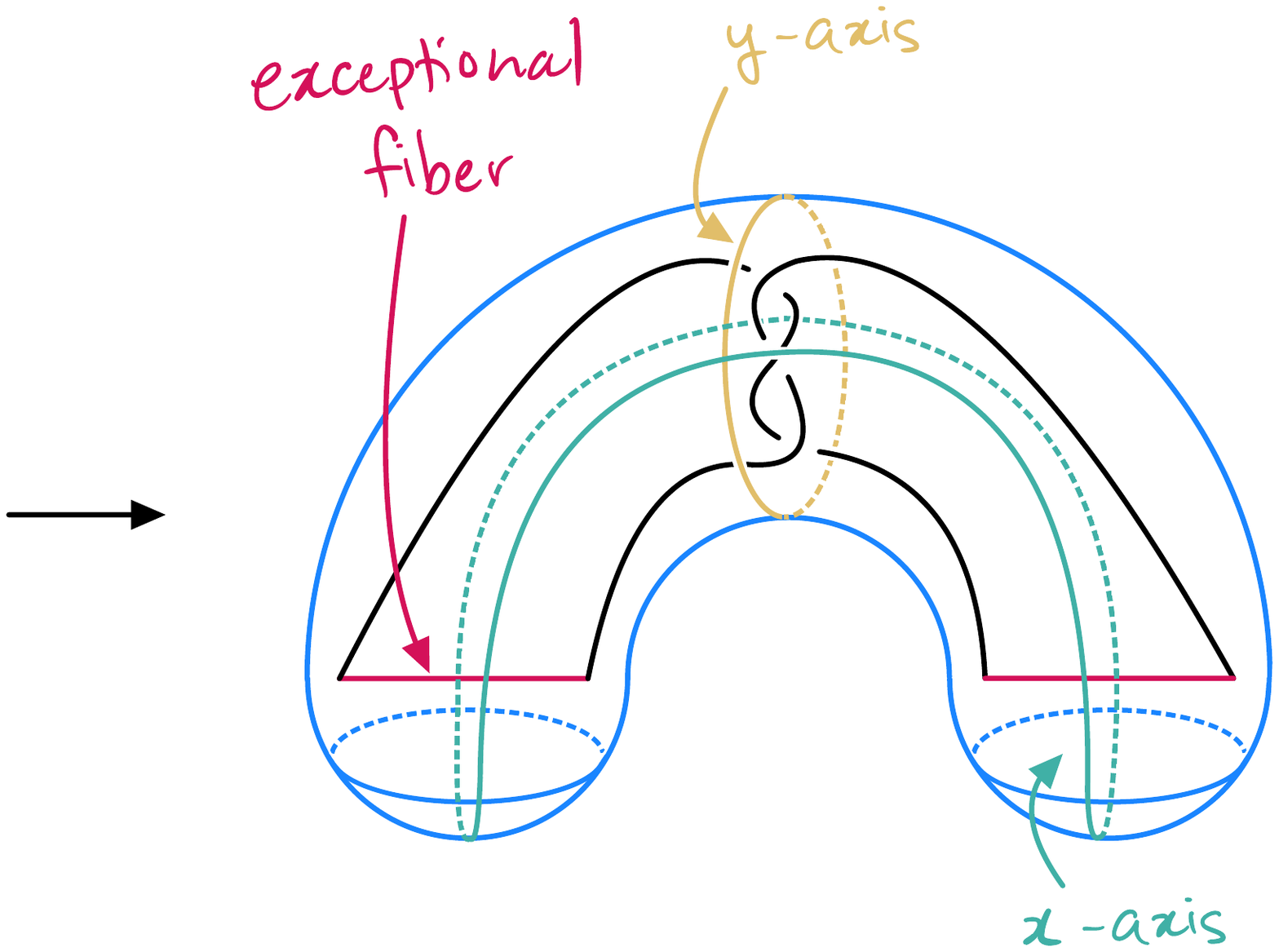}
	\end{minipage}
	\caption{Left: $(T, p/q)$ for $p/q = 1/3$, Right: $(L, p/q)$.}
	\label{OrbiFibration}
\end{figure}

\begin{figure}[h!]
\centering
	\begin{minipage}[b]{0.45\linewidth}
		\centering
		\includegraphics[viewport = 0 90 650 725, scale = 0.25, clip]{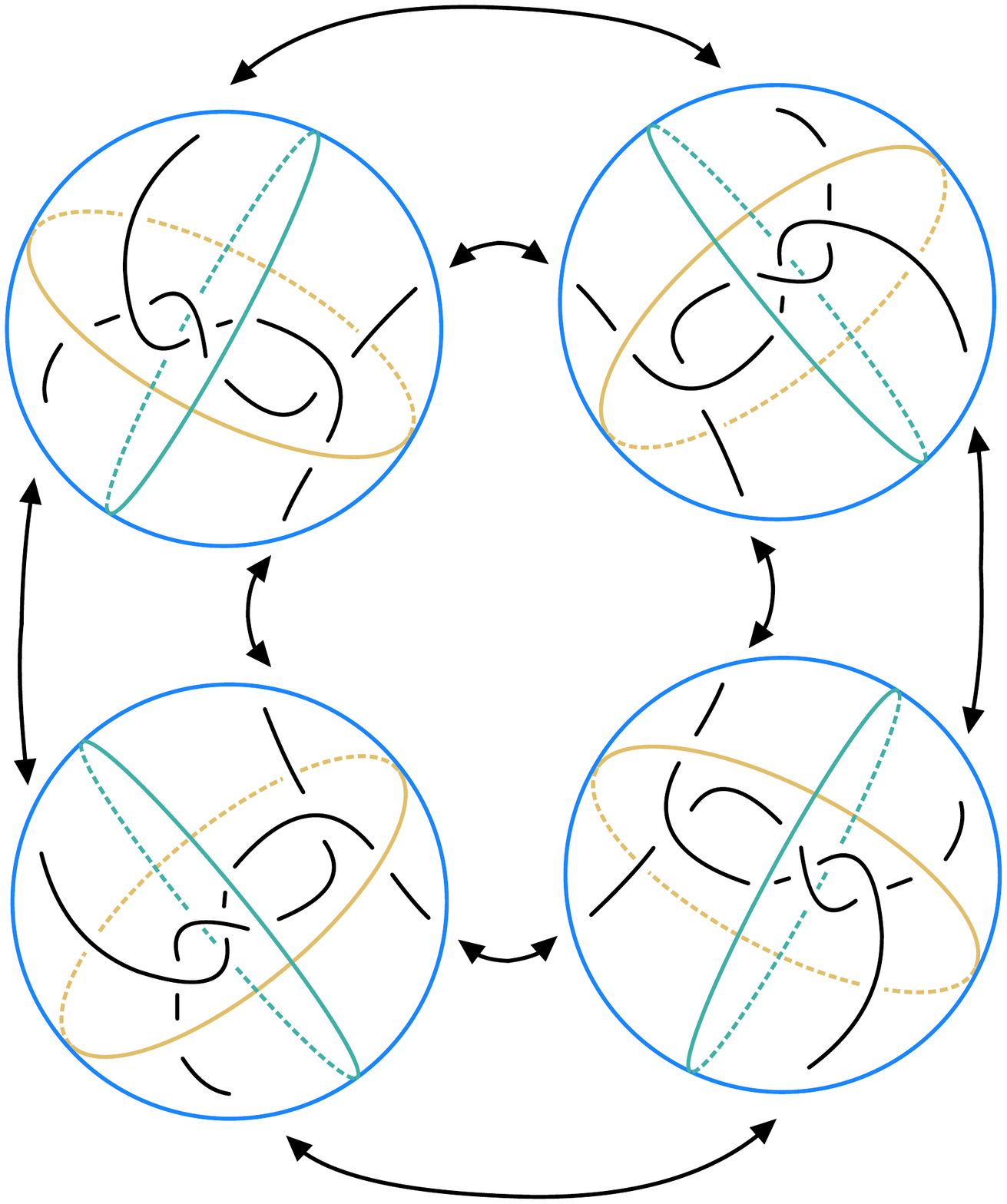}
	\end{minipage}
	\begin{minipage}[b]{0.9\linewidth}
		\centering
		\includegraphics[viewport = 0 140 650 725, scale = 0.25, clip]{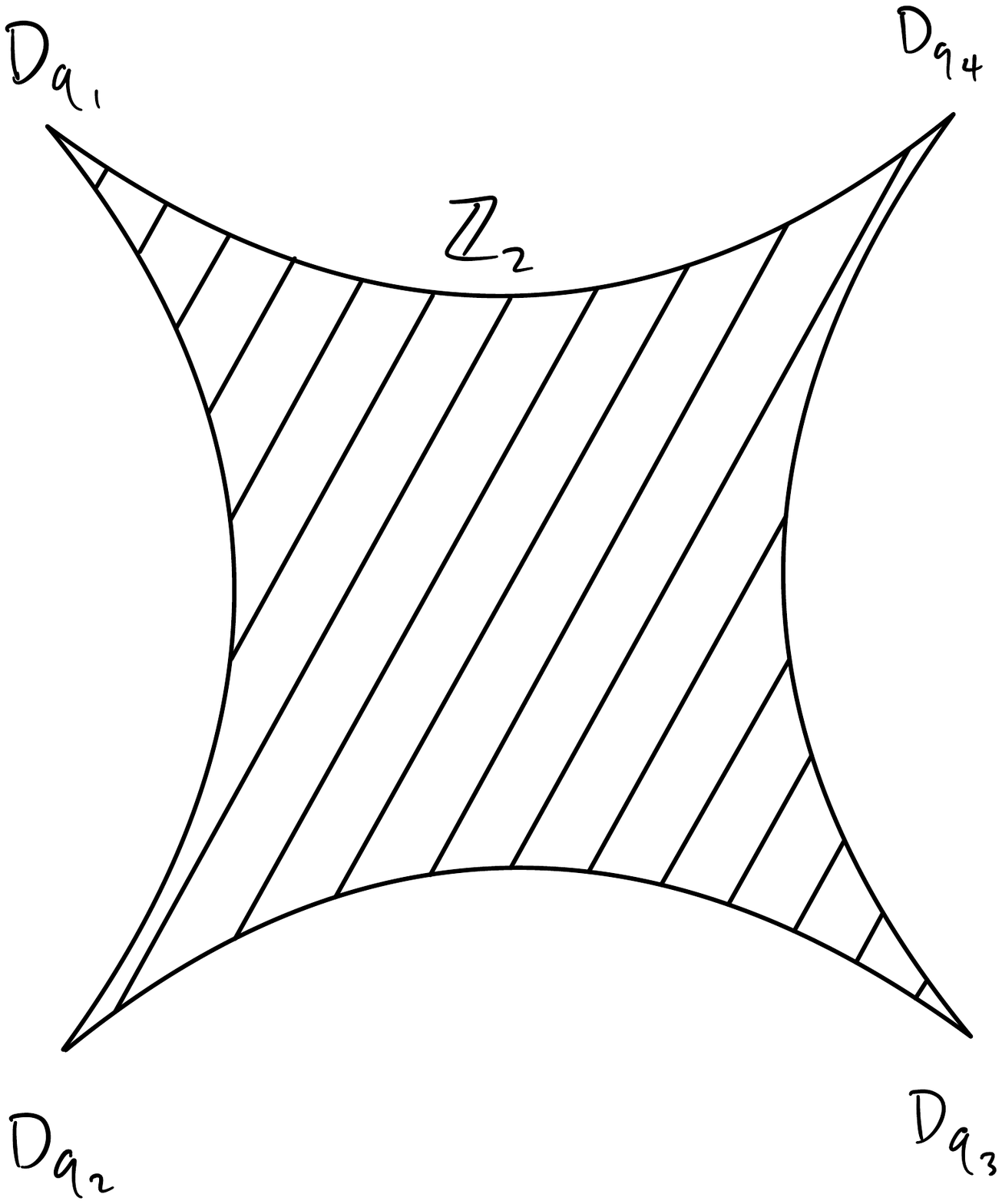}
	\end{minipage}
	\caption{Upper: Construction of the orbifold $O_K$ for $K=K(p_1/q_1,\ldots,p_4/q_4)$, Lower: An example of the base 2-orbifold of the 3-orbifold $O_K$.}
	\label{2Orbifold}
\end{figure}

\begin{Prop}\label{S_CForm}
	For an incompressible 4-punctured sphere, $S_C$, in a Montesinos knot complement, the associated closed curve in the knot diagram that represents $S_C$ intersects four distinct arcs between rational tangles. Moreover, if two $S_C$'s, belonging to the collection via a surface $S$ as in Theorem~\ref{Oertel1}, bound the same rational tangles, then they must be parallel, and thus, isotopic.
\end{Prop}
	
Proposition~\ref{S_CForm} will follow immediately from the next two lemmas.

	\begin{Lemma}\label{S_CForm1}
		Incompressible, peripheral incompressible 4-punctured spheres $S_C$ are represented by essential arcs in the base 2-orbifold of the orbifold fibration $O_K$.
	\end{Lemma}
	\begin{proof}
		By Corollary 2.8 of \cite{oertel1984}, an incompressible, peripheral incompressible punctured surface in $(S^3, K)$ is either vertical (union of fibers) or horizontal (transverse to fibers) in the orbifold fibration of $O_K$. By Lemma 2.9 of \cite{oertel1984}, the $S_C$'s in the collection of Theorem~\ref{Oertel1} are all vertical in the orbifold fibration, by the fact that $\sum_{i=1}^k p_i/q_i \neq 0$ implies there exists no horizontal incompressible, peripheral incompressible punctured surfaces. In the orbifold fibration $O_K$, which has base 2-orbifold as in Figure~\ref{2Orbifold}, vertical 4-punctured spheres $S_C$'s are represented by arcs in the base orbifold. We show, using Corollary~\ref{Oertel3}, that for the arcs to represent incompressible, peripheral incompressible $S_C$'s, then the arcs must be essential, i.e. not homotopic into a boundary component of the base 2-orbifold, relative to its base points. Otherwise, for the complement of an inessential arc in the base 2-orbifold, there are two regions. Either one region contains exactly one rational tangle, marked as a singular point of order $D_{q_i}$, the dihedral group of order $2q_i$, or it contains no rational tangles (see e.g. Figure~\ref{Essential_Curves}). This implies that the corresponding $S_C$ bounds either one rational tangle on one side or no rational tangles. Contradiction to Corollary~\ref{Oertel3}, since incompressible $S_C$'s must bound at least two rational tangles on either side.
	\end{proof}
	\begin{Lemma}\label{S_CForm2}
		If two $S_C$'s, coming from the collection of 4-punctured spheres of a surface $S$ as in Theorem~\ref{Oertel1}, bound the same rational tangles, then they must be parallel, and thus, isotopic.
	\end{Lemma}
	\begin{proof}
	Now, from Theorem~\ref{Oertel1} the collection of $S_C$'s are disjoint, so the $S_C$'s are represented by disjoint arcs in the base 2-orbifold. This implies that the collection of essential arcs corresponding to the $S_C$, coming from Lemma~\ref{S_CForm1}, all have base points on the two distinct, non-adjacent boundary components of order $\mathbb{Z}_2$. Moreover, since the $S_C$'s bound the same rational tangles, then the components of the complement of the arcs contain the same singular, not of order $\mathbb{Z}_2$, points of the base orbifold. We may isotope these arcs to each other in the base orbifold, which gives an isotopy, via the fibers, of the corresponding $S_C$'s in the orbifold fibration. 
	\end{proof}

\begin{figure}[h!]
		\centering
		\includegraphics[viewport = 0 150 650 650, scale = 0.3, clip]{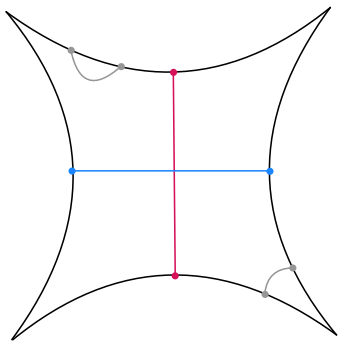}
		\caption{Example of essential arcs in the base 2-orbifold, shown in magenta and blue, and inessential arcs, shown in gray.}
		\label{Essential_Curves}
\end{figure}

\begin{Prop}\label{NonisotopicSCs}
	Let $k > 3$ be the number of rational tangles of a Montesions knot $K$. Then there are $\frac{k(k - 3)}{2}$ non-isotopic incompressible 4-punctured spheres $S_C$'s in the complement of $K$.
\end{Prop}
	\begin{proof}
		By Corollary~\ref{Oertel3}, the closed curve on the projection sphere which represents the incompressible 4-punctured sphere $S_C$ must bound at least two rational tangles on both sides. We count the number of choices of first and last rational tangles for which the closed curve on the projection sphere bounds. There are $k$ choices for the first rational tangle. After making this choice, label it as the $1^{\text{st}}$ rational tangle and order the rest of the rational tangles in a cyclic fashion as $\{1, 2, \ldots, k\}$. The second choice can be made following this cyclic labelling where the closed curve bounds all the rational tangles from the $1^{\text{st}}$ till the second choice. However, the second choice has only $k - 3$ options. This is because we cannot choose the same rational tangle twice, namely choosing the $1^{\text{st}}$ rational tangle, since this curve will not bound two rational tangles. Moreover, the curve cannot bound all the rational tangles, which is equivalent to the choice of the last rational tangle to be $k$ because all the rational tangles are contained on one side of the closed curve. Lastly, the curve cannot bound all except one rational tangle, which is equivalent to the choice of the $(k - 1)^{\text{th}}$ rational tangle, since one side of the curve has only one rational tangle. Thus, we get $k(k - 3)$ choices for closed curves representing incompressible $S_C$'s. Since we are projecting onto the sphere, these choices bound rational tangles on both sides and hence up to isotopy, there are at most $k(k - 3)/2$ incompressible $S_C$'s. 
		
		These $k(k - 3)/2$ incompressible $S_C$'s are non-isotopic since any isotopy between the $S_C$'s must be fiber-preserving by \cite{Tollefson1978INVOLUTIONS}. This is impossible, since the isotopy must pass through the fibers of $O_K$. Thus, there are exactly $k(k - 3)/2$ non-isotopic incompressible $S_C$'s.
	\end{proof}

We establish a few lemmas about punctured partitioned chord diagrams. First, using the above proposition, the following lemma shows the existence of punctured partitioned chord diagrams for closed, connected, orientable essential surfaces in Montesinos knot complements.

\begin{Lemma} \label{ExistencePPCD} 
	Every embedded closed essential surface of genus $g$ in a Montesinos knot complement has a corresponding punctured partitioned chord diagram which has $g - 1$ base points in each of the four regions. 
\end{Lemma}
	\begin{proof}
	By Theorem~\ref{Oertel1}, every closed essential surface $S$ is isotopic to a surface obtained from a tubing of incompressible $S_C$'s. In other words, every closed essential surface can be thought of as a disjoint collection of 4-punctured spheres $S_C$'s and tubes between pairs of punctures of the $S_C$'s. To obtain a chord diagram from this decomposition of $S$, we take the punctures of the $S_C$'s to be the base points of the chord diagram and the tubes between punctures to be chords between base points. We first show that this chord diagram is a punctured partitioned chord diagram. By Proposition~\ref{S_CForm}, the collection of $S_C$'s are disjoint and parallel, so we get an ordering via a nesting of the $S_C$'s in the knot projection (see e.g. Figure~\ref{Example_Surface}, right). Moreover, all punctures of the $S_C$'s lie on four distinct arcs and by taking an orientation of the knot we get the four regions of the chord diagram based on this orientation. This gives that each region contains exactly $g - 1$ base points, the number of distinct $S_C$'s. Observe that regions are well defined up to cyclic ordering. Now, if we label the punctures of the $S_C$ as $S_{C,j}^i$ where $i$ represents the $i^{\text{th}}$ $S_C$ with $i=1$ is the innermost $S_C$ with respect to the projection of the knot and $j$ represents the arc that the puncture lies on. Since each tube does not contribute to the Euler characteristic of a surface, then the Euler characteristic of $S$ is determined by the number of $S_C$'s. Hence, all chord diagrams of $S$ must have the same number of base points since each $S_C$ has exactly four punctures. Tubing between pairs of punctures correspond to chords between the corresponding base points in the chord diagram. Since there are only two choices for tubing any pair of punctures, then the chord diagram must contain a puncture to record these two choices. The following lemma shows that the chords do not intersect and cannot be isotopic into a single region, which completes the proof.

	\begin{Lemma} \label{ExistencePPCD2} 
	Every chord in the punctured partitioned chord diagram of Lemma~\ref{ExistencePPCD} does not pairwise intersect with any other chord in the diagram and cannot be isotopic into a single region.
	\end{Lemma}
\textit{Proof}. Suppose for contradiction that two chords of a chord diagram of $S$ intersect. We may assume that the intersection is transverse, after an isotopy of the chords reducing the number of points of intersection. Using the geometric interpretation of a punctured partitioned chord diagram, observe that a point of intersection of two chords represents a meridian of the knot. This implies $S$ has self intersection which cannot be resolved, a contradiction to the surface being embedded. Hence, a corresponding chord diagram for $S$ does not have intersecting chords.
	
	Now, suppose there is a chord isotopic into a single region of the punctured partitioned chord diagram. This represents a tube that is completely contained on a single arc of the knot $K$. In particular, this tube does not pass through a rational tangle. Contradiction, since this would mean that the surface is compressible by Theorem~\ref{Oertel2}.
	\end{proof}

%
%
%
\section{Surface Counts of Genus 2 Surfaces in Montesinos Knot Complements}\label{g2case}

Throughout this section, we assume that every surface $S$ will have genus $g = 2$. We prove there are at most 12 punctured partitioned chord diagrams for this case. 
\begin{Prop}\label{2TreeLeaves}
	Every closed, connected, essential, orientable essential surface $S$ in a Montesinos knot complement has a corresponding punctured partitioned chord diagram with exactly 4 base points, 1 in each region, which has a dual tree with exactly 2 leaves.
\end{Prop}
	\begin{proof}
		By Lemma~\ref{ExistencePPCD}, $S$ has a punctured partitioned chord diagram. Moreover, by the computation in the remark after Definition~\ref{lengthofchord}, there is only one 4-punctured sphere $S_C$ in the collection in Theorem~\ref{Oertel1} which $S$ is isotopic to a tubing of. Thus, the punctured partitioned chord diagram has exactly one base point in each of its four regions, by Proposition~\ref{S_CForm}. Two tubes are required to tube the four punctures of the $S_C$. There are only 6 possible punctured partitioned chord diagrams with exactly four base points and two chords, up to isotopy of the chords and puncture (see Figure~\ref{PPCD2Leaves}). Each of the 6 chord diagrams have dual trees with exactly 2 leaves, completing the proof.
	\end{proof}

\begin{figure}[h!]
		\centering
		\includegraphics[viewport = 0 380 850 720, scale = 0.7, clip]{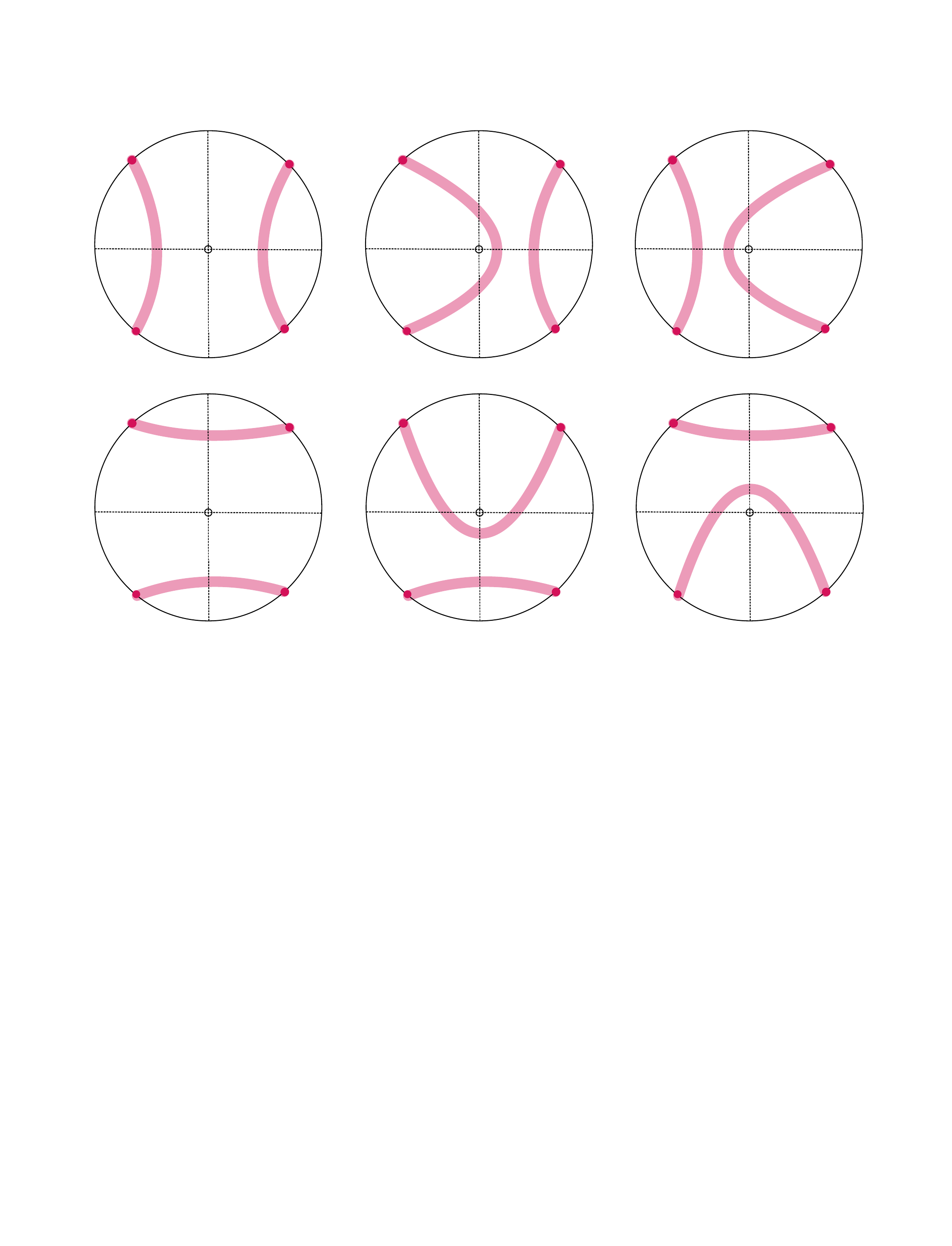}
		\caption{The 6 possible punctured partitioned chord diagrams with exactly four base points and two chords.}
		\label{PPCD2Leaves}
\end{figure}

\begin{Prop}[Genus 2 Case] \label{G2Case}
	The number of closed incompressible surfaces of genus $g = 2$ in the complement of a hyperbolic Montesinos knot $K = K(p_1/q_1,\ldots,p_k/q_k)$ with each $q_i\geq 3$ and $k=4$, is at most $12$.
\end{Prop}
	\begin{proof}
		By Proposition~\ref{2TreeLeaves}, a genus $g = 2$ surface $S$ must have a punctured partitioned chord diagram with a total of four base points, one in each region. Also from Proposition~\ref{2TreeLeaves}, the dual tree of this chord diagram must have exactly 2 leaves. Note that if the punctured partitioned chord diagram has a chord of length 3, then there is no choice for the last chord. However, there are punctured partitioned chord diagrams which do not have any chords of length 3. Now, since each surface has a punctured partitioned chord diagram, we need only to count the number of ways to construct punctured partitioned chord diagrams with four base points and two chords such that the dual tree has exactly 2 leaves. Figure~\ref{PPCD2Leaves} shows the only such possible punctured partitioned chord diagrams. For each of the 6 possibilities, there are two choices for the incompressible 4-punctured sphere to tube (see Figure~\ref{S_Cs}). Thus, the number of closed essential surfaces of genus $g = 2$ in a Montesinos knot complement is bounded by $6(2) = 12.$
	\end{proof}

%
%
%
\section{Surface Counts of Surfaces with Genus Greater than 2 in Montesinos Knot Complements}\label{generalcase}

We will show that counting surfaces via all possible punctured partitioned chord diagrams gives the upper bound in Theorem~\ref{main}. In this section, we assume all surfaces are of genus $g$ greater than 2.

%
\begin{Prop}\label{TreeLeaves}
	A tubing description of 4-punctured $S_C$'s, corresponding to a closed essential surface not containing a genus 2 component in a Montesinos knot complement, has a corresponding punctured partitioned chord diagram which has a dual tree containing exactly 3 leaves. In other words, the punctured partitioned chord diagram contains exactly 3 chords with the property that the two base points of each chord is adjacent to each other.
	
	Moreover, any punctured partitioned chord diagram that either
	\begin{enumerate}
		\item has exactly 2 chords of length 1 and no chord of maximum possible length $4g - 5$,
		\item or has 3 or more chords of length 1
	\end{enumerate}
	corresponds to a disconnected surface containing a genus 2 component.
\end{Prop}

We establish the following lemma in order to prove the above proposition.
\begin{Lemma}\label{1MaxChord}
	Suppose that a punctured partitioned chord diagram comes from a tubing description of a closed essential surface. Then the leaves of the dual tree correspond to chords of length 1 or maximum possible length $4g-5$, where $g$ is the genus of the surface.
\end{Lemma}
	\begin{proof}
		Leaves correspond to complementary regions which are adjacent to only one other unique region. In particular, the boundary of such a region, call it $R$, consists of a single chord $C$ and a single segment of the boundary circle that does not contain any base points. If the puncture is not in $R$, then $C$ has length 1. Otherwise, $C$ has maximum possible length $4g - 5$.
	\end{proof}

%
%
We now prove Proposition~\ref{TreeLeaves}.
\begin{proof}[Proof of Proposition~\ref{TreeLeaves}]
	We exhaust the possibilities for the number of leaves in a dual tree. The second part of the statement is proved in the cases concerning dual trees with more than 3 leaves. Namely, the surface must be disconnected or compressible in every punctured partitioned chord diagram having a dual tree containing more than 3 leaves. Observe that there cannot exist a dual tree containing more than 5 leaves. This is because there are only 4 allowable locations for chords of length 1, i.e. between adjacent regions, and only one region containing the puncture which corresponds to the chord of maximal possible length can exist in any diagram. This implies any other leaf must correspond to a chord of length 1 completely contained in a single region, a contradiction to Theorem~\ref{Oertel2} since the corresponding surface will be compressible. 
	
	\textit{4 and 5 leaves}: Suppose we have a chord diagram whose unique dual tree contains 4 or 5 leaves. In the chord diagram, we show that two leaves will correspond to tubing 4 punctures belonging to the same $S_C$ and hence $S$ is contains a genus 2 component. Since we are only concerned about tubing the same $S_C$, we will label the punctures of the disjoint collection $S_C$'s and show two tubes are between four punctures of the same $S_C$. Fixing  a projection of $K$ onto $S^2$, label the 4-punctured spheres and punctures as $S_{C,j}^i$ where $i$ represents the $i^{\text{th}}$ innermost $S_C$ with respect to the projection of the knot and $j$ represents the $j^{\text{th}}$ arc that the punctures lie on. Note that $j$ is well-defined, up to cyclic ordering. Fix an orientation of the knot. Following this orientation gives an ordering, that is, a permutation, of the sequence $\{S_{C,1}^1,S_{C,2}^1,\ldots, S_{C,3}^{g - 1},S_{C,4}^{g - 1}\}$. By Proposition~\ref{S_CForm}, the $S_C$'s are nested and all punctures contained on a single arc must belong to distinct $S_C$'s. Thus, the ordering must be of the form $\{S_{C,1}^1, S_{C,1}^2,\ldots, S_{C,1}^{g - 1},S_{C,2}^{g - 1}, S_{C,2}^{g - 2},\ldots,S_{C,4}^2,S_{C,4}^1\}$, up to cyclic ordering. Note that the ordering of the sequence gives a labeling of the base points of the chord diagram. By Lemma~\ref{1MaxChord}, the leaves correspond to either length 1 or maximum length chords. There can only be one chord of maximal possible length, otherwise two such chords would intersect in the chord diagram. We look at the leaves which correspond to length 1 chords and show two of these chords correspond to a tubing of one distinct $S_C$. By Theorem~\ref{Oertel2}, each tube must pass through at least one rational tangle, so there are exactly 4 possibilities for the location of the chords of length 1: between base points of adjacent regions. Base points of adjacent regions, call them $b_1$ and $b_2$, must belong to the same $S_C$. This is because there must be at least one rational tangle between the two regions, which the base points belong to, and this rational tangle must be either inside or outside of the $S_C$ of one of base points. However, by Theorem~\ref{Oertel1}, then the base points must belong to the same $S_C$ otherwise the collection of $S_C$'s will not be disjoint if the base points belong to different $S_C$'s. So, without loss of generality, the base points of the chords of length 1 are $S_{C,1}^{g - 1},S_{C,2}^{g - 1},S_{C,3}^{g - 1},S_{C,4}^{g - 1}$, by the ordering of the base points. This is a contradiction, since this will be a tubing of $S_C^{g - 1}$ which gives a genus 2 component of the surface.
	
	\textit{1 and 2 leaves}: Immediately, we see any nonempty tree must contain at least 2 leaves. Suppose that the dual tree of a punctured partitioned chord diagram contains exactly 2 leaves. The complement of the base points of the 2 leaves partitions the boundary circle into four segments: two segments containing base points and two segments without base points. We call the boundary circle of the chord diagram to have two sides based on the two segments which contains base points. Now, there are two possibilities: either both sides of the leaves contain the same number of base points or one side of the leaves contains a different number of base points than the other side. For the latter case, observe that in the chord diagram, we have that there will be base points on one side of the leaves which cannot have a chord to a base point on the other side of the leaves. Hence, there must exist a chord between adjacent base points, otherwise if all remaining base points are not adjacent, then in order to chord the remaining base points, it must intersect the chords whose base points lie on different sides of the leaves, which gives a contradiction to chords being disjoint. Since there exists a chord between adjacent base points, then this implies that the tree has at least 3 leaves which is impossible by assumption. For the case where both sides of the leaves contain the same number of base points, notice that the two leaves cannot both correspond to chords between base points contained in the same region. This is due to the fact, by disjointness of the chords, only one of these can have maximal length. This implies that the other chord, which has length 1, represents a tubing of base points which does not pass through a rational tangle region, contradicting Theorem~\ref{Oertel2}. Now, if the two leaves correspond to chords between base points of adjacent regions, then this contradicts the surface not having a genus 2 component.
\end{proof}

We bound the number of punctured partitioned chord diagrams in terms of the genus of the surface. However, we first establish the following proposition for the number of surfaces corresponding to a punctured partitioned chord diagram:

%
\begin{Prop}\label{MaxPPCDs}
	If $S$ and $S'$ are closed essential surfaces where $S$ is a tubing of $S_C$'s and $S'$ is a tubing of $S_{C'}$'s with $C \neq C'$, then $S$ and $S'$ are not isotopic.

	In particular, every punctured partitioned chord diagram containing four regions, each having $g - 1$ base points, corresponds to exactly 2 non-isotopic essential surfaces.
\end{Prop}

We first prove the following lemma.

\begin{Lemma}\label{DecomposeByAnnuli}
	Let $S$ be a closed essential surface in $M = S^3 \backslash \mathring{N}(K)$ obtained from tubing $S_C$'s. Let $\widetilde{M} = M\backslash\backslash S$ be the complement of a regular neighborhood of $S$ in $M$. Suppose that $A, B$ are embedded essential annuli in $\widetilde{M}$ each having exactly one component on $\partial M$ which is a meridian of $K$. Then $A, B$ are isotopic to disjoint annuli.
\end{Lemma}
	\begin{proof}
		We may assume that $A$ and $B$ are in general position. Let $\partial_0 A$ and $\partial_0 B$ be one boundary component of $A$ and $B$, respectively, which is a meridian of $K$. By further isotoping $\partial_0 A$ and $\partial_0 B$, we may assume that $\partial_0 A \cap \partial_0 B$ is empty. If $A \cap B$ is empty, then we are done. Otherwise, note that we may simplify $A \cap B$ to consists of essential arcs and essential simple closed curves. Since $\partial_0 A \cap \partial_0 B$ is empty, then no essential arc of intersection has a base point on $\partial_0 A$ or $\partial_0 B$. Observe that we may resolve the intersections corresponding to simple closed curves by an innermost annulus argument. Since the simple closed curves of intersection in $A$ and $B$ are parallel to $\partial_0 A$ or $\partial_0 B$, we may find an innermost annulus $\tilde{A}$ in $A$, with respect to $\partial_0 A$ and $\partial_1 A$, bounded by two simple closed curves, say $\gamma$ and $\gamma'$. Observe that $\gamma$ and $\gamma'$ also bound an annulus $\tilde{B}$ in $B$, otherwise there would be another simple closed curve of intersection of $A$ and $B$ in between $\gamma$ and $\gamma'$, contradicting that $\gamma$ and $\gamma'$ were innermost in $A$. Thus, we resolve $\gamma$ and $\gamma'$ by an isotopy of $B$ pushing $\tilde{B}$ past $\tilde{A}$. Continuing in this fashion, we may resolve all simple closed curves of intersection, since there are only finitely many due to $A$ and $B$ being embedded and in general position. Note that if necessary, we may resolve the final simple closed curve of intersection by interchanging $\partial_0 A$ and $\partial_0 B$ on $\partial \widetilde{M}$ via an isotopy.
		
		To resolve the arcs of intersection, we use an innermost disk argument. Without loss of generality, take an arc of intersection $\alpha$ which is innermost on $A$ bounding a disk $D_1$ in $A$. Since $A$ and $B$ are essential surfaces, then there exists a disk $D_2$ in $B$ bounded by $\alpha$. Note that $\alpha$ need not be innermost on $B$. The boundary of the disk $D = D_1 \cup D_2$ must also bound a disk $D'$ in $\partial \widetilde{M}$. Thus, since $\widetilde{M}$ is irreducible, then we may resolve $\alpha$ by an isotopy of $B$ that pushes $D_2$ past $D_1$. Since both $A$ and $B$ are embedded and in general position, then there are only finitely many arcs of intersection of $A$ and $B$. Thus, continue this process of taking innermost disks until all such intersections are resolved, completing the proof.
	\end{proof}

\begin{proof}[Proof of Proposition~\ref{MaxPPCDs}]
	We know that every embedded closed essential surface $S$ can be decomposed, via tubes, into a collection of incompressible 4-punctured spheres $S_C$'s. We show this collection of $S_C$'s consists of parallel copies of one of the two non-isotopic $S_C$'s. Suppose that $\alpha_1, \ldots, \alpha_k$ are disjoint essential simple closed curves on $S$ which decomposes $S$ into incompressible 4-punctured spheres and annuli. Let $S_0 = S$ and $A_0$ an embedded essential annulus in $M \backslash\backslash S_0$ with $\partial_0 A_0 = [\mu]$. Then, take $S_1$ to be the surface obtained after an annular compression of $S_0$ along $A_0$ and let $A_1$ be an embedded essential annulus in $M \backslash\backslash S_1$ with $\partial_0 A_1 = [\mu]$. Take $S_2$ to be the surface obtained after an annular compression of $S_1$ along $A_1$ and let $A_2$ be an embedded essential annulus in $M \backslash\backslash S_2$ with $\partial_0 A_2 = [\mu]$. Continuing in this fashion, we have a sequence of compressions of $S$ which gives the decomposition into 4-punctured spheres and tubes. Note that $A_i$ deformation retracts onto $\alpha_i$. 
	
	Now, suppose that we have another such collection $\beta_1, \ldots \beta_k$ with corresponding annuli $B_1, \ldots, B_k$. We will show that $\sqcup \alpha_n = \sqcup \beta_n$, up to isotopy and each $\alpha_i$ is parallel to some $\beta_j$. If $\alpha_i$ is disjoint from $\beta_j$ for all $i$ and $j$ and each $\alpha_i$ is parallel to some $\beta_j$, then we are done since this implies that the decompositions of $S$ via $\alpha_i$ and $\beta_j$ produce isotopic $S_C$'s. So, we take the first $\alpha_i$ which intersects some $\beta_j$. This implies that the annuli $A_i$ and $B_j$ also intersect. By Lemma~\ref{DecomposeByAnnuli}, $A_i$ and $B_j$ can be made disjoint. This implies that $\alpha_i$ and $\beta_j$ can also be made disjoint. Continue this process of resolving the intersections of $\alpha_i$'s and $\beta_j$'s until $\sqcup \alpha_n \cap \sqcup \beta_n$ is empty. 
	
	We now prove that each $\alpha_i$ is parallel to some $\beta_j$. Suppose not, and there exists an $\alpha_i$ not parallel to any $\beta_j$. Then in the decomposition of $S$ via the $\beta_j$'s into 4-punctured spheres $S_C$'s and tubes, $\alpha_i$ will be contained in one of the $S_C$'s and is not parallel to any of the punctures. So, we may use $\alpha_i$ to take an annular compression of the $S_C$ along $\alpha_i$. A contradiction to Theorem~\ref{Oertel1}, since $S_C$ is annularly incompressible. Thus, each $\alpha_i$ is parallel to some $\beta_j$. Note that, without loss of generality, it may be the case that there are multiple $\alpha_i$'s parallel to some $\beta_j$. This implies that each of the $\alpha_i$ produces the same annular compression as $\beta_j$. Hence, this implies that the decomposition of $S$ into tubes and $S_C$'s is canonical up to the $S_C$'s. 
	
	There are only two choices, up to isotopy, of incompressible $S_C$'s, by Theorem~\ref{Oertel2}. So, if $S$ and $S'$ are tubings of $S_C$'s and $S_{C'}$'s, where $C \neq C'$, then $S$ and $S'$ canonically decompose into non-isotopic 4-punctured spheres and thus are non-isotopic. Moreover, this also implies that every punctured partitioned chord diagram corresponds to exactly 2 non-isotopic surfaces.
\end{proof}

We now show that punctured partitioned chord diagrams are determined by its chord of length $4g-5$, the maximum possible length. Then, we count all the possible locations of such a chord that give rise to a closed, connected, essential, orientable surface. Observe that by Proposition~\ref{MaxPPCDs}, every punctured partitioned chord diagram corresponds to exactly 2 non-isotopic essential surfaces. Combining these two facts gives the largest number of closed, connected, essential, orientable surfaces of genus $g$ in the complement of a Montesinos knot.

\begin{Prop}\label{UniquelyDetermined}
	Every punctured partitioned chord diagram corresponding to a closed essential surface,  not containing a genus 2 component, is uniquely determined by the location of the chord of maximum possible length $4g-5$.
	
	Moreover, let $c$ be the chord of length $4g - 5$ contained in a region, say $R_1$, and fix $c', c''$ to be the chords of length 1. Let $i$ be the number of base points of $R_1$ which are chorded to base points not in $R_1$. Let $i > 0$ if the set of base points are closer to $c''$ and $i < 0$ if the base points are closer to $c'$. Then, there are $g - 1$ chords parallel to $c$, $\frac{g - 1 - i}{2}$ chords parallel to $c'$ and $\frac{g - 1 + i}{2}$ chords parallel to $c''$. See Figure~\ref{FixedChords}.
\end{Prop}

To prove the proposition, we establish a few lemmas.

\begin{Lemma}\label{ExistenceMaxChord}
	Every punctured partitioned chord diagram coming from a closed essential surface, not containing a genus 2 component, in a Montesinos knot complement must have a unique chord of maximum possible length. In other words, each punctured partitioned chord diagram coming from such a closed essential surface has exactly one chord of length $4g-5$.
\end{Lemma}
	\begin{proof}
		Suppose that the punctured partitioned chord diagram $D$ of such a surface $S$ has no chord of length $4g-5$. There is only one way to obtain a chord of length $4g-5$. Namely, by chording adjacent base points so that the chord is contained in all four regions (e.g. see blue chord in Figure~\ref{PPCD}). By Proposition~\ref{TreeLeaves}, $D$ must have exactly three leaves. Thus, all three leaves of $D$ must have corresponding chords with length 1 by Lemma~\ref{1MaxChord}. These leaves are only allowed to be located at four possible places; namely between base points of adjacent regions. If one of the three leaves is not between base points of adjacent regions, then such a leaf must correspond to a chord of length 1 completely contained in a single region. Contradiction to Theorem~\ref{Oertel2}, since this chord would not pass through a rational tangle. Thus, all three leaves must be between base points of adjacent regions. However this gives another contradiction, since this implies that $S$ has a genus 2 component by Proposition~\ref{TreeLeaves}. 
	\end{proof}

\begin{Lemma}\label{ChordDiagram2Leaves}
	Every general chord diagram which
	\begin{enumerate}
		\item contains four or more base points,
		\item does not contain punctures,
		\item and does not contain intersecting chords 
	\end{enumerate}
	must have at least two chords of length 1.
	
	Moreover, if the two sides, with respect to two leaves, do not contain the same number of base points, then the chord diagram contains three or more chords of length 1. 
\end{Lemma}
	\begin{proof}
		Let $D$ be a chord diagram which does not contain punctures or intersecting chords. Observe that $D$ must contain only odd length chords. Otherwise, one of the complementary regions of an even length chord contains an odd number of base points. By definition, every chord has exactly two base points and every base point belongs to a unique chord. Thus, an odd number of base points implies that there must be intersecting chords, a contradiction. So, every chord in $D$ is odd length.
		
		Now, we can find a chord of largest length in $D$, since there are only finitely many chords. If the largest length among all chords is 1, then every chord must be length 1. Thus, since $D$ contains four or more base points, then $D$ has at least two chords of length 1. 
		
		Suppose that a largest length chord in $D$ is not length 1. Call this chord $c_0$. The complement of $c_0$ in $D$ consists of two sub-diagrams, say $D_1$ and $D_2$. We will consider both $D_1$ and $D_2$ to be sub-diagrams containing the chord $c_0$ and the base points of $c_0$. Observe that when $c_0$ is viewed as a chord in $D_1$ and $D_2$, it is a chord of length 1. Without loss of generality, we look at the sub-diagram $D_1$. Note that both $D_1$ and $D_2$ satisfy conditions \textit{1} through \textit{4}. Continue this process of locating the largest length chord in $D_1$, call it $c_1$. Observe that the length of $c_1$ in $D$ is less than or equal to the length of $c_0$ in $D$ since $c_0$ was the largest length chord in $D$ and chords do not intersect. Moreover, this process of finding the largest length chord in each consecutive sub-diagram must stop at a length 1 chord. If not, then at some point in this process a sub-diagram must consist of chords of the same length and must have length 3 or more. However, the only way this can happen is if chords intersect in the sub-diagram. Contradiction, since these intersecting chords in this sub-diagram must also be intersecting chords in $D$. Thus, we have found a length 1 chord in $D_1$ which is a chord between adjacent base points of both $D_1$ and $D$. This process also applies to the $D_2$ and thus $D$ has at least two chords of length 1.
		
		The second part of the statement can be proved using the same proof of the case of two leaves in Proposition~\ref{TreeLeaves}. Assume the two sides, with respect to two leaves, do not contain the same number of base points. Then there are base points on one side of the leaves which cannot have a chord to a base point on the other side of the leaves. Hence, there must exist a chord between adjacent base points. Otherwise, if all remaining base points are not adjacent, then in order to chord the remaining base points, it must intersect the chords whose base points lie on different sides of the leaves, which gives a contradiction to chords being disjoint. Thus, there exists a chord between adjacent base points, which implies that the dual tree has at least 3 leaves.
	\end{proof}

\begin{Lemma}\label{Length1LeavesLocation}
	The location of the two length 1 chords of a punctured partitioned chord diagram of a closed, connected, essential, orientable surface of genus $g$ greater than 2 must be located between base points of regions disjoint from the region containing the base points of the chord of maximum possible length.
\end{Lemma}
	\begin{proof}
		Let $D$ be the punctured partitioned chord diagram of such a surface. By Proposition~\ref{TreeLeaves} and Lemma~\ref{ExistenceMaxChord}, $D$ contains exactly one chord of maximum possible length and two chords of length 1. Label the chord of maximum possible length as $c_0$, and the two chords of length 1 as $c_1, c_2$. Observe that $D$ has four regions each of which contain $g - 1$ base points. Label the region containing the base points of the chord of maximum possible length $R_1$. Going counterclockwise from $R_1$, label the other three regions as $R_2, R_3$ and $R_4$, respectively (see Figure~\ref{RegionsOfPPCD}). 
		
		The case for genus $g = 3$ is done, since each region has exactly two base points. Thus, there are only two possible places for the two length 1 chords: between $R_2$ \& $R_3$ and $R_3$ \& $R_4$. Any other location would imply that $D$ has a chord of length 1 completely contained in a singe region. Contradiction to Theorem~\ref{Oertel2}, since this chord would correspond to a tube that does not pass through a rational tangle.

\begin{figure}[h!]
		\centering
		\includegraphics[viewport = 0 210 620 690, scale = 0.4, clip]{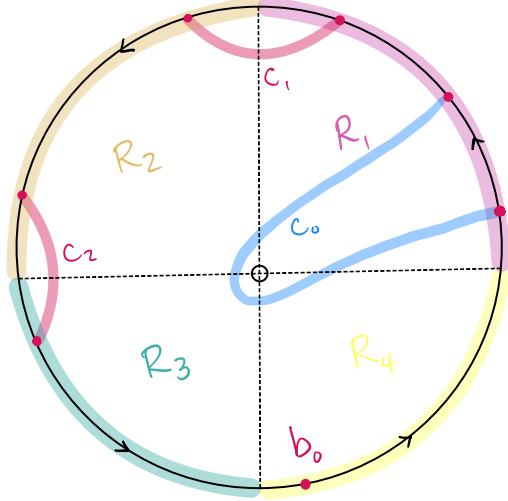}
		\caption{Labeling of the punctured partitioned chord diagram $D$.}
		\label{RegionsOfPPCD}
\end{figure}

		Now we prove the case of $g > 3$ by contradiction. Suppose without loss of generality that $c_1$ has one base point in $R_1$ and the other base point in $R_2$. Then there are only two locations for $c_2$: between $R_1$ \& $R_4$ and $R_2$ \& $R_3$. Note that if $c_2$ is between $R_3$ \& $R_4$, then by Proposition~\ref{TreeLeaves}, this would mean that there is a disconnected component of the closed, connected, essential, orientable essential surface, a contradiction. Without loss of generality, say that $c_2$ is between $R_2$ \& $R_3$. The same argument applies if $c_2$ is between $R_1$ \& $R_4$. We mimic the proof of Lemma~\ref{ChordDiagram2Leaves}, where we reduce $D$ to sub-diagrams by looking at the complement of a chord in $D$. Note that $D$ contains $g - 4$ base points in $R_1$, $g - 3$ base points in $R_2$, $g - 2$ base points in $R_3$, and $g - 1$ base points in $R_4$. By Lemma~\ref{ChordDiagram2Leaves}, both of these new diagrams must have at least two chords of length 1. We show that one of the sub-diagrams contains a chord of length 1 which is also a chord of length 1 in $D$, distinct from $c_1, c_2$. This gives a fourth leaf in $D$, contradicting Proposition~\ref{TreeLeaves}. 
		
		Let $b_0$ be a base point in $R_4$ which is immediately adjacent to the region $R_3$. We show that every possible chord with $b_0$ as a base point results to $D$ having a third chord of length 1 contradicting Proposition~\ref{TreeLeaves}. Let $c$ be the chord with $b_0$ as one of its base points. If $c$ is length 1, we have a third length 1 chord in $D$, contradicting Proposition~\ref{TreeLeaves}. Suppose that $c$ is not length 1. Since no chords in $D$ intersect, then $c$ partitions $D$ into two sub-diagrams. There are two cases to consider: one of the sub-diagrams contain at least one of $c_0, c_1, c_2$ or contain none of these chords. Suppose one of the sub-diagrams, call it $D'$, does not contain any of $c_0, c_1$ or $c_2$. This implies, by Lemma~\ref{ChordDiagram2Leaves}, that there exists two chords of length 1 in $D'$, one of which is $c$, and hence the other chord is of length 1 in $D$ which is not $c_1$ or $c_2$, a contradiction. 
		
		Now suppose that one of the sub-diagrams contains at least one of $c_0, c_1, c_2$. We look at the sub-diagram $D'$ that contains exactly one of $c_0, c_1, c_2$. Without loss of generality, suppose $D'$ contains $c_2$. Observe that $b_0$ has at least $g - 2$ base points in $D$ between itself and any base point of $c_0, c_1, c_2$. This is because $R_3$ has $g - 2$ base points disjoint from $c_2$ and $R_4$ has $g - 2$ base points disjoint from $b_0$. There does not exist $g - 2$ consecutive base points in either $R_2$ and $R_1$ in between leaves. Thus, there must be a different number of base points between $b_0$ \& $c_2$ and $b_1$ \& $c_2$ in $D'$. This means that the two sides of $D'$, with respect to the leaves $c$ and $c_2$, must contain a different number of base points in $D'$. Two of these leaves in $D'$ are $c$ and $c_2$, so by Lemma~\ref{ChordDiagram2Leaves} $D'$ has three or more chords of length 1. The last chord is between base points adjacent in $D'$ which are also adjacent in $D$. This gives a third chord of length 1 in $D$, contradicting Proposition~\ref{TreeLeaves}. 
	\end{proof}

\begin{Prop}\label{TypesOfChords}
	Let $D$ be a punctured partitioned chord diagram with a chord $c$ of maximum possible length $4g - 5$ whose base points are completely contained in a single region and exactly two chords, $c', c''$, of length 1 whose base points are in regions distinct from the region containing the base points of $c$. Then there are exactly three types of chords in $D$, whose type is based on if exactly one of $c, c', c''$, is completely contained in one of the sub-diagrams of the chord.
\end{Prop}
	\begin{proof}
		This is an immediate consequence of Lemma~\ref{ChordDiagram2Leaves}. Suppose there is a chord $\overline{c}$ which partitions $D$ into two sub-diagrams $D'$, $D''$ such that exactly one of $D'$ and $D''$ does not contain any of $c, c', c''$. Without loss of generality, suppose $D'$ does not contain any of $c, c', c''$. Then, by Lemma~\ref{ChordDiagram2Leaves}, there exists a chord of length 1 in $D'$ which is not $c, c'$ or $c''$. Thus, this chord is also of length 1 in $D$. Contradiction, since $c'$ and $c''$ were the only chords of length 1 in $D$. Therefore, every chord in $D$ must partition $D$ into two sub-diagrams such that one of the sub-diagrams contains exactly one of $c, c'$ or $c''$.
	\end{proof}

\begin{Lemma}\label{ChordsOfMaxLength}
	Suppose $D$ is a punctured partitioned chord diagram as in Proposition~\ref{TypesOfChords}. Chords which have a base point in the region containing the base points of $c$ must have a sub-diagram which contains $c$, but does not contain $c'$ or $c''$.
	
	Moreover, the base points of such chords are equidistant to the base points of $c$. In other words, the chords are parallel to $c$.
\end{Lemma}
	\begin{proof}
	Let $4(g - 1)$ be the number of base points in a punctured partitioned chord diagram $D$, and that each of the four regions has $g - 1$ base points. The case for $g = 3$ was proved in Lemma~\ref{Length1LeavesLocation}, so we may assume that $g \geq 4$.
	
	Fix the location of the chord of maximum possible length, call it $c$. Without loss of generality, let $R_1$ be the region containing the base points of $c$. Then by Lemma~\ref{Length1LeavesLocation}, the two chords of length 1, call them $c', c''$, must be located between $R_2$ \& $R_3$ and $R_3$ \& $R_4$, respectively (see Figure~\ref{FixedChords}). This is because chords of length 1 cannot be between base points in a single region, contradicting Theorem~\ref{Oertel2}. The chord of maximum possible length cannot be between base points of adjacent regions since this implies that the surface is disconnected, by Proposition~\ref{TreeLeaves}. Moreover, Lemma~\ref{Length1LeavesLocation} gives that the location of length 1 chords are determined by the location of the chord of maximum length.

\begin{figure}[h!]
	\centering
	\includegraphics[viewport = 0 190 640 730, scale = 0.4, clip]{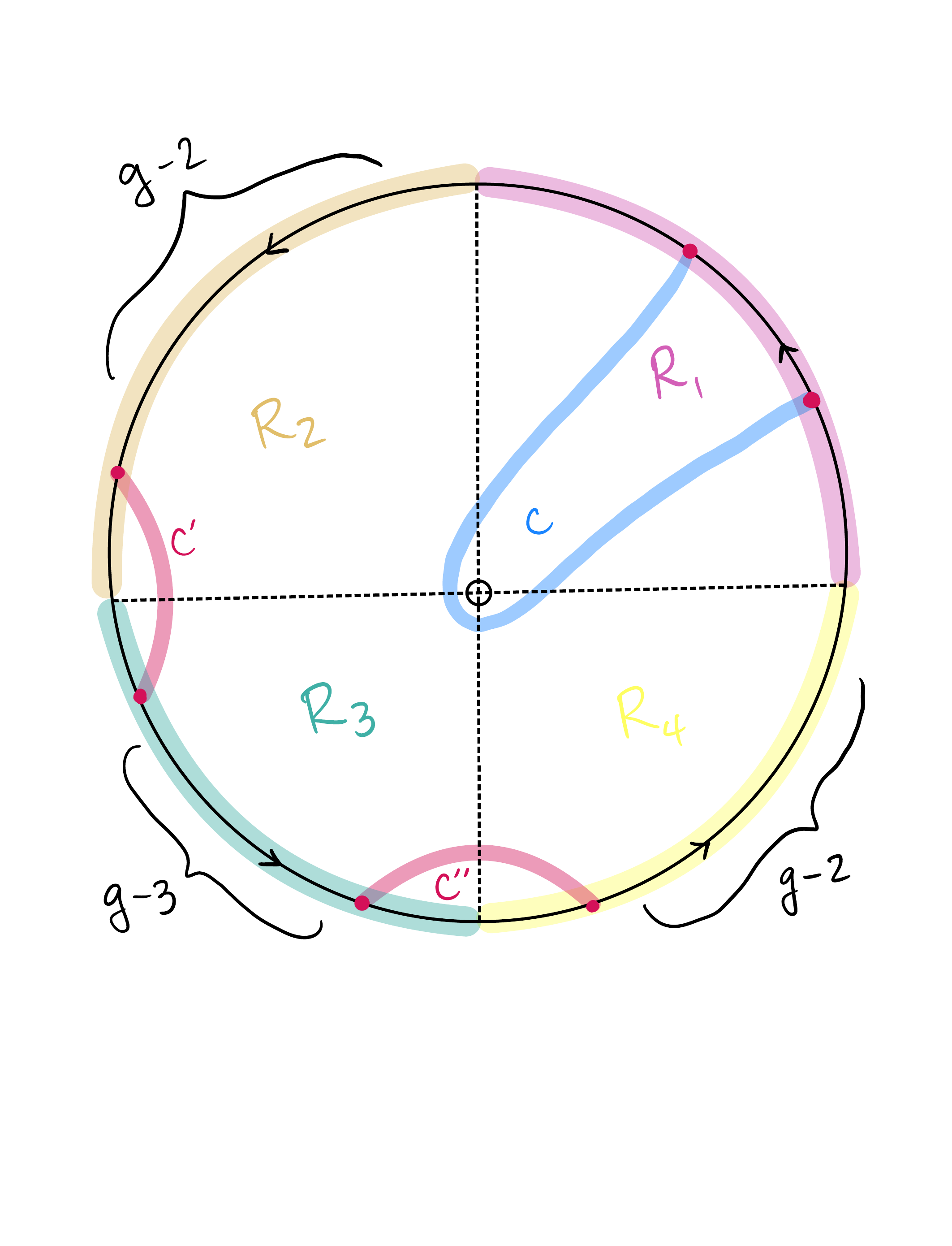}
	\caption{The three leaves of a punctured partitioned chord diagram, along with the chords $c, c', c''$ and the number of base points in each region.}
	\label{FixedChords}
\end{figure}

	Suppose a chord $\overline{c}$ has a base point in $R_1$ and neither sub-diagram does not only contain $c$. Thus, one sub-diagram, call it $D'$, either contains only $c'$ or $c''$. However, in either case, the distance between one pair of base points is always at least $g - 2$ and the distance between the other set of base points is at most $g - 3$. Hence, $D'$ does not have the same number of base points on both sides of the leaves. So, by Lemma~\ref{ChordDiagram2Leaves}, there exists another leaf in $D'$ that is not $\overline{c}$ or either of $c'$ or $c''$. This implies that this leaf is also a leaf in $D$ with corresponding chord having length 1. Contradiction, since $D$ has only two chords of length 1.
	
	We now prove the second statement of the lemma. Suppose the base points of a chord $\overline{c}$, which has one base point in $R_1$, does not have both its base points equidistant to the base points of $c$. Thus, a sub-diagram of $\overline{c}$ does not have the same number of base points on both sides of the leaves. By Lemma~\ref{ChordDiagram2Leaves}, there exists another leaf in $D'$ that does not correspond to $\overline{c}$ or $c$. Thus, this leaf corresponds to a chord of length 1 in $D$ that is not $c'$ or $c''$, again a contradiction.
	\end{proof}

Using the established lemmas, we prove Proposition~\ref{UniquelyDetermined}.

\begin{proof}[Proof of Proposition~\ref{UniquelyDetermined}]	
	By Proposition~\ref{TypesOfChords}, there are only three types of chords. We show that there is only one way to construct a punctured partitioned chord diagram with these types of chords which corresponds to a tubing description of a closed essential surface of genus $g$. Recall that $i$ represents the number of base points in $R_1$ chorded to base points not in $R_1$. Note that $0 \leq i \leq g - 3$. By Lemma~\ref{ChordsOfMaxLength}, we need only consider a chord diagram $\overline{D}$ with the following properties: $\overline{D}$ does not contain a puncture, contains $0$ base points in $R_1$, $g - 2 - i$ base points in $R_2$, $g - 3$ base points in $R_3$, and $g - 2$ base points in $R_4$ (see Figure~\ref{ConstructPPCD}). Observe that the parity of $g$ and $i$ must be opposite, otherwise $g - 1$ and $i$ would have the opposite parity. Since $g - 1$ is the number of base points in $R_1$, then $g - 1 - i$ must be odd. However, $g - 1 - i$ is the number of base points in $R_1$ chorded together, a contradiction since chording an odd number of base points in the same region would give two intersecting chords. 
	
\begin{figure}[h!]
	\centering
	\includegraphics[viewport = 0 190 625 700, scale = 0.4, clip]{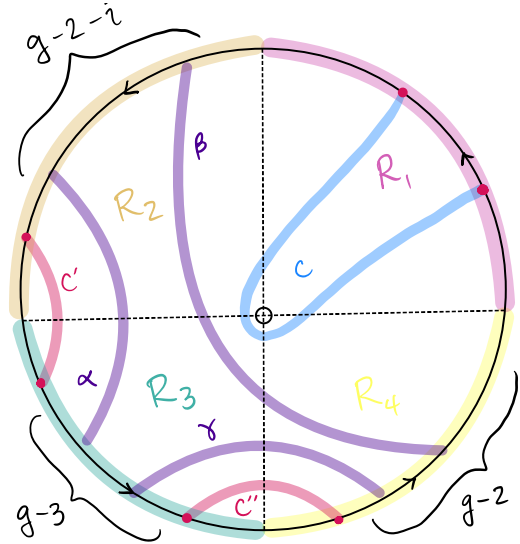}
	\caption{The only allowable chords, drawn in purple, in $\overline{D}$, along with the chords $c, c', c''$ and the number of base points in each region.}
	\label{ConstructPPCD}
\end{figure}

	There are only three types of allowable chords for $\overline{D}$, as shown in Figure~\ref{ConstructPPCD}. This is because any other types of chords would result in a fourth leaf in the dual tree, by Lemma~\ref{ChordDiagram2Leaves}. A contradiction to Proposition~\ref{TreeLeaves}. A chording of the chord diagram $\overline{D}$ can be expressed as a system of linear equations in terms of the number of types of allowable chords $\alpha, \beta, \gamma$ and the number of base points. Namely, the system of equations is
	\begin{align*}
		\alpha + \beta &= g - 2 - i\\
		\alpha + \gamma &= g - 3\\
		\beta + \gamma &= g - 2
	\end{align*}
which can be expressed in matrix notation as 
	\begin{align*}
		\begin{bmatrix}
		1 & 1 & 0\\
		1 & 0 & 1\\
		0 & 1 & 1
		\end{bmatrix}
		\begin{bmatrix}
		\alpha \\
		\beta\\
		\gamma
		\end{bmatrix}
		=
		\begin{bmatrix}
		g - 2 - i \\
		g - 3\\
		g - 2
		\end{bmatrix}
	\end{align*}
Since the matrix has non-zero determinant, then if there is a solution, in particular we require it must be non-negative integer solutions, it must be unique. The solution can be expressed in terms of $g$ and $i$ as 
	\begin{align*}
		\alpha &= g - \frac{g + i -3}{2} - 3 = \frac{g - 3 - i}{2}\\
		\beta &= \frac{g + i - 3}{2} - i + 1 =\frac{g - 1 - i}{2}\\
		\gamma &= \frac{g + i - 3}{2}
	\end{align*}
Each of $\alpha, \beta,$ and $\gamma$ are integers, since $g$ and $i$ have opposite parity. By assumption, $g \geq 4$ and hence $g + i - 3$ is an even number between $0$ and $2(g - 3)$, inclusive. Observing that each of $\alpha, \beta,$ and $\gamma$ are non-negative and integral, we have a unique positive integer solution. Therefore, this unique solution corresponds to the unique punctured partitioned chord diagram and the proposition is proved. 

The second statement of the proposition follows from Lemma~\ref{ChordsOfMaxLength} and the number of $\alpha, \beta$ and $\gamma$ chords. Lemma~\ref{ChordsOfMaxLength} implies that there are $i$ and $\frac{g - 1 - i}{2}$ chords parallel to $c$. Moreover, the $\beta$ chords and $c$ itself are also parallel to $c$, so the total number of chords parallel to $c$ is $i + \frac{g - 1 - i}{2} + \frac{g - 3 - i}{2} + 1 = g - 1$ chords parallel to $c$. The $\alpha$ chords and $c'$ itself are all parallel to $c'$, so there are $\frac{g - 3 - i}{2} + 1 = \frac{g - 1 - i}{2}$ parallel to $c'$. Lastly, the $\gamma$ chords and $c''$ itself are parallel to $c''$, so there are $\frac{g + i - 3}{2} + 1 = \frac{g - 1 + i}{2}$ chords parallel to $c''$.
\end{proof}

We will find the punctured partitioned chord diagrams that give a description of a closed, connected, essential, orientable surface by utilizing Claim 2 of \cite{lee2021essential} which uses tools of \cite{AHW2002}.

\begin{Prop}\label{MaxPPCDsGeneral}
	 Let $\phi(g - 1)$ be the Euler totient function of $g - 1$ and let $g > 2$. There are at most $8\phi(g - 1)$ punctured partitioned chord diagrams which correspond to tubing descriptions of closed, connected, essential, orientable surfaces of genus $g$.
\end{Prop}
	\begin{proof}
		We count the number of punctured partitioned chord diagrams that give a tubing description of a closed, connected, essential, orientable surface. Fix $c$ to be the chord of maximum possible length in the diagram. Without loss of generality, let $c$ be in $R_1$. A disconnected surface happens exactly when a proper subset of the collection of 4-punctured spheres $S_C$'s from Theorem~\ref{Oertel1} are tubed together. From the Euler characteristic argument in Example~\ref{Example}, there are $g - 1$ $S_C$'s of a closed surface of genus $g$. Label the 4-punctured spheres and punctures as $S_{C,j}^i$ where $i$ represents the $i^{\text{th}}$ innermost $S_C$ with respect to the projection of the knot and $j$ represents the $j^{\text{th}}$ arc that the punctures lie on. Thus, for a disconnected surface, we would get a collection of chords which have a set of base points of the form $\{S_{C,1}^{i_1}, S_{C,2}^{i_1}, S_{C,3}^{i_1}, S_{C,4}^{i_1}, \ldots, S_{C,3}^{i_k}, S_{C,4}^{i_k}\}$, where $0 < k < g - 1$. 
		
		Since we want to eliminate the case that a proper subset of $S_C$'s are tubed together, then we may take the viewpoint of \cite{AHW2002} via integer pairings as follows. Fix a projection of the knot along with the $S_C$'s and then associate to each $S_C^i$ the integer $i$, where $i = 1$ represents the innermost and $i = g - 1$ represents the outermost $S_C$ (see Figure~\ref{ConnectedPPCD}). We may think of chords as pairings where base points will be viewed as integers and a chord is a map from one integer to another. Moreover, since we only want to find which $S_C$'s are tubed together, then we get additional pairings based on identifying the punctures belonging to the same $S_C$.
		
		Observe that the problem now reduces to finding orbits of pairings, since every orbit corresponds to exactly one component of the surface obtained from the tubing description of the punctured partitioned chord diagram. Thus, we find the number of orbits coming from the chords and the identification of the base points coming from the same $S_C$. Note that the chords which are parallel to $c'$ and the chords parallel to $c''$ do not contribute to the number of orbits, since these chords have base points belonging to the same $S_C$. By Proposition~\ref{UniquelyDetermined}, there are $g - 1$ such chords. Since in any punctured partitioned chord diagram there are $2(g - 1)$ chords, then we look at the orbits of the remaining $g - 1$ chords.

\begin{figure}[h!]
\centering
	\begin{minipage}[b]{0.45\linewidth}
		\centering
		\includegraphics[viewport = 20 170 710 710, scale = 0.36, clip]{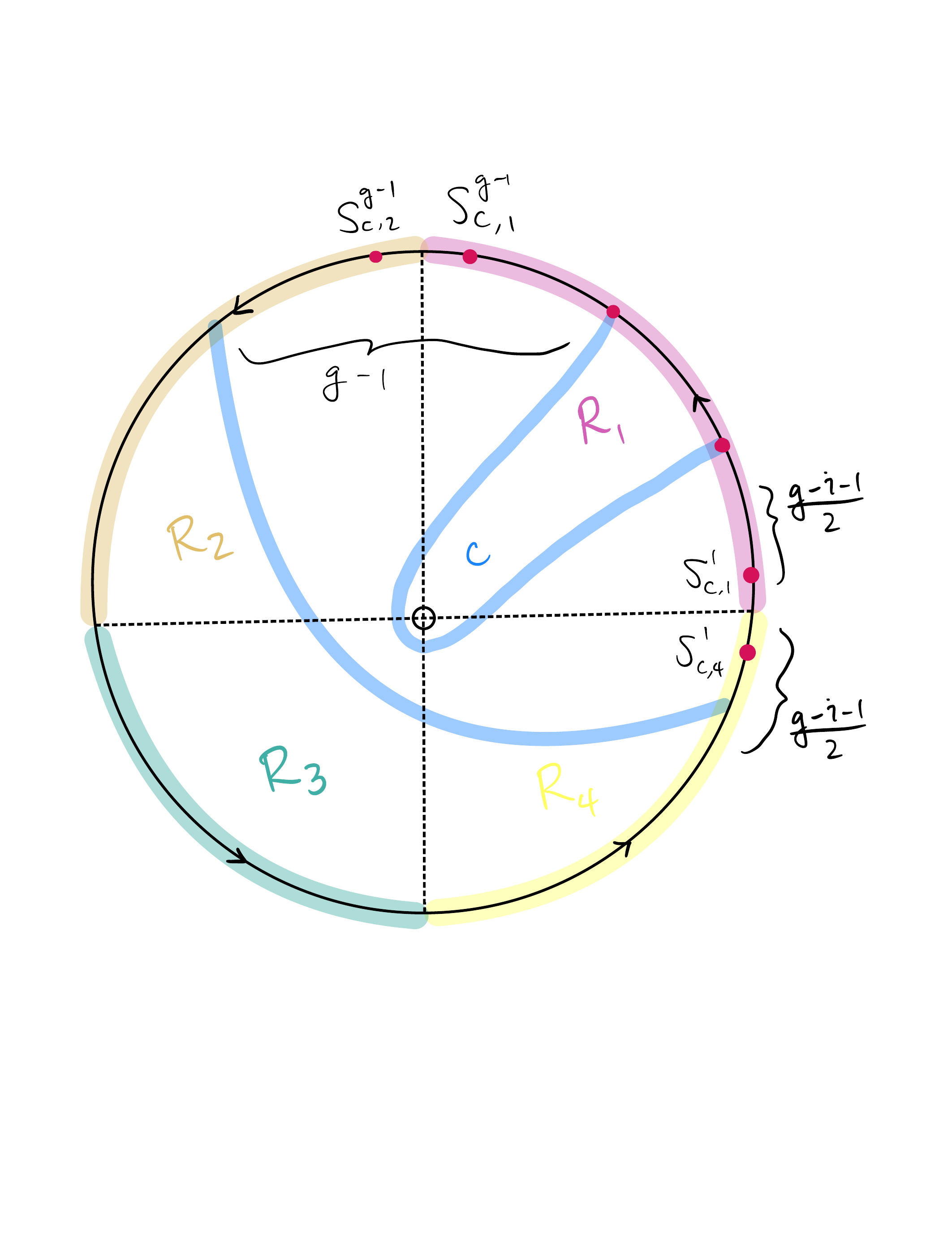}
		\caption{The chords and pairings along with the number of orbits.}
		\label{ConnectedPPCD}
	\end{minipage}
	\begin{minipage}[b]{0.45\linewidth}
		\centering
		\includegraphics[viewport = 0 270 650 650, scale = 0.36, clip]{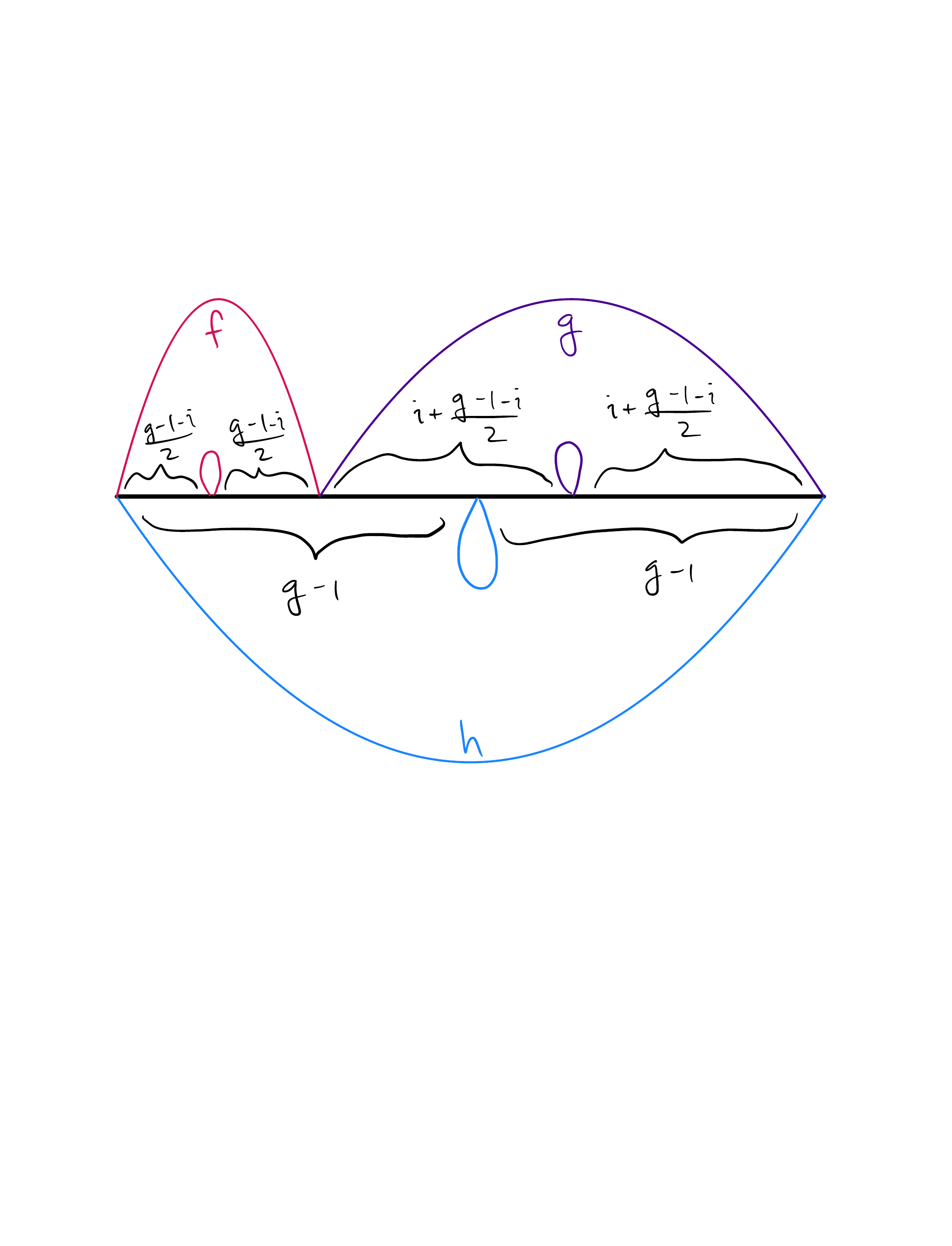}
		\caption{The orientation reversing pairings $f, g,$ and $h$ obtained from a punctured partitioned chord diagram.}
		\label{OrbitsOfPPCD}
	\end{minipage}
\end{figure}

		By Proposition~\ref{UniquelyDetermined}, the remaining $g - 1$ chords are parallel to $c$. We may view these chords together as one pairing, call it $h$, which is orientation reversing, with $g - 1$ orbits. By Proposition~\ref{UniquelyDetermined}, there are $\frac{g - 1 - i}{2}$ $S_C$'s of $R_4$ having chords parallel to $c$. These base points will be identified with $\frac{g - 1 - i}{2}$ base points in $R_1$ since these base points belong the same $S_C$. Since $R_1$ and $R_4$ are adjacent, this gives an orientation reversing pairing, call it $g$, with $\frac{g - 1 - i}{2}$ orbits. Lemma~\ref{ChordsOfMaxLength} implies that the remaining $i + \frac{g - 1 - i}{2}$ chords have base points in $R_1$ and $R_2$ which belong to $i + \frac{g - 1 - i}{2}$ $S_C$'s. Again, since $R_1$ and $R_2$ are adjacent, gives the last orientation reversing pairing, call it $f$, with $i + \frac{g - 1 - i}{2}$ orbits (see Figure~\ref{OrbitsOfPPCD}).

		Let $u = \frac{g - 1 - i}{2}$ and $v = i + \frac{g - 1 - i}{2}$ which gives $u + v = g - 1$. Applying Claim 2 of \cite{lee2021essential} to $u, v, u + v, f, g$ and $h$, we get that there are is only one orbit if and only if $gcd(u, v) = 1$. This gives that there are at most $\phi(g - 1)$ locations for $c$ in $R_1$ which gives a tubing description yielding a closed, connected, essential, orientable surface. Moreover, by symmetry of the punctured partitioned chord diagram, there are four other regions each with $\phi(g - 1)$ choices for $c$ that gives a tubing description yielding a closed, connected, essential, orientable surface. Since there are two choices for the $S_C$, we get that the total number of closed, connected, essential, orientable surfaces is at most $8\phi(g - 1)$.
	\end{proof}

%
%
%
\section{Distinctness of Surfaces from Punctured Partitioned Chord Diagrams} \label{DistinctnessOfPPCDs}

Lastly, we prove Theorem~\ref{main} by showing that each punctured partitioned chord diagram giving a tubing description of a closed, connected, essential, orientable surface gives distinct, non-isotopic surfaces. We give necessary definitions used to prove distinctness. Distinctness will be proved in two cases: for punctured partitioned chord diagrams corresponding to surfaces of genus greater than 2 and then for genus 2 surfaces. The genus greater than 2 case will contain two subcases: when the chords of maximum possible length of two punctured partitioned chord diagrams are contained in the same region and when they are in separate regions. The former case will utilize incompressible branched surfaces and the latter case will utilize results of Waldhausen \cite{MR224099} on isotopies between incompressible surfaces.\\

For an in-depth treatment of incompressible branched surfaces, see \cite{FLOYD1984117} and \cite{Oertel1984385}. A \textit{branched surface} $\mathcal{B}$ in a 3-manifold $M$ is a subspace locally modeled as in Figure~\ref{LocalModel}. If $\mathcal{B}$ is embedded in $M$, then $N = N(\mathcal{B})$ denotes a fibered regular neighborhood of $\mathcal{B}$ in $M$ as in Figure~\ref{Neighborhood}. The portions of $\partial N - \partial M$ called $\partial_h N$ (horizontal boundary) and $\partial_v N$ (vertical boundary) are also shown in {Figure~\ref{Neighborhood}. 
A \textit{monogon} in $M - \mathring{N}$ is a disk $D$ with $D \cap N = \partial D$ which intersects $\partial_v N$ in a single fiber. A \textit{disk of contact} is a disk $D$ embedded in $N$ transverse to fibers with $\partial D \subseteq \text{int}(\partial_v N)$.

\begin{figure}[h!]
\centering
	\begin{minipage}[b]{0.45\linewidth}
		\centering
		\includegraphics[viewport = 20 60 640 640, scale = 0.36, clip]{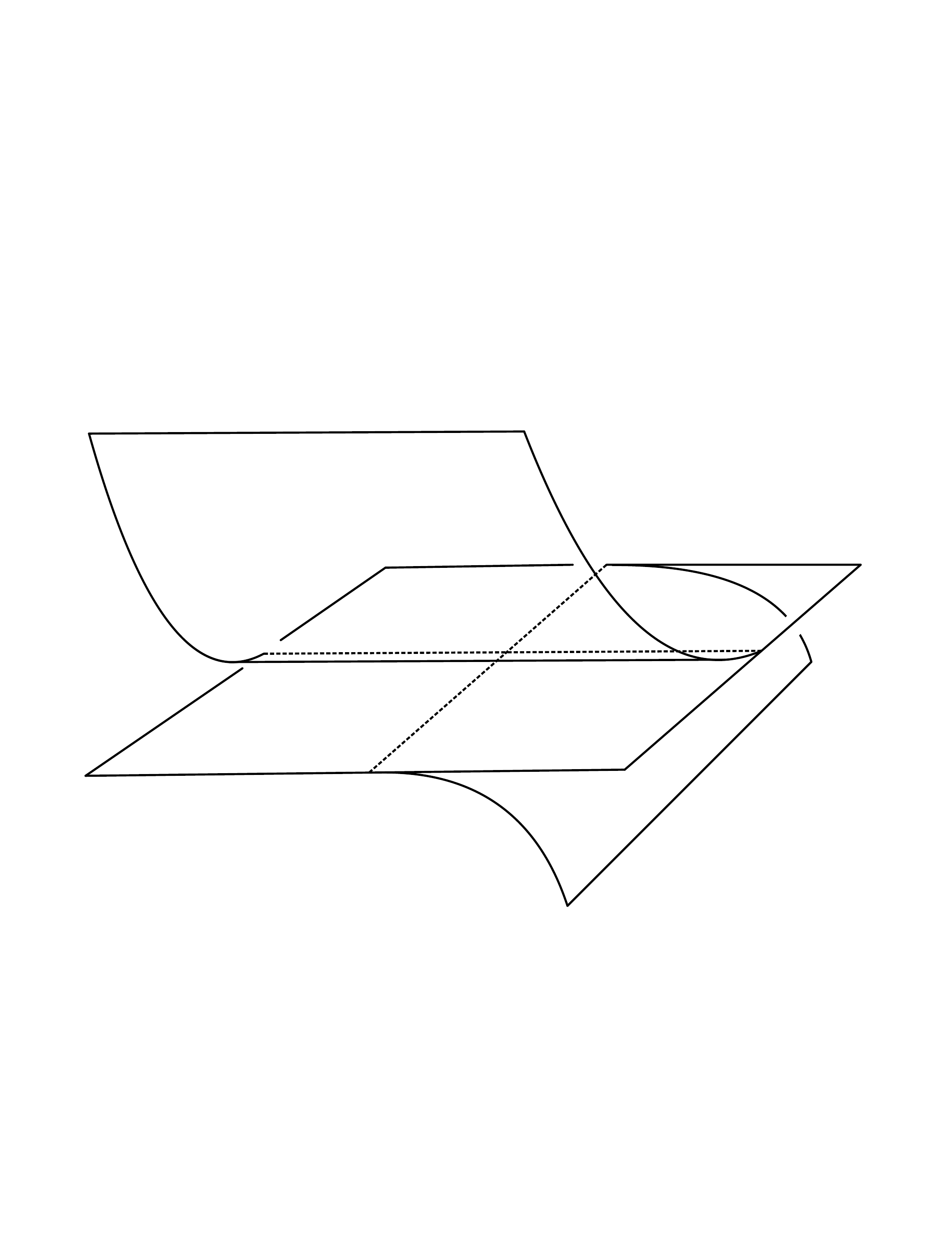}
		\caption{Local model of a branched surface $\mathcal{B}$.}
		\label{LocalModel}
	\end{minipage}
	\begin{minipage}[b]{0.45\linewidth}
		\centering
		\includegraphics[viewport = 0 225 670 610, scale = 0.36, clip]{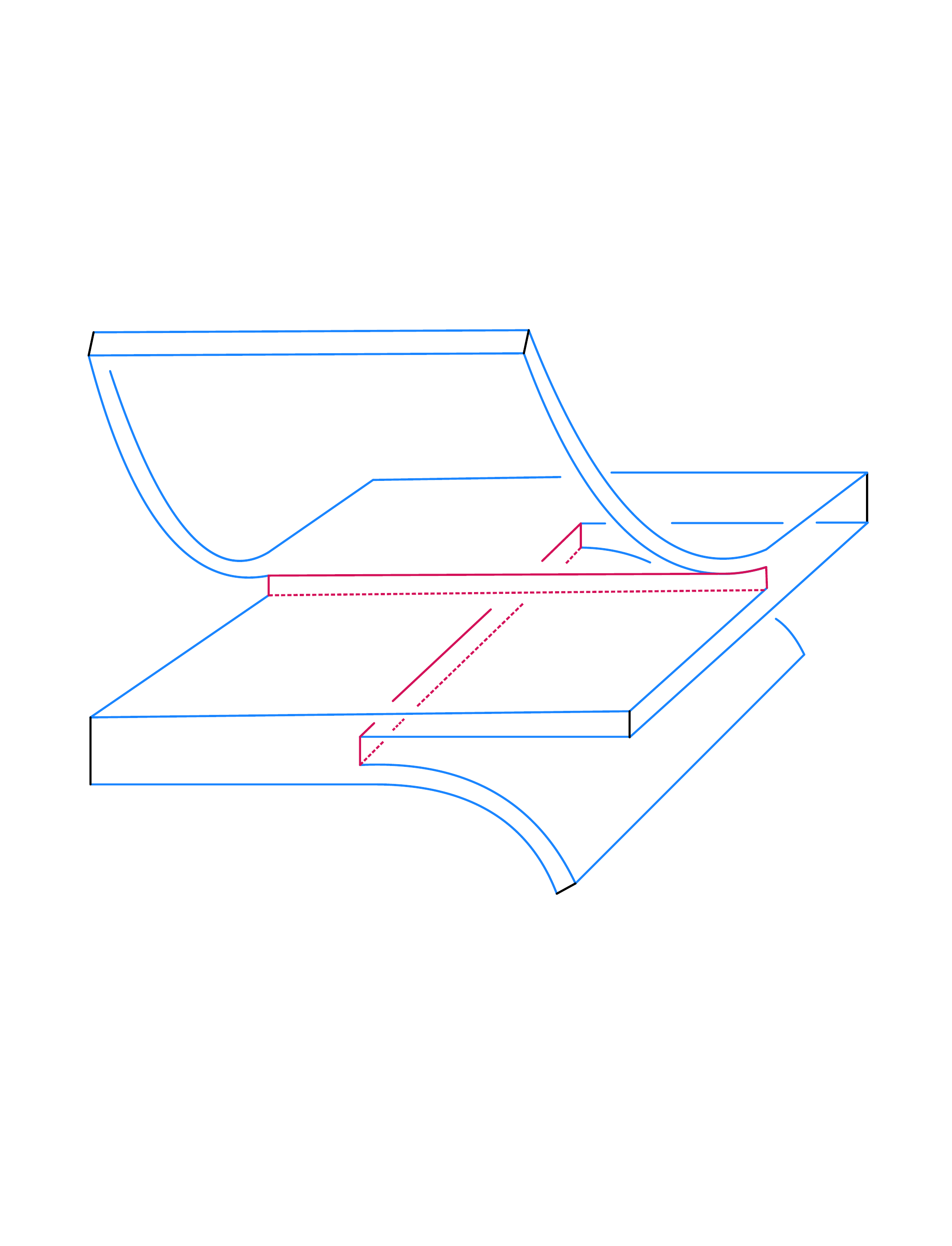}
		\caption{A local neighborhood $N(\mathcal{B})$ of the branched surface $\mathcal{B}$ with the horizontal boundary components drawn in blue and the vertical boundary components drawn in red.}
		\label{Neighborhood}
	\end{minipage}
\end{figure}

An \textit{incompressible branched surface} $\mathcal{B}$ is an embedded branched surface which satisfies the following three conditions:
	\begin{enumerate}
		\item $\mathcal{B}$ has no disks of contact or half-disks of contact,
		\item $\partial_h N$ is incompressible and $\partial$-incompressible in $M - \mathring{N}$, where a $\partial$-compressing disk for $\partial_h N$ is assumed to have boundary in $\partial M \cup \partial_h N$,
		\item and there are no monogons in $M - \mathring{N}$.
	\end{enumerate}

We state three results which will be used throughout the proof of distinctness.

\begin{Th}[Oertel \cite{Oertel1984385}, Theorem 1]\label{Oertel4}
	Suppose $\mathcal{B}$ is an incompressible branched surface and suppose $\mathcal{B}(\bf{w})$ is a 2-sided surface carried by $\mathcal{B}$ with positive weights ($w_i > 0$ for all $i$). If $\mathcal{B}(\bf{v})$ is any surface carried by $\mathcal{B}$ ($v_i \geq 0$ for all $i$), then $\mathcal{B}(\bf{v})$ is isotopic to $\mathcal{B}(\bf{w})$ if and only if $\mathcal{B}(\bf{v})$ can be isotoped to $\mathcal{B}(\bf{w})$ using a sequence of simple isotopies (and isotopies in $N(\mathcal{B})$ respecting fibers of $N(\mathcal{B})$).
\end{Th}

\begin{Prop}[Oertel \cite{FLOYD1984117}, Proposition 3]\label{IncompBranch}
	Let $M$ be an irreducible 3-manifold and $S$ a 2-sided incompressible surface embedded in $M$. Then the branched surface $\mathcal{B}_S$, constructed by using $S$, is incompressible.
\end{Prop}

For the general construction of such a branched surface, see \cite{FLOYD1984117}. For the specific case of Montesinos knots with 4 rational tangles, see Section~\ref{ConstructionOfIncompBS} below.

\begin{Prop}[Waldhausen \cite{MR224099}, Proposition 5.4]\label{Waldhausen}
	Let $M$ be an irreducible $3-$manifold. In $M$, let $F$ and $G$ be incompressible surfaces, such that $\partial F \subseteq \partial F \cap \partial G$, and $F \cap G$ consists of mutually disjoint simple closed curves, with transversal intersection at any curve which is not in $\partial F$. Suppose there is a surface $H$ and a map $f : H \times I \rightarrow M$ such that $f | H \times 0$ is a covering map onto $F$, and $$f(\partial(H \times I) - H \times 0) \subset G.$$ Then there is a surface $\tilde{H}$ and an embedding $\tilde{H} \times I \rightarrow M$, such that $$\tilde{H} \times 0 = \tilde{F} \subset F, (\partial (\tilde{H} \times I) - \tilde{H} \times 0) = \tilde{G} \subset G$$ (i.e., a small piece of $F$ is parallel to a small piece of $G$), and that moreover $\tilde{F} \cap G = \partial \tilde{F}$, and either $\tilde{G} \cap F = \partial \tilde{G}$, or $\tilde{F}$ and $\tilde{G}$ are disks. 
\end{Prop}

We use Theorem~\ref{Oertel4} for the case where two punctured partitioned chord diagrams have their maximum possible length chords in the same region. When punctured partitioned chord diagrams have their maximum possible length chords in different regions, we utilize Proposition~\ref{Waldhausen}. 

By Theorem~\ref{Oertel2}, every incompressible surface is carried by the branched surface $\mathcal{B}$ where $\mathcal{B}$ is shown in Figure~\ref{CarrierBranch} for $K(\frac{1}{3}, \frac{1}{3}, \frac{2}{3}, \frac{1}{2})$. However, to use Theorem~\ref{Oertel4}, we need to take a branched subsurface $\mathcal{B}_S$, constructed from a closed connected incompressible surface $S$, which carries all surfaces having their corresponding maximum possible length chord in the same region. Thus, we define the following components of the branched surface $\mathcal{B}$.

\begin{figure}[h!]
		\centering
		\includegraphics[viewport = 0 170 620 740, scale = 0.55, clip]{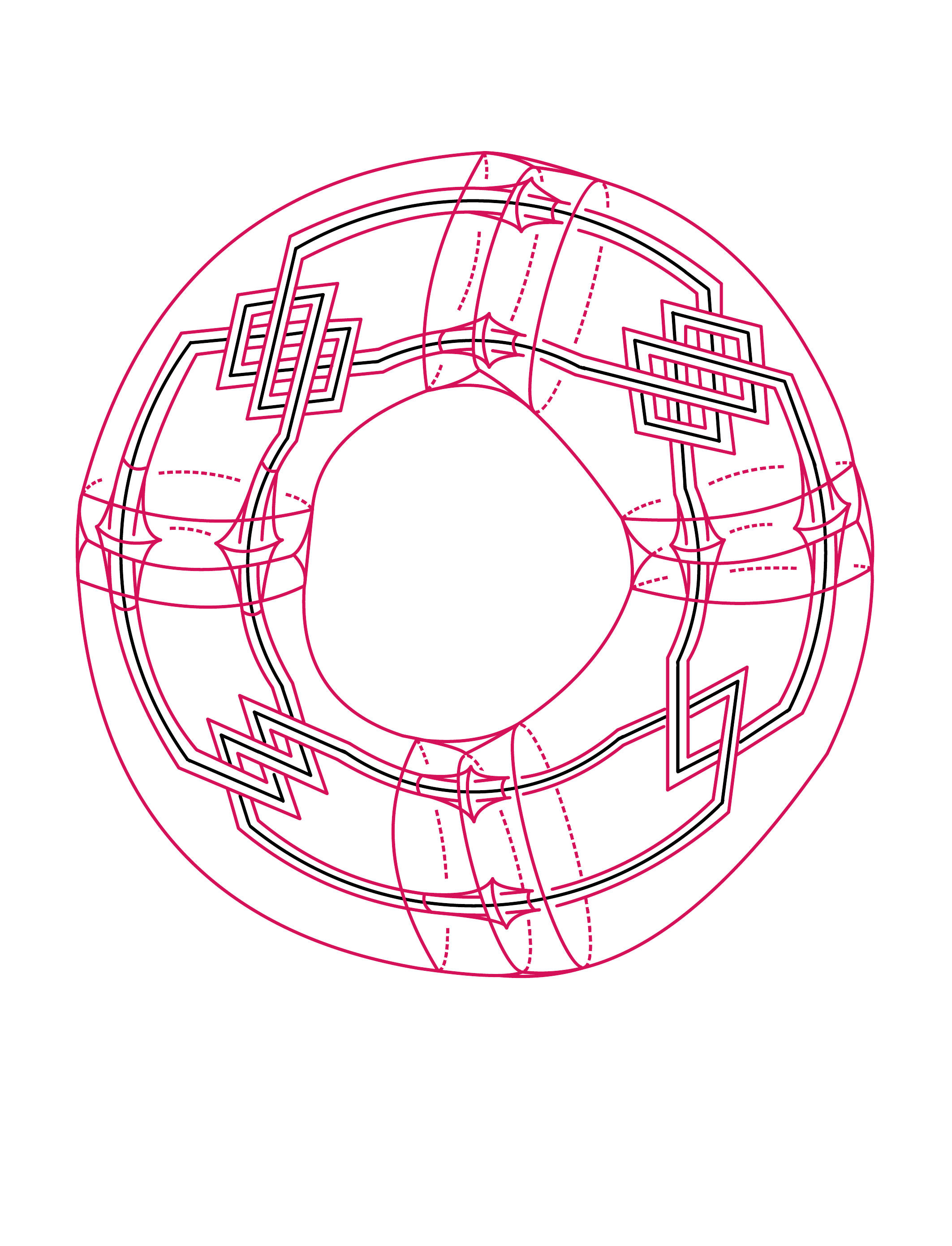}
		\caption{The branched surface $\mathcal{B}$ for $K(\frac{1}{3}, \frac{1}{3}, \frac{2}{3}, \frac{1}{2})$.}
		\label{CarrierBranch}
\end{figure}

\begin{Def} A \textit{switch component} of the branched surface $\mathcal{B}$ is the union of two horizontal boundary components coming from the tubes, which are branched along a singular locus which is a puncture of the $S_C$ horizontal boundary component. See Figure~\ref{Switch}. 
\end{Def}

\begin{Def} A \textit{two-way switch component} of the branched surface $\mathcal{B}$ is the union of three horizontal boundary components, all coming from the tubes, which are branched along a singular locus containing three components: a puncture of the $S_C$ horizontal boundary component and two components coming from the intersection of the tubes. See Figure~\ref{TwoWaySwitch}.
\end{Def}

\begin{Def} A \textit{one-way switch component} of the branched surface $\mathcal{B}$ is the union of two horizontal boundary components branched along a singular locus consisting of two components: a puncture of the $S_C$ horizontal boundary component and a component coming from the intersection of the two tubes. See Figure~\ref{OneWaySwitch}.
\end{Def}

\begin{figure}[h!]
\centering
	\begin{minipage}[b]{0.32\linewidth}
		\centering
		\includegraphics[viewport = 0 210 640 620, scale = 0.24, clip]{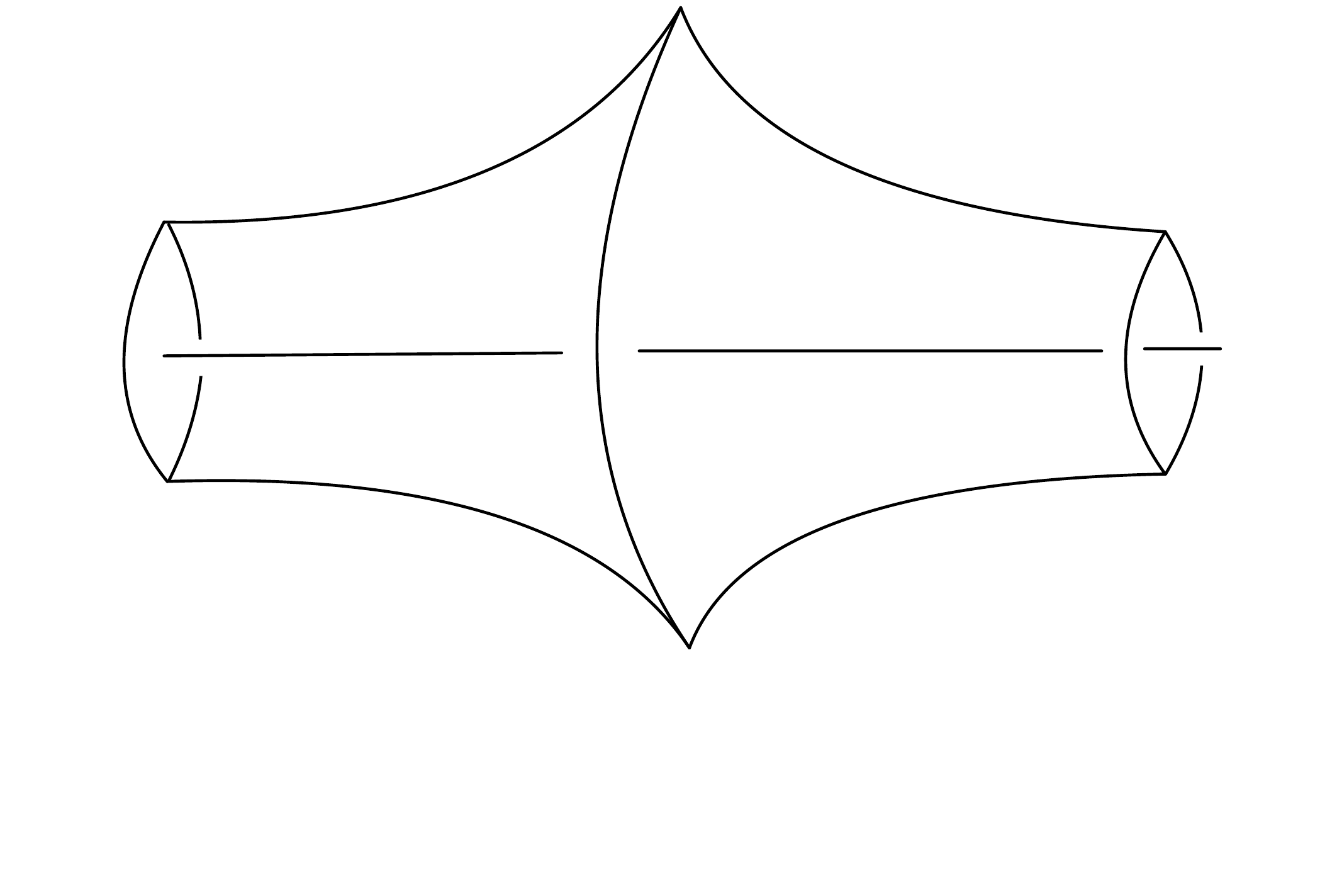}
		\caption{Switch component of $\mathcal{B}$.}
		\label{Switch}
	\end{minipage}
	\begin{minipage}[b]{0.32\linewidth}
		\centering
		\includegraphics[viewport = 0 210 650 600, scale = 0.24, clip]{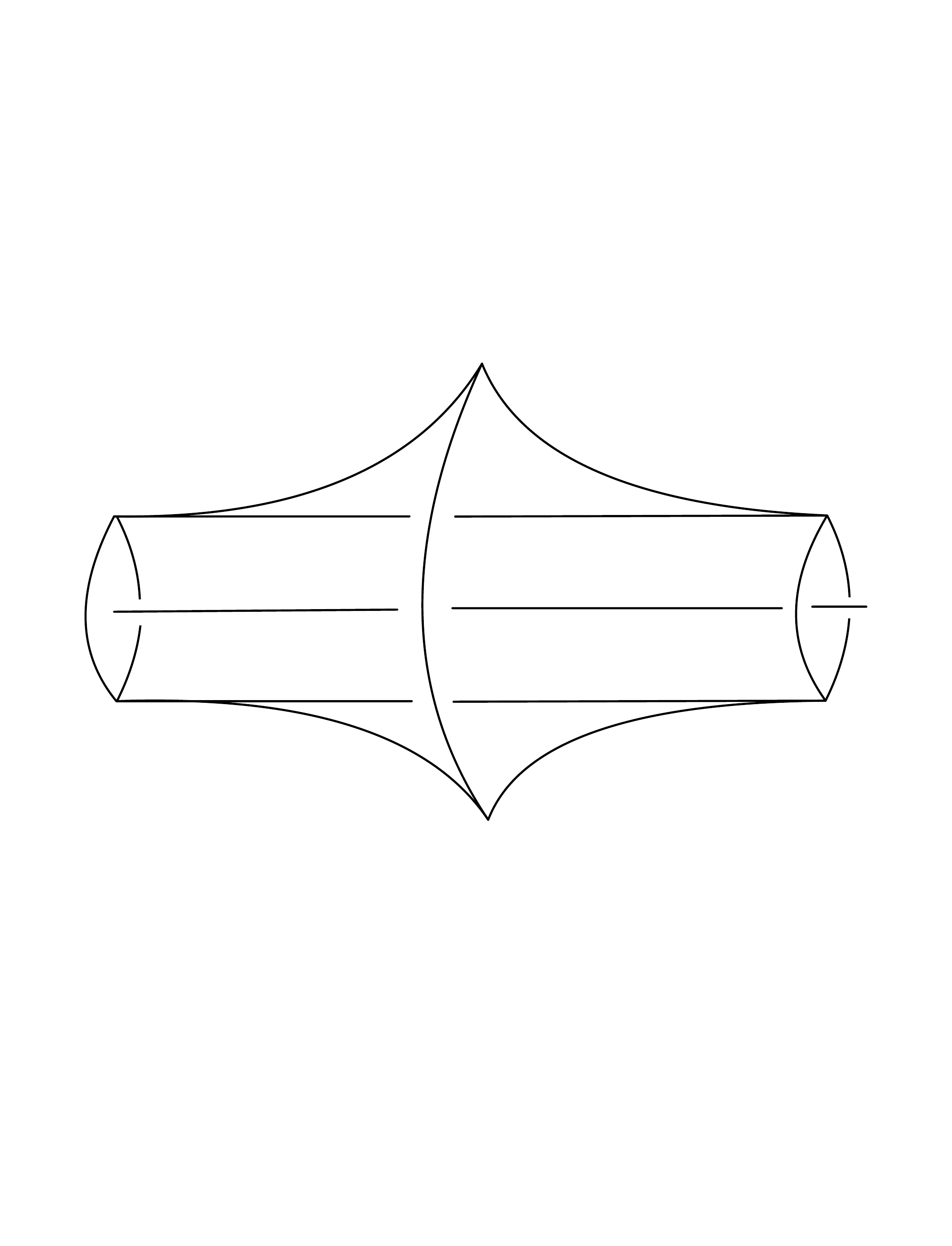}
		\caption{Two-way switch component of $\mathcal{B}$.}
		\label{TwoWaySwitch}
	\end{minipage}
	\begin{minipage}[b]{0.32\linewidth}
		\centering
		\includegraphics[viewport = 0 210 660 580, scale = 0.24, clip]{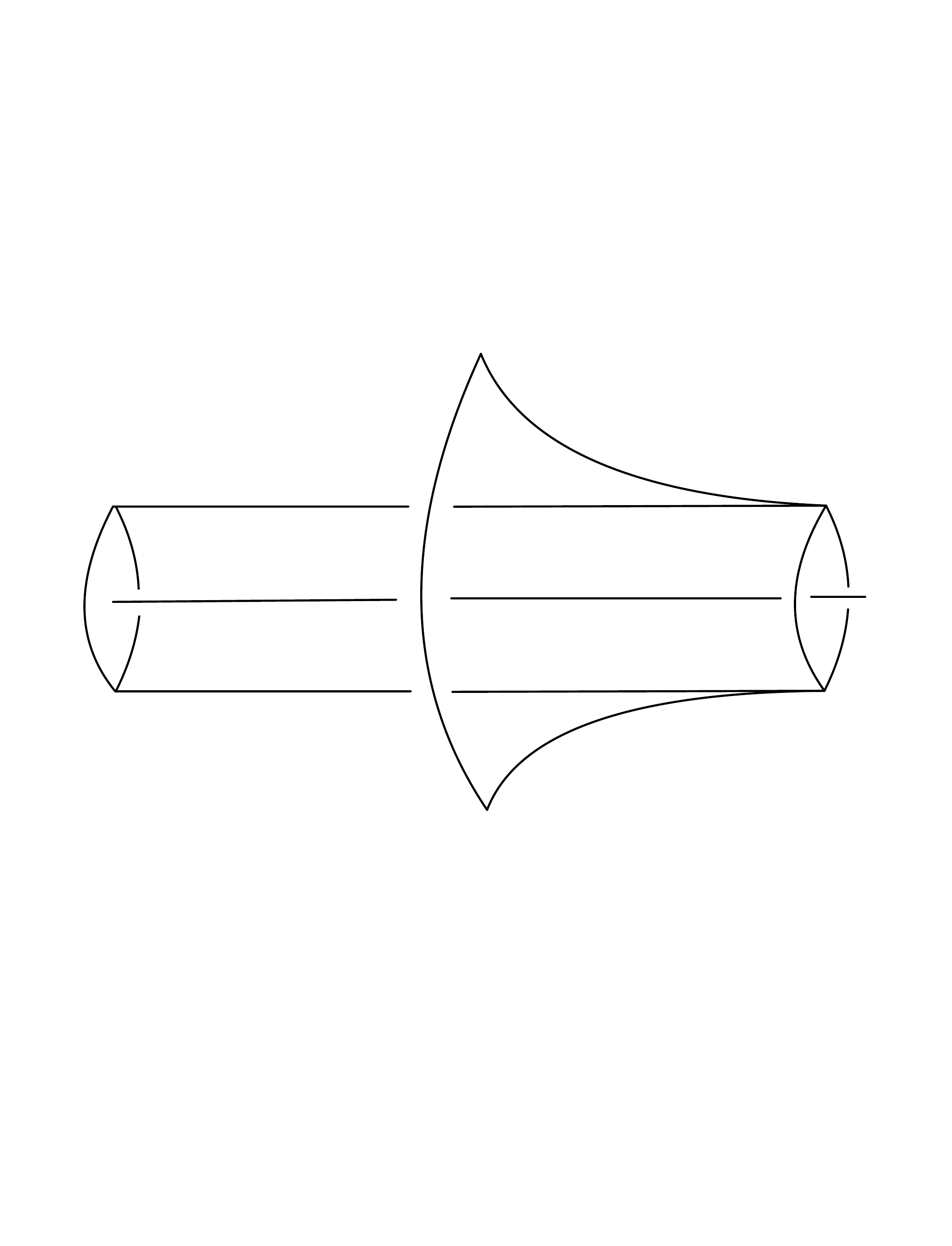}
		\caption{One-way switch component of $\mathcal{B}$.}
		\label{OneWaySwitch}
	\end{minipage}
\end{figure}

%
%
%
\subsubsection{Construction of Incompressible Branched Surface from Punctured Partitioned Chord Diagram}\label{ConstructionOfIncompBS}
We show how to construct an incompressible branched surface $\mathcal{B}_S$ from a punctured partitioned chord diagram. We follow the process given by Floyd and Oertel \cite{FLOYD1984117} yielding an incompressible branched surface from an incompressible surface. Let $D$ be a punctured partitioned chord diagram corresponding to a closed, connected, essential, orientable surface $S$ of genus greater than 2. By Propositions~\ref{UniquelyDetermined} and~\ref{TypesOfChords}, $D$ is uniquely determined by the maximum possible length chord and there are only three types of chords in $D$. Without loss of generality, let $R_1$ be the region containing the base points of the chord of maximum possible length and label the rest of the regions counterclockwise from $R_1$ as $R_2, R_3$, and $R_4$. We first choose the incompressible 4-punctured sphere $S_C$ that we use to build $\mathcal{B}_S$. After this choice of $S_C$, the four regions of $D$ correspond to four arcs of the knot, which we label as $a_1, a_2, a_3$, and $a_4$, containing the punctures of the $S_C$. The chord of maximum possible length represents the innermost tube and has the property that the tube has both ends on the arc $a_1$, corresponding to the region $R_1$, of the knot and passes through every rational tangle exactly once. Thus, this gives a switch component of the branched surface contained on $a_1$ since the tube is innermost and hence only has two horizontal boundary components both coming from the tube. By Proposition~\ref{TypesOfChords}, all chords in which have a base point in either $R_2$ or $R_4$ must pass through $R_3$. This gives that all tubes having one end on either the arc $a_2$ or $a_4$ are carried by a one-way switch component. Lastly, the region $R_3$ consists of chords parallel to the chords of length 1 and portions of chords parallel to the chord of maximum possible length. Hence, we have three contributions of horizontal boundary components coming from the tubes where one tube is innermost amongst the three. This gives that the corresponding arc $a_3$ has a two-way switch component. Thus, the branched surface we obtain from any such punctured partitioned chord diagram $D$ is shown in Figure~\ref{BranchedSurface}. By Proposition~\ref{IncompBranch}, $\mathcal{B}_S$ is incompressible since it does not have any disks of contact.

\begin{figure}[h!]
		\begin{minipage}[b]{0.47\linewidth}
		\centering
		\includegraphics[viewport = 0 180 620 700, scale = 0.36, clip]{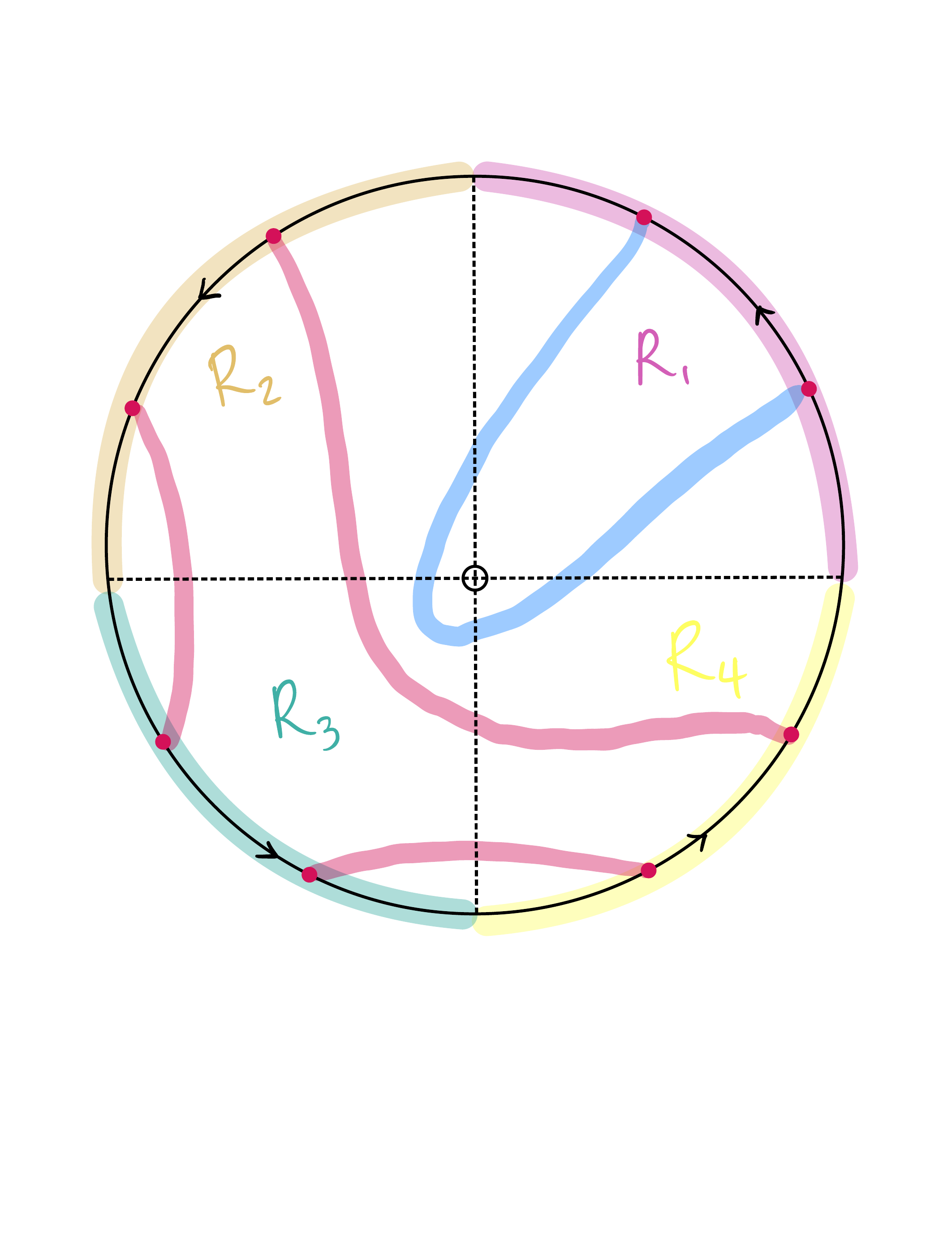}
		\end{minipage}
		\begin{minipage}[b]{0.47\linewidth}
		\centering
		\includegraphics[viewport = 0 150 620 700, scale = 0.36, clip]{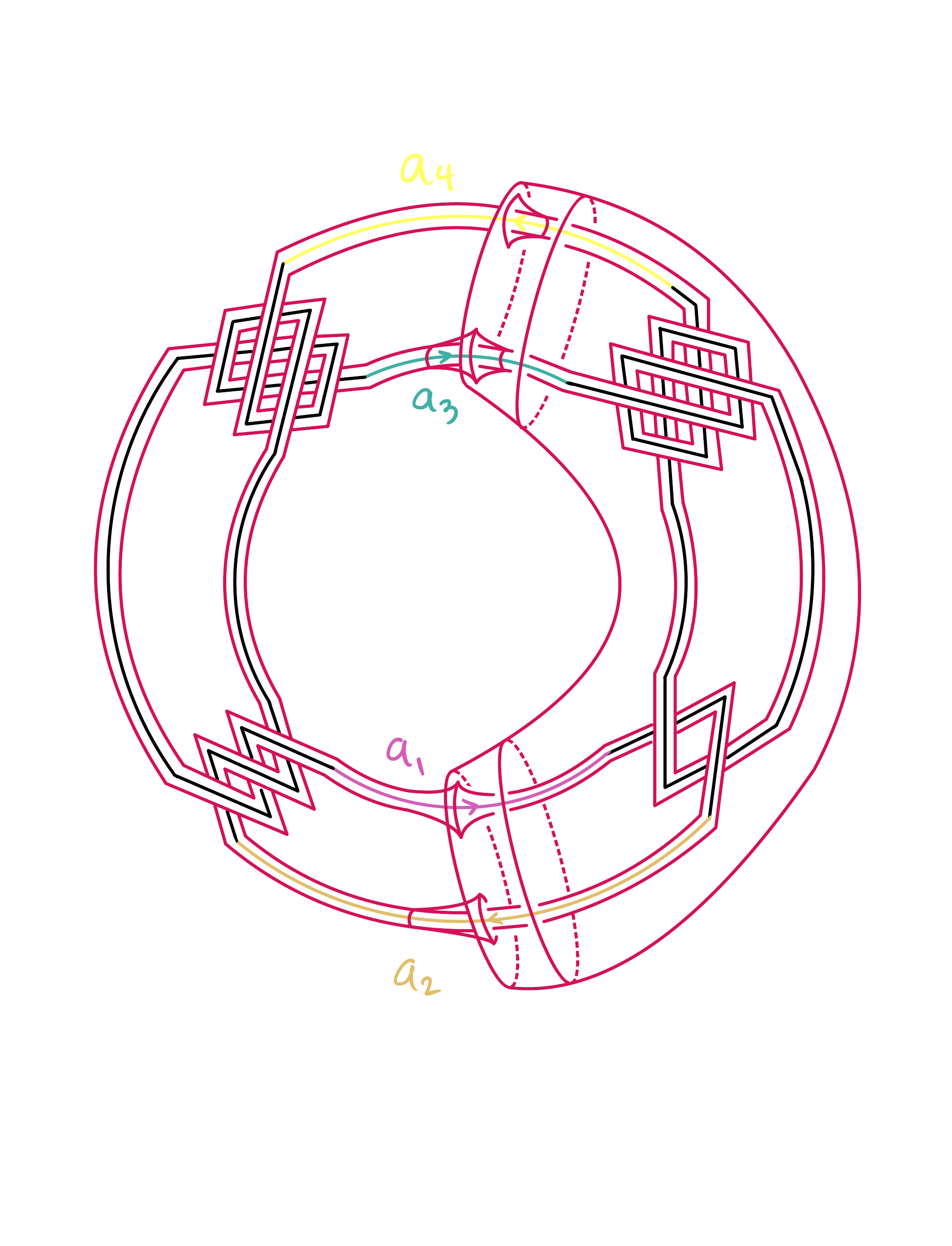}
		\end{minipage}
		\caption{The corresponding branched surface $\mathcal{B}_S$ of the punctured partitioned chord diagram.}
		\label{BranchedSurface}
\end{figure}

Observe that the construction of $\mathcal{B}_S$ is independent of the specific rational tangles, since the regions of the punctured partitioned chord diagrams only depend on which arcs are adjacent by traversing the knot. Moreover, $\mathcal{B}_S$ carries all surfaces which have the chord of maximum possible length in the same region and have the same choice for the 4-punctured sphere $S_C$.

\begin{Prop}\label{RegionsOfBS}
	Let $S$ be a closed, connected, essential, orientable surface of genus greater than 2. Then, none of the complementary regions of the corresponding incompressible branched surface $\mathcal{B}_S$ are product regions. In particular, any surface carried by $\mathcal{B}_S$ is not isotopic to any other surface also carried by $\mathcal{B}_S$. 
\end{Prop}

	\begin{proof}
		By the construction of $\mathcal{B}_S$, it carries all surfaces which are tubing descriptions of surfaces utilizing the same choice of 4-punctured sphere $S_C$ and having the chord of maximum possible length in the same region. We show that the four complementary regions of $\mathcal{B}_S$ are not product regions. It follows from Theorem~\ref{Oertel4} that there are no isotopies between the surfaces that $\mathcal{B}_S$ fully carries.

		The region which is a cusp neighborhood of the knot $K$ is bounded by of one horizontal boundary component and one vertical boundary component, where each component in an annulus, whose union is a torus. However, since there is only one horizontal boundary component, then this is not a product region.

		The region bounded by a two-way switch component is homeomorphic to a solid torus whose boundary consists of three vertical annuli and three horizontal annuli. In particular, it contains more than two horizontal components. Thus, this region is not a product region.

		The region inside of the $S_C$, with respect to the projection onto the sphere, and outside the horizontal tube has only one horizontal boundary component and one vertical boundary component. So, it is not a product region.

		The final region is the region outside of the $S_C$, the horizontal tubes, and two vertical annuli. The boundary of this region is homeomorphic to a genus 2 surface with two vertical annuli which deformation retract onto non-separating curves of the surface. Thus, the boundary consists of exactly one horizontal component and two vertical components. Hence, it is not a product region, which completes the proof.
	\end{proof}

Proposition~\ref{RegionsOfBS} shows that two surfaces obtained by tubing the same choice of $S_C$ which have punctured partitioned chord diagrams with the maximum possible length chord in the same region, are distinct. We now prove the case that punctured partitioned chord diagrams having the maximum length chord in the different regions are distinct. Proposition~\ref{Waldhausen} will be used extensively in analyzing the regions bounded by two surfaces in the complement of the knot. We analyze such regions and show each are not pocket nor product regions, as in the statement of Proposition~\ref{Waldhausen}. In light of the proposition, we will define such regions contained in the punctured partitioned chord diagram containing two sub-diagrams each corresponding to closed, connected, essential, orientable  surface of genus greater than 2. The following heavily utilizes the geometric description of punctured partitioned chord diagrams.

Let $D_1$ and $D_2$ be two punctured partitioned chord diagrams, each giving rise to genus $g$ surfaces, for $g$ greater than 2, $S_1$ and $S_2$, respectively. We may represent both on a single punctured partitioned chord diagram $D$ (see Figure~\ref{SinglePPCD}). Since $D$ represents a cusp neighborhood of the knot, we will analyze each complementary region of the chords of $D$ and show that none are product or pocket regions. By Proposition~\ref{Waldhausen}, we may then conclude that $S_1$ and $S_2$ are non-isotopic surfaces. We first define possible regions of $D$ and show in Lemma~\ref{RegionsOfD} that $D$ can be made to have only these regions.

\begin{figure}[h!]
		\centering
		\includegraphics[viewport = 0 180 620 690, scale = 0.4, clip]{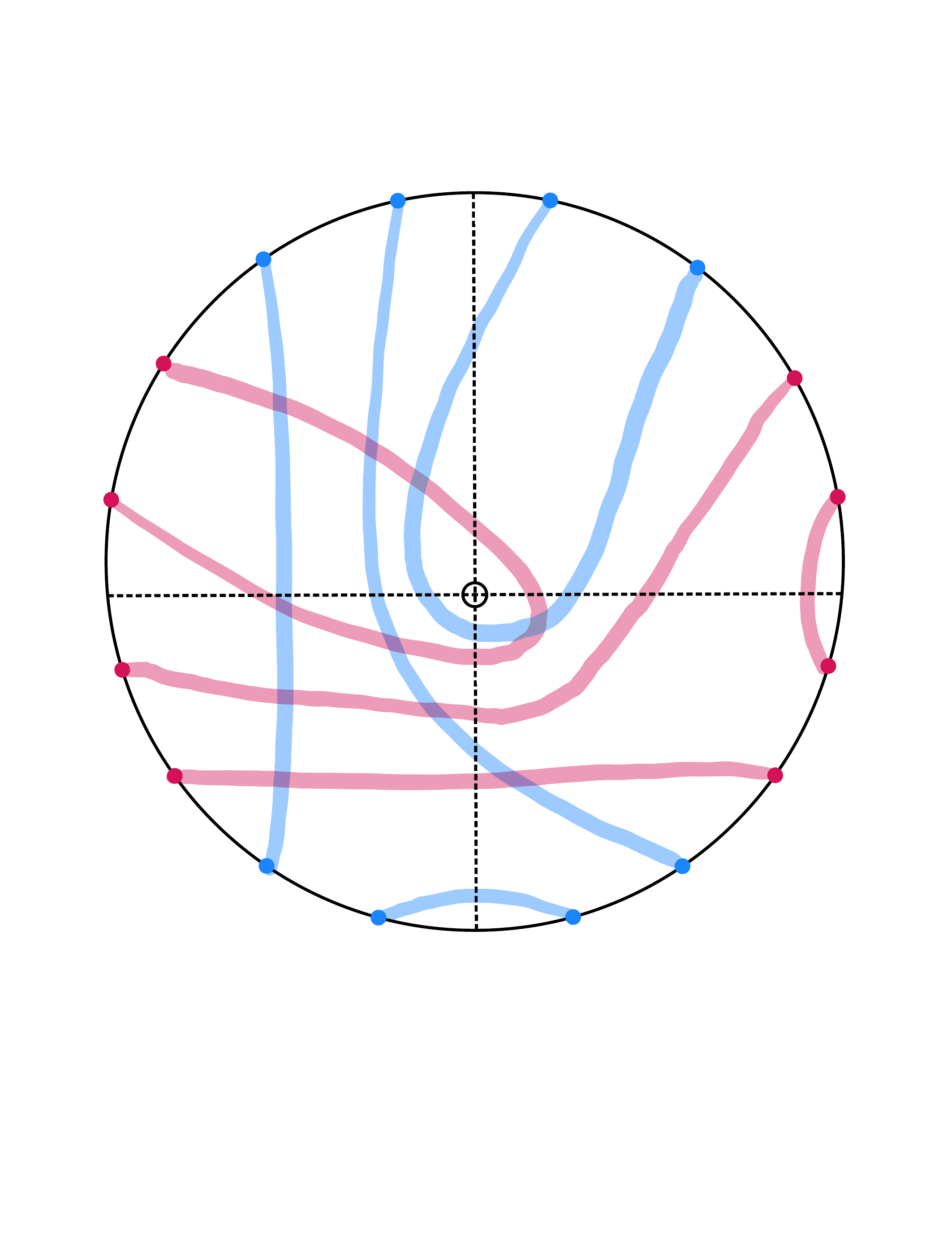}
		\caption{An example of $D$ with sub-diagrams $D_1$ and $D_2$ drawn in red and blue, respectively.}
		\label{SinglePPCD}
\end{figure}

\begin{Def} Let $D$ be a punctured partitioned chord diagram, containing sub-diagrams $D_1$ and $D_2$, as above. A complementary region of $D$ is defined to be one of the following types:
	\begin{enumerate}[leftmargin = 2.54cm, label = \textit{Type \Roman*}:]
		\item A region bounded by two intersecting segments of chords, one from $D_1$ and one from $D_2$ which contains the puncture of $D$.
		\item A region bounded by $2n$ intersecting segments of chords, for $n \in \mathbb{N}_{\geq 2}$, with $n$ segments from chords of $D_1$ and $n$ segments from chords of $D_2$, in an alternating fashion.
		\item A region bounded by 4 segments: a segment of the boundary circle, two segments from chords of $D_i$ and one segment from a chord of $D_j$ where $i, j \in \{1, 2\}$ and $i \neq j$.
		\item A region bounded by $2n + 1$ segments for $n \in \mathbb{N}_{\geq 1}$ with $n$ segments from chords of $D_1$, $n$ segments from chords of $D_2$, and one segment of the boundary circle, in an alternating fashion. Moreover, the segment of the boundary circle is adjacent to a segment of a chord of $D_1$ and a segment of a chord of $D_2$.
		\item A region bounded by 6 segments: two segments of the boundary circle, two segments from chords and a chord of length 1 in $D$ all coming from $D_i$, and one segment from a chord of $D_j$ where $i, j \in \{1, 2\}$ and $i \neq j$.
	\end{enumerate}
 See Figure~\ref{TypesOfRegions}.
\end{Def}

\begin{figure}[h!]
\centering
	\begin{minipage}[b]{0.19\linewidth}
		\centering
		\includegraphics[viewport = 20 120 710 650, scale = 0.16, clip]{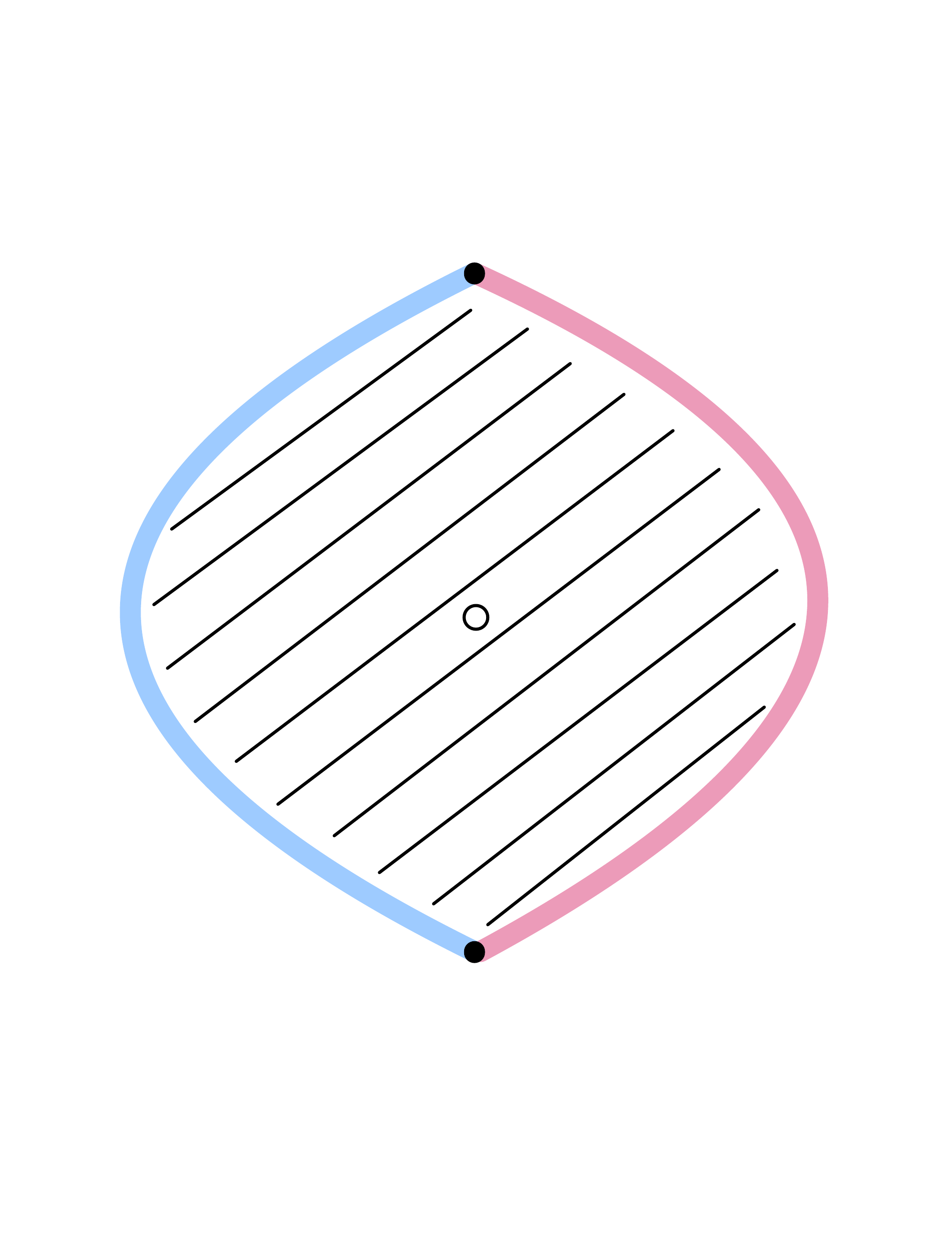}
	\end{minipage}
	\begin{minipage}[b]{0.19\linewidth}
		\centering
		\includegraphics[viewport = 20 120 610 650, scale = 0.16, clip]{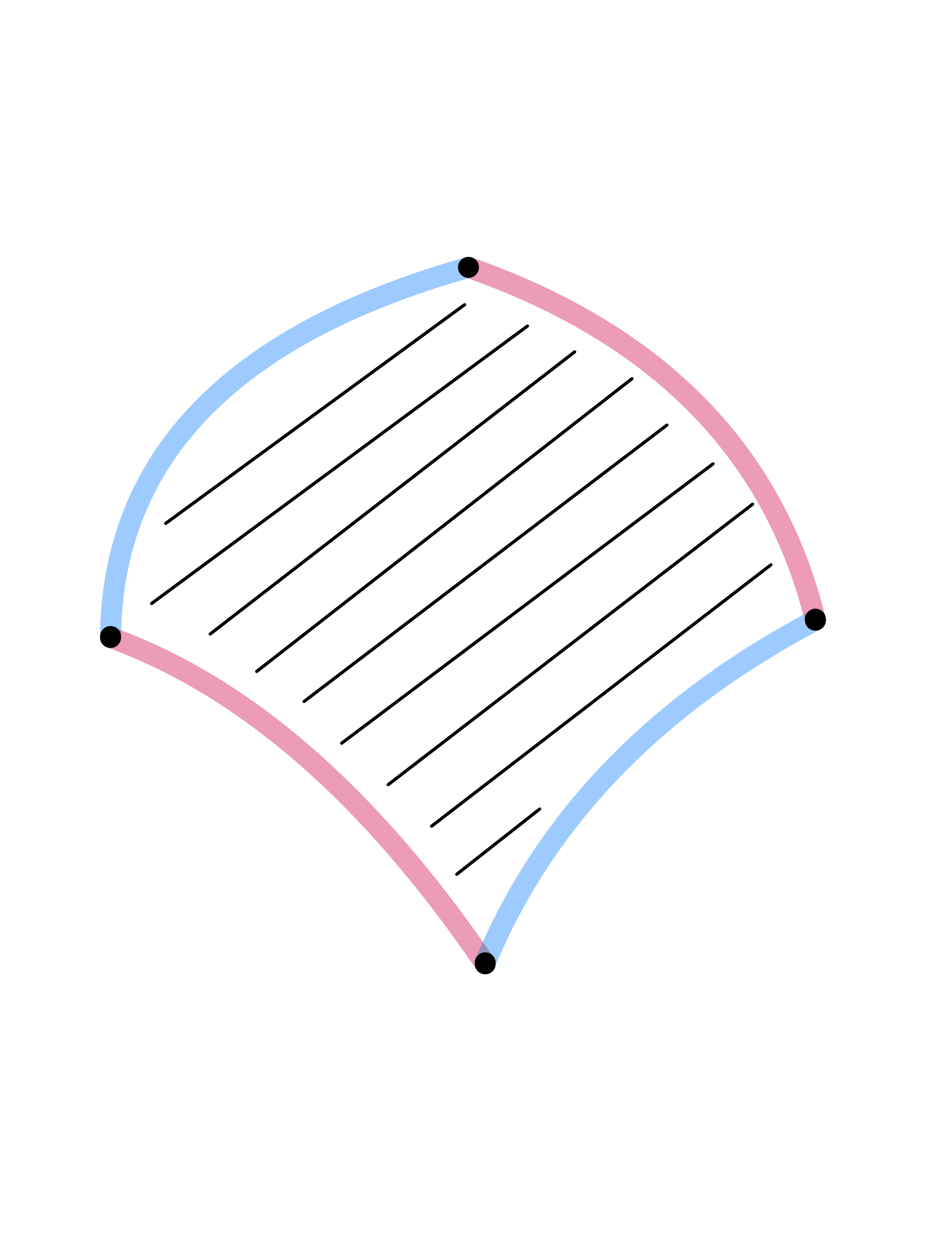}
	\end{minipage}
	\begin{minipage}[b]{0.19\linewidth}
		\centering
		\includegraphics[viewport = 25 120 610 670, scale = 0.16, clip]{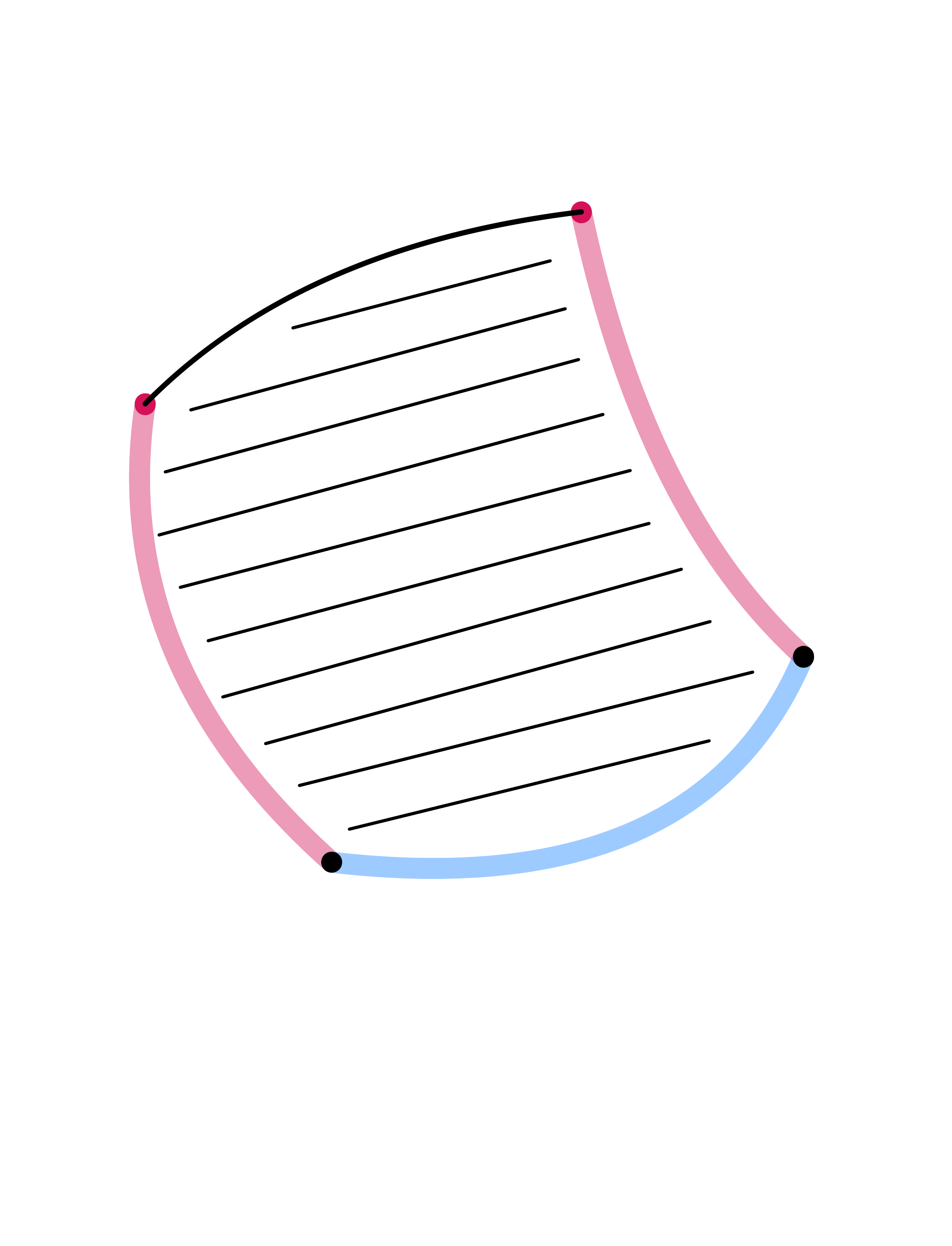}
	\end{minipage}
	\begin{minipage}[b]{0.19\linewidth}
		\centering
		\includegraphics[viewport = 40 150 610 740, scale = 0.16, clip]{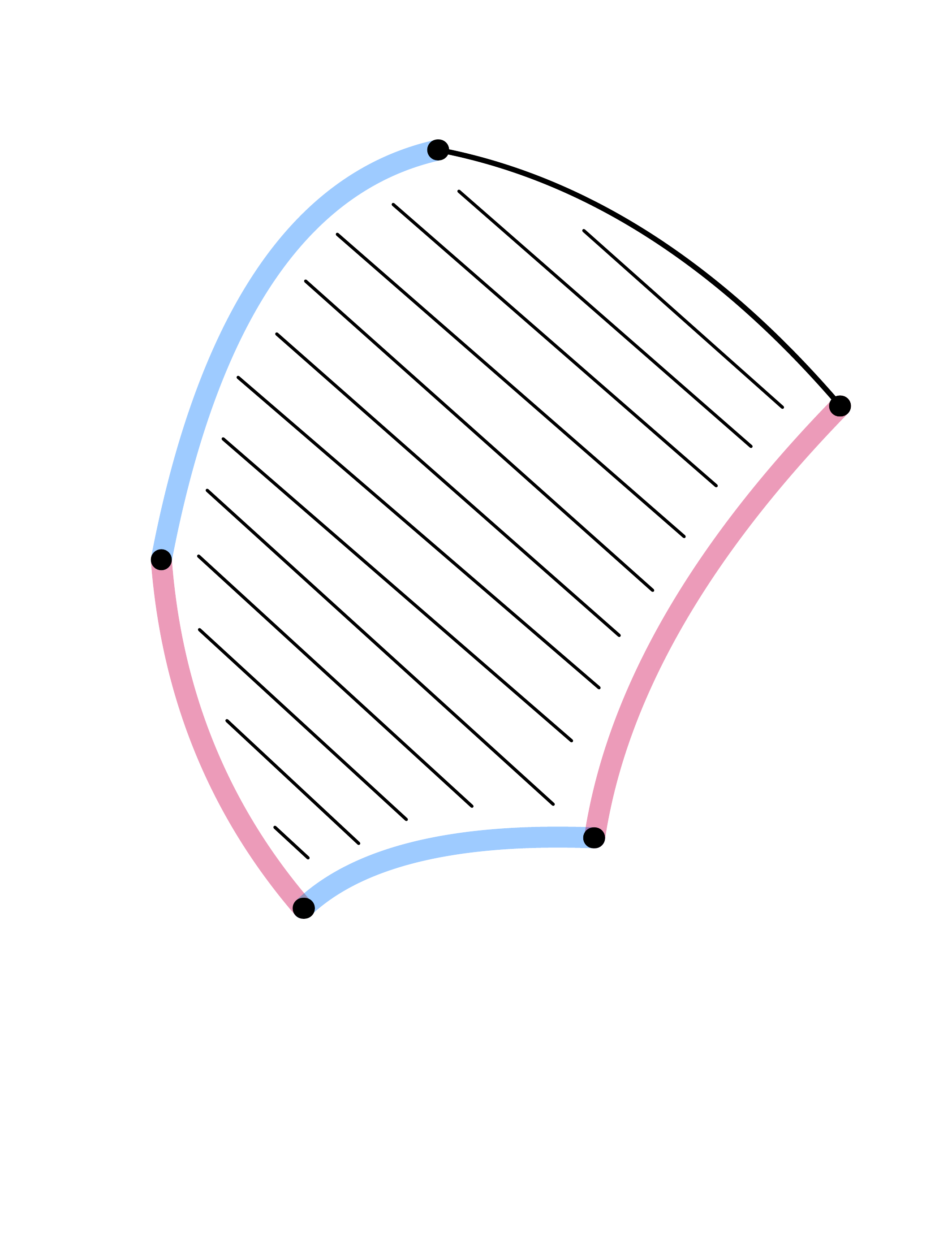}
	\end{minipage}
	\begin{minipage}[b]{0.19\linewidth}
		\centering
		\includegraphics[viewport = 20 120 610 680, scale = 0.16, clip]{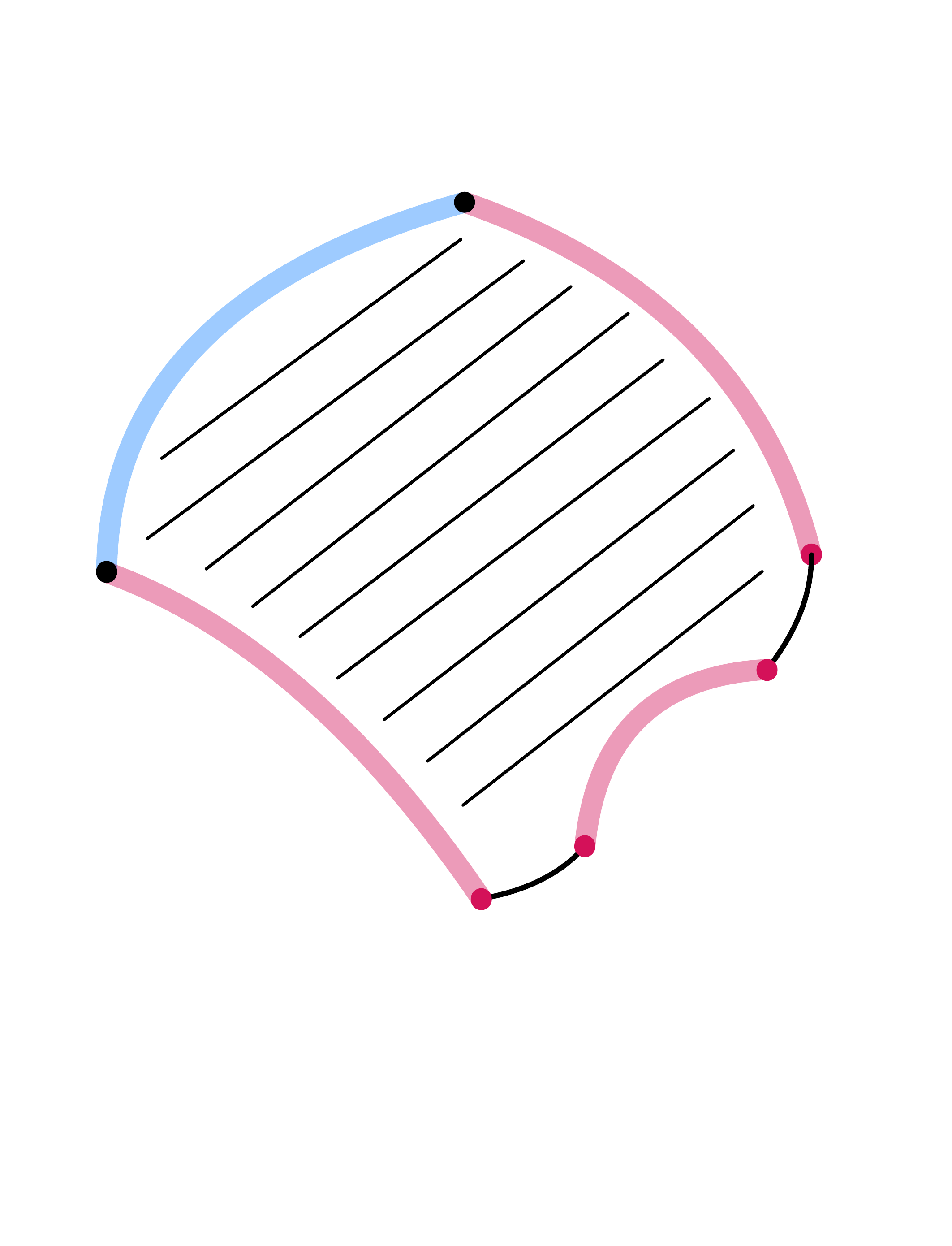}
	\end{minipage}
	\caption{Left to right: Examples of type \textit{I}, type \textit{II}, type \textit{III}, type \textit{IV} and type \textit{V} regions.}
	\label{TypesOfRegions}
\end{figure}

Observe that the segments in each type of region may belong to the same chord in the respective sub-diagram. We now prove a few properties of $D$.

\begin{Lemma}\label{PropertiesOfD}
	Let $D_1$ and $D_2$ be two punctured partitioned chord diagrams, each giving rise to surfaces $S_1$ and $S_2$, respectively. Moreover, assume that both $S_1$ and $S_2$ have genus greater than 2. Consider the superimposed union of $D_1$ and $D_2$ and call it $D$. Then, after an isotopy of the chords, $D$ has the following properties:
	\begin{enumerate}
		\item Chords of $D_1$ only intersect chords of $D_2$.
		\item The base points of $D$ on the boundary circle can be arranged in a cyclic way so that $2(g - 1)$ consecutive base points, with respect to some orientation the boundary circle, are from $D_i$, then $2(g - 1)$ consecutive base points from $D_j$, with $i,j$ alternating between $1$ and $2$.
		\item The intersection of any two chords is 4-valent and transverse.
		\item Chords intersect in an alternating fashion.
	\end{enumerate}
\end{Lemma}
	\begin{proof}
		Property \textit{1} follows immediately from the definition of punctured partitioned chord diagrams. We may take an isotopy of the base points of $D_i$, preserving the respective order and adjacency of the base points, to obtain the cyclic arrangement stated in property \textit{2}. Moreover, this isotopy can be made so that chords of $D_i$ do not intersect with any other chords of $D_i$, preserving property \textit{1} throughout the isotopy. Every intersection is 4-valent since chords intersect in their interiors and any tangential intersection in the interior can be resolved by an isotopy of the chord which does not result in any additional intersections with any other chords. Thus, giving property \textit{3}. Finally, property \textit{4} follows from property \textit{1} since intersections of any two chords implies that the chords do not come from the same sub-diagram.
	\end{proof}

The existence of a region of type \textit{IV} for $n \geq 2$ is crucial to proving distinctness. We prove the existence of such a region with the following lemmas and prove all regions are of types $I, II, III, IV,$ or $V$.

\begin{Lemma}\label{ChordsOfD}
	Let $D$ be as in Lemma~\ref{PropertiesOfD}. Let $c_1$ and $c_2$ be the two maximum possible length chords of $D$, coming from $D_1$ and $D_2$, respectively. Then, there exists two chords $c_1'$ and $c_1''$ parallel to $c_1$ which have one pair of base points adjacent to each other and the other pair of base points contain $2(g - 1)$ base points of $D_2$.
	
	Moreover, there are also two chords $c_2'$ and $c_2''$ parallel to $c_2$ which are adjacent to each other at one end and contain $2(g - 1)$ base points of $D_1$ in between $c_2'$ and $c_2''$.
\end{Lemma}
	\begin{proof}
		By Proposition~\ref{UniquelyDetermined}, there exists at least two chords $c_1'$ and $c_1''$ parallel to $c_1$ which are adjacent to each other. Note that we allow $c_1'$ or $c_1''$ to be $c_1$. Namely, the chords $c_1'$ and $c_1''$ must exist since $c_1$ and a $\beta$ type chord, as in Proposition~\ref{UniquelyDetermined}, exists. So, one of $c_1'$ or $c_1''$ may be $c_1$ itself. Moreover, $c_1$ has both its base points in the same region and there are $g - 1$ chords parallel to $c_1$ by Proposition~\ref{UniquelyDetermined}. Observe that there are $g - 1$ base points of $D_1$ in each region, and $c_1$ has adjacent base points. Then, using property \textit{2} of Lemma~\ref{PropertiesOfD}, $c_1'$ and $c_1''$ contain $2(g - 1)$ base points of $D_2$ in between a pair of base points. This proves the existence of $c_1'$ and $c_1''$ with such properties. Similarly, we may apply the same argument to obtain two chords $c_2'$ and $c_2''$ parallel to $c_2$ which are adjacent to each other and contain $2(g - 1)$ base points of $D_1$ in between $c_2'$ and $c_2''$. 
	\end{proof}

\begin{Lemma}\label{RegionsOfD}
	Let $D$ be as in Lemma~\ref{PropertiesOfD}. Then, every complementary region of $D$ is one of types \textit{I, II, III, IV}, or \textit{V}.
\end{Lemma}
	\begin{proof}
		Observe that $D$ contains exactly two chords of maximum possible length $c_1$ and $c_2$ coming from the chords of maximum possible length in $D_1$ and $D_2$, respectively. The puncture of $D$ is contained in exactly one region (see e.g. Figure~\ref{SinglePPCD}). Note that $c_1$ and a segment of the boundary circle in $D_1$ bounds the region containing the puncture. Similarly for $c_2$ in $D_2$. Thus, this implies that the puncture of $D$ lies in exactly one region bounded by segments of $c_1$ and $c_2$. In particular, this is a type \textit{I} region.
		
		Regions of $D$ are bounded by segments of chords of $D_1$, $D_2$, and segments of the boundary circle. By Lemma~\ref{PropertiesOfD}, all regions are bounded by alternating segments and boundary segments. We eliminate the region bounded by exactly one segment of a chord of $D_1$ and one segment of a chord of $D_2$ which does not contain the puncture. In this case, we may isotope the chords to have no intersections using the disk bounded by the segments. 
		
		We eliminate regions bounded by more than 2 segments of the boundary circle. By Proposition~\ref{TypesOfChords}, $D_1$ and $D_2$ have only one region which has more than 2 segments from the boundary circle. The region in $D_i$ is bounded by exactly three chords, each with the property that it is parallel to exactly one of $c_i, c_i',$ and $c_i''$. Thus, the region contains base points opposite of each other on the boundary circle. However, by Lemma~\ref{ChordsOfD} there exists a chord of $D_j$ which separates these base points. This chord partitions the region into two regions, each of which has at most 2 segments from the boundary circle. Thus, the only possible types of regions are of one of types \textit{I, II, III, IV}, or \textit{V}.
	\end{proof}

\begin{Lemma}\label{TypeIV}
	Let $D$ be as in Lemma~\ref{PropertiesOfD}. Then, there exists a region of type \textit{IV} for $n \geq 2$.
\end{Lemma}
	
	\begin{proof}
		By Lemma~\ref{ChordsOfD}, we have the existence of chords $c_1'$, $c_1''$, and $c_2'$ $c_2''$. In particular, this implies that one of $c_1'$ or $c_1''$ partitions the base points of $D_2$ into two sets each containing $2(g - 1)$ base points. Similarly, one of $c_2'$ or $c_2''$ partitions the base points of $D_1$ into two sets, based on the complement of such a chord, each containing $2(g - 1)$ base points. Without loss of generality, assume that $c_1''$ and $c_2''$ have this property (see e.g. Figrue~\ref{ExistenceOfTypeIV}).

\begin{figure}[h!]
		\centering
		\includegraphics[viewport = 0 230 620 730, scale = 0.4, clip]{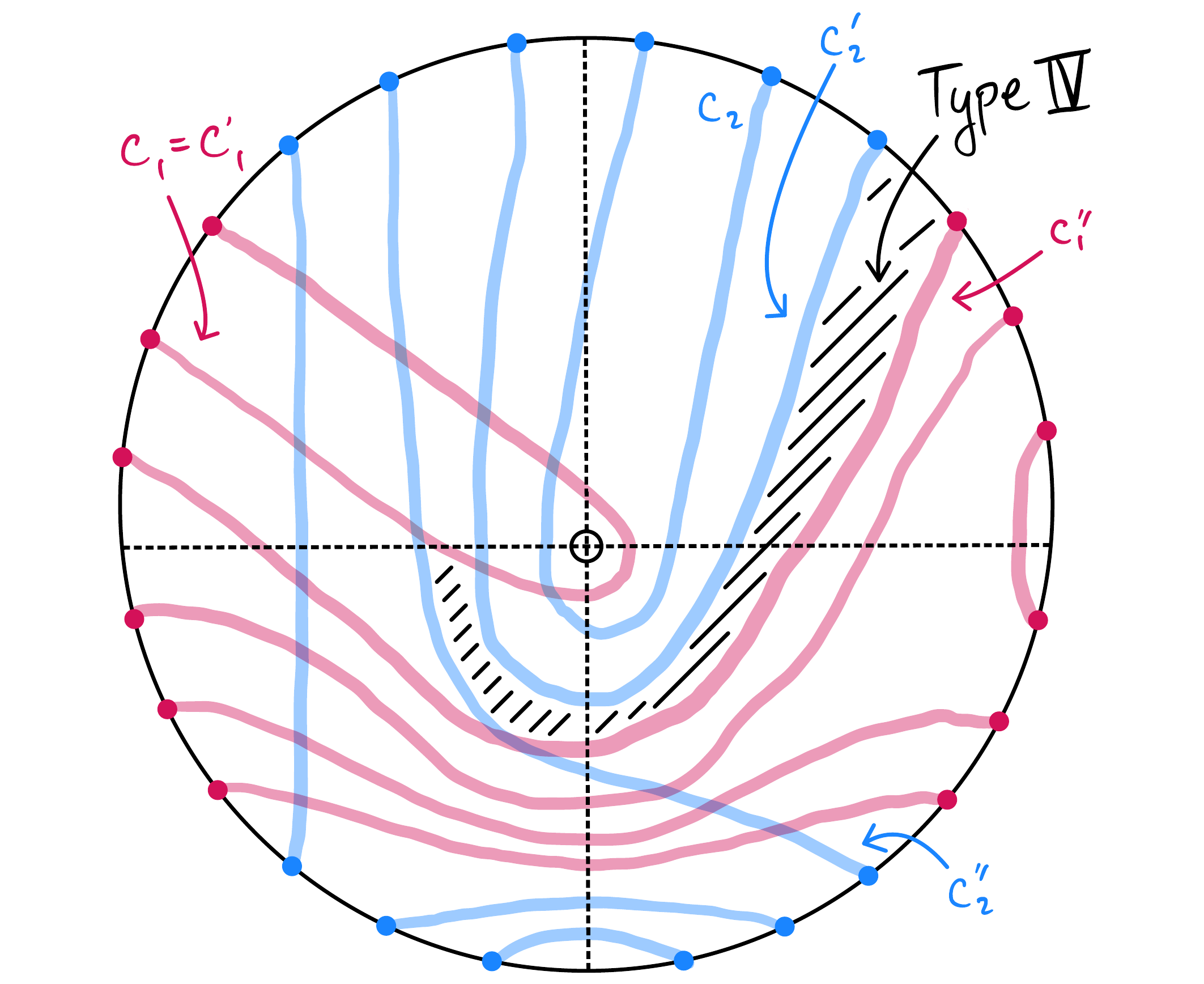}
		\caption{An example of a type \textit{IV} in $D$ with $D_1$ and $D_2$ drawn in red and blue, respectively.}
		\label{ExistenceOfTypeIV}
\end{figure}

		Thus, $c_1''$ and $c_2''$ intersect transversely exactly once. Moreover, since $c_1''$ partitions the base points of $D_2$, then $c_2'$ has its base points on one side of $c_1''$. In particular, $c_2'$ has exactly one base point adjacent to $c_1''$ and $c_2'$ does not intersect $c_1''$. Now, $c_2'$ must intersect $c_1'$ since $c_2'$ is between $c_1'$ and $c_1''$ which are both parallel to $c_1$. Thus, we have a region bounded by the following: a segment of the boundary circle between the adjacent base points of $c_2'$ and $c_1''$, a segment of $c_1''$, a segment of $c_2''$, a segment of $c_1'$, and a segment of $c_2'$. Therefore, we have the existence of a region of type \textit{IV}.
	\end{proof}

We now prove that for any two punctured partitioned chord diagrams corresponding to surfaces of genus greater than 2 and having their maximum possible length chords in distinct regions give rise to distinct, non-isotopic surfaces. This completes the proof of the case that for punctured partitioned chord diagrams corresponding to surfaces of genus greater than 2, these surfaces are non-isotopic.

\begin{Lemma}\label{Pockets}
	Pocket regions of two surfaces $S_1$ and $S_2$ are bounded by exactly two components: one component from $S_1$ and the other from $S_2$, each of which are connected.
\end{Lemma}
	\begin{proof}
		By Proposition~\ref{Waldhausen}, pocket regions are regions homeomorphic to a product $H \times I$, for some surface $H$. Observe that pocket regions are bounded by one component of $S_1$ and one component of $S_2$ intersecting in a single simple closed curve. So, we have that $H \times \{0\}$ is the component of $S_1$ and $\partial(H \times I) - (H \times \{0\})$ is the component of $S_2$. Thus, pocket regions of $S_1$ and $S_2$ are bounded by exactly one component coming from $S_1$ and another component from $S_2$, each of which are connected.
	\end{proof}

\begin{Lemma}\label{TypesIandII}
	Let $D_1$ and $D_2$ be two punctured partitioned chord diagrams with maximum possible length chords $c_1$ and $c_2$, respectively. Let $D$ be the superimposed union of $D_1$ and $D_2$. Suppose that $c_1$ and $c_2$ have their base points in distinct regions of $D$. Then, type \textit{I} and \textit{II} regions are not pocket regions.
\end{Lemma}
	\begin{proof}
	 	We show that the complementary regions of $D$, along with the complementary regions of the $S_C$'s in the tubing description, are neither parallel nor pocket regions. By Lemma~\ref{PropertiesOfD}, the arrangement of the base points corresponds to all of the 4-punctured spheres being nested. Without loss of generality, we may assume that the $S_C$'s of $S_1$ are innermost, in the projection onto the projection sphere, and the $S_C$'s of $S_2$ are outermost (see Figure~\ref{NestedS_Cs}). Thus, with this arrangement, all pocket regions are contained in a collared neighborhood of $K$, represented by $D$. We show that any of the type \textit{I} and \textit{II} regions of $D$ do not correspond to pocket regions.

\begin{figure}[h!]
	\includegraphics[viewport = 0 150 610 690, scale = 0.4, clip]{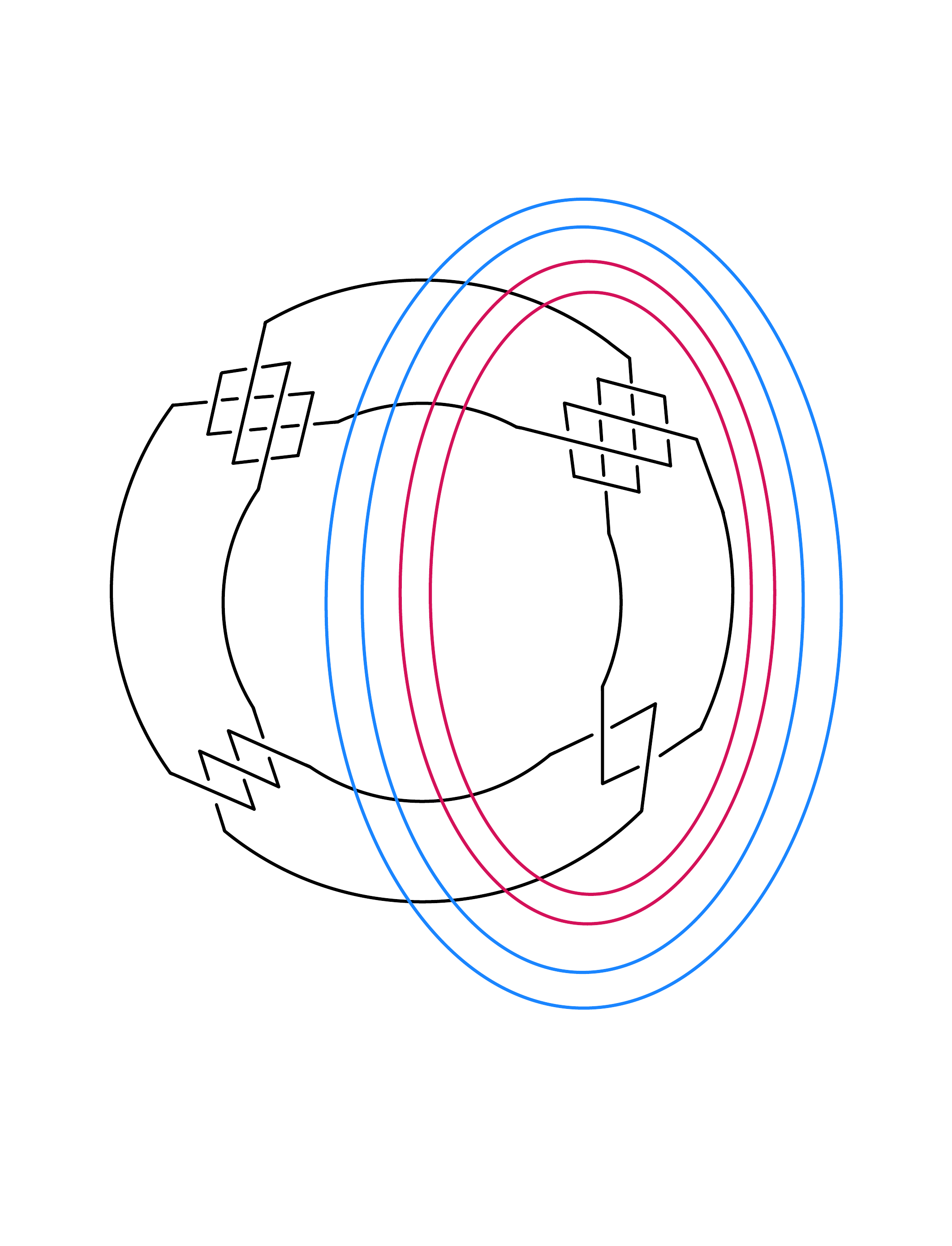}
	\caption{An example of arranging the $S_C$'s of $S_1$ and $S_2$ drawn in red and blue, respectively.}
	\label{NestedS_Cs}
\end{figure}

		There is exactly one type \textit{I} region. This region corresponds to a cusp neighborhood of the knot. This is not a product region, since the components of $S_1$ and $S_2$ are annuli and do not contribute to the portion of the region corresponding to the knot. The regions of type \textit{II} are also not pocket regions. This is because these regions in $D$ correspond to solid tori in the knot complement whose boundary torus has four connected components: two from $S_1$ and two from $S_2$. Thus, by Lemma~\ref{Pockets}, this is not a pocket region.
	\end{proof}

We now prove the remaining cases of regions of types \textit{III, IV} and \textit{V}.

\begin{Prop}\label{NonIsotopicD}
	Let $D_1$ and $D_2$ be two punctured partitioned chord diagrams corresponding to closed, connected, essential, orientable surfaces of genus greater than 2. Let $c_1$ and $c_2$ denote the maximum possible length chords of $D_1$ and $D_2$, respectively. Let $D$ be the superimposed union of $D_1$ and $D_2$. If $c_1$ and $c_2$ have their base points in distinct regions of $D$, then the corresponding surfaces $S_1$ and $S_2$ of $D_1$ and $D_2$, respectively, are non-isotopic.
\end{Prop}
	 \begin{proof}		
		We may assume that $S_1$ and $S_2$ have isotopic 4-punctured spheres $S_C$'s, otherwise $S_1$ and $S_2$ are necessarily non-isotopic by Proposition~\ref{NonisotopicSCs}. By Lemma~\ref{TypesIandII}, regions of types \textit{I} and \textit{II} are not pocket regions. Moreover, the $S_C$'s of $S_1$ are innermost, in the projection onto the projection sphere, and the $S_C$'s of $S_2$ are outermost. Note that regions of types \textit{III} and \textit{IV} contain segments of the boundary circle. Thus, these pocket regions also consist of regions between the $S_C$'s. However, since the $S_C$'s are nested, then these regions consists of regions between the $S_C$'s and regions of types \textit{III, IV}, and \textit{V} of $D$.
		
		Observe that the regions of types \textit{III} and \textit{V} has exactly one horizontal boundary component which is an annulus (see blue component in Figure~\ref{TypesOfRegions}). Any horizontal boundary component of a region has non-positive Euler characteristic since it is a union of tubes and $S_C$'s. Any region which has a component of the boundary circle contains an $S_C$ and hence has negative Euler characteristic. Thus, \textit{III} and \textit{IV} regions are not pocket regions since the boundary components of the region have different Euler characteristics.
		
		Lastly, the regions of type \textit{IV} correspond to the region between an $S_C$ of $S_1$ and an $S_C$ of $S_2$. There is only one such region since all the $S_C$'s are nested with the $S_C$'s of $S_1$ innermost and the $S_C$'s of $S_2$ outermost. By Lemma~\ref{TypeIV}, there exists a type \textit{IV} region with $n \geq 2$. This implies that the region between an $S_C$ of $S_1$ and an $S_C$ of $S_2$ has disconnected components. Thus, by Lemma~\ref{Pockets}, this is not a product region. Applying Proposition~\ref{Waldhausen} implies that $S_1$ and $S_2$ are non-isotopic, which completes the proof. 
	 \end{proof}

Combining Propositions~\ref{RegionsOfBS} and~\ref{NonIsotopicD}, we obtain the results of Theorem~\ref{main} for the case of closed, connected, essential surfaces of genus greater than 2. We finally prove the case of genus 2 surfaces. We first prove the following lemma.

\begin{Lemma}\label{Annulus}
	Let $S$ be an orientable surface which is not the $2$-sphere and $P = S \times I$. Suppose that $A$ is an incompressible annulus in $P$ with one component of $\partial A$ on $P \times \{0\}$ and the other on $P \times \{1\}$, each disjoint from $\partial S \times \{0\} \cup \partial S \times \{1\}$, respectively. Then, $P \setminus A$ is a pocket region.
\end{Lemma}
	\begin{proof}
		By Lemma 3.4 of \cite{MR224099}, $A$ is isotopic to a vertical annulus, constant on $S \times \{0\} \cup \partial S \times I$. Thus, since $A$ is a union of fibers, then $P \setminus A$ is a union of pocket regions. 
	\end{proof}

\begin{Lemma}\label{Genus2Case1}
	Let $D_1$ and $D_2$ be punctured partitioned chord diagrams which correspond to closed, connected, orientable genus 2 surfaces $S_1$ and $S_2$, respectively. Let $D$ be the superimposed union of $D_1$ and $D_2$. If $D_1$ and $D_2$ both contain chords of maximum possible length, $S_1$ and $S_2$ are non-isotopic. 
\end{Lemma}
	\begin{proof}
		We may assume that $S_1$ and $S_2$ have isotopic $S_C$'s, otherwise $S_1$ and $S_2$ are necessarily non-isotopic by Proposition~\ref{NonisotopicSCs}. Suppose we have the case where $D$ contains two chords of maximum possible length, as shown in Figure~\ref{Genus2D1}. Then these two chords bound a type \textit{I} region, which is not a pocket region by Lemma~\ref{TypesIandII}. We consider the complementary region $R$ between the $S_C$'s of $S_1$ and $S_2$. 
		
\begin{figure}[h!]
\centering
	\begin{minipage}[b]{0.42\linewidth}
		\centering
		\includegraphics[viewport = 0 180 710 750, scale = 0.24, clip]{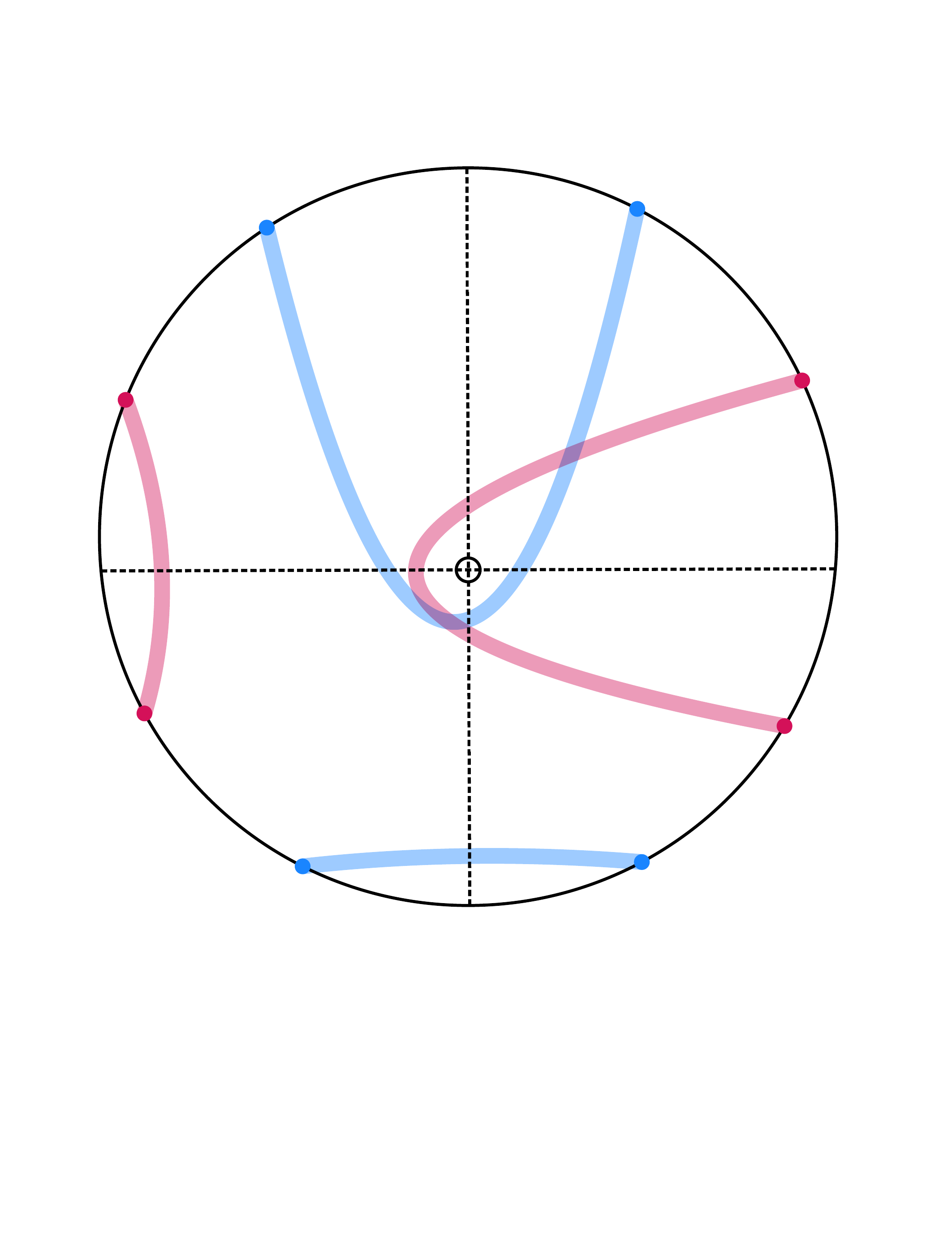}
		\caption{One possibility of $D_1$ and $D_2$, drawn in red and blue, respectively.}
		\label{Genus2D1}
	\end{minipage}
	\begin{minipage}[b]{0.42\linewidth}
		\centering
		\includegraphics[viewport = 0 140 710 750, scale = 0.24, clip]{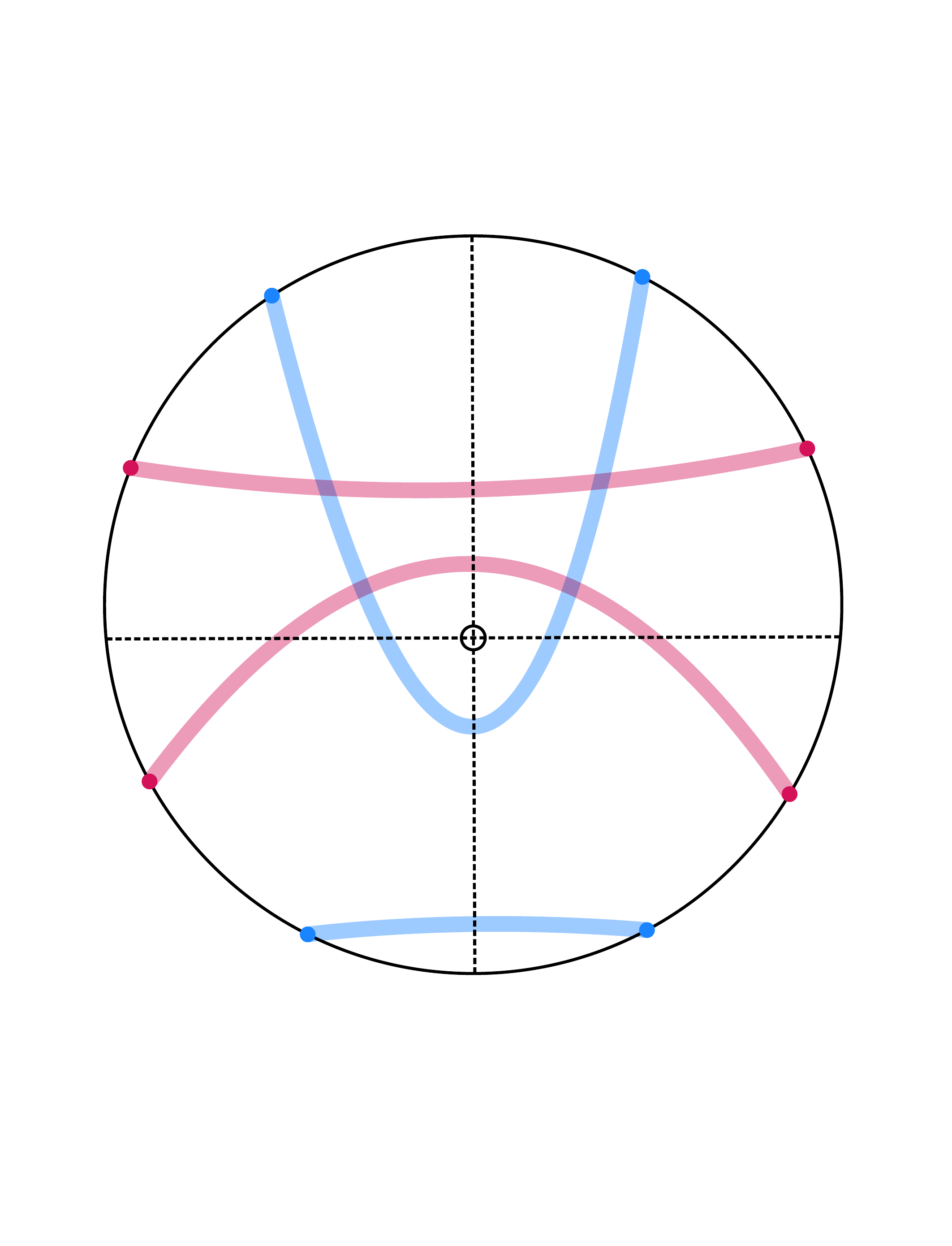}
		\caption{Last possibility of $D_1$ and $D_2$, drawn in red and blue, respectively.}
		\label{Genus2D5}
	\end{minipage}
\end{figure}
		
		Take three arcs with the following properties: one base point of the arc lies on a chord of $D_1$ and the other base point of the arc lies on a chord of $D_2$ where the chords are disjoint and the arc is completely contained in a single region of $D$ and parallel to a segment of the boundary circle (see Figure~\ref{Annulus4}). These three arcs represent incompressible annuli in $R$ since each $S^1$ fiber over every point of the arcs are meridians of $K$. Taking the complement of these three annuli decomposes $R$ such that one component is bounded into two disjoint annuli of $S_1$ and two disjoint annuli of $S_2$. Thus, Lemma~\ref{Pockets} implies that this component is not a pocket region. By Lemma~\ref{Annulus}, since this component is not a pocket region, then $R$ is also not a pocket region. The last region to consider is outside of the $S_C$. However, this is a type \textit{III} region and hence not a pocket, by Proposition~\ref{NonIsotopicD}.

\begin{figure}[h!]
		\centering
		\includegraphics[viewport = 0 170 620 660, scale = 0.36, clip]{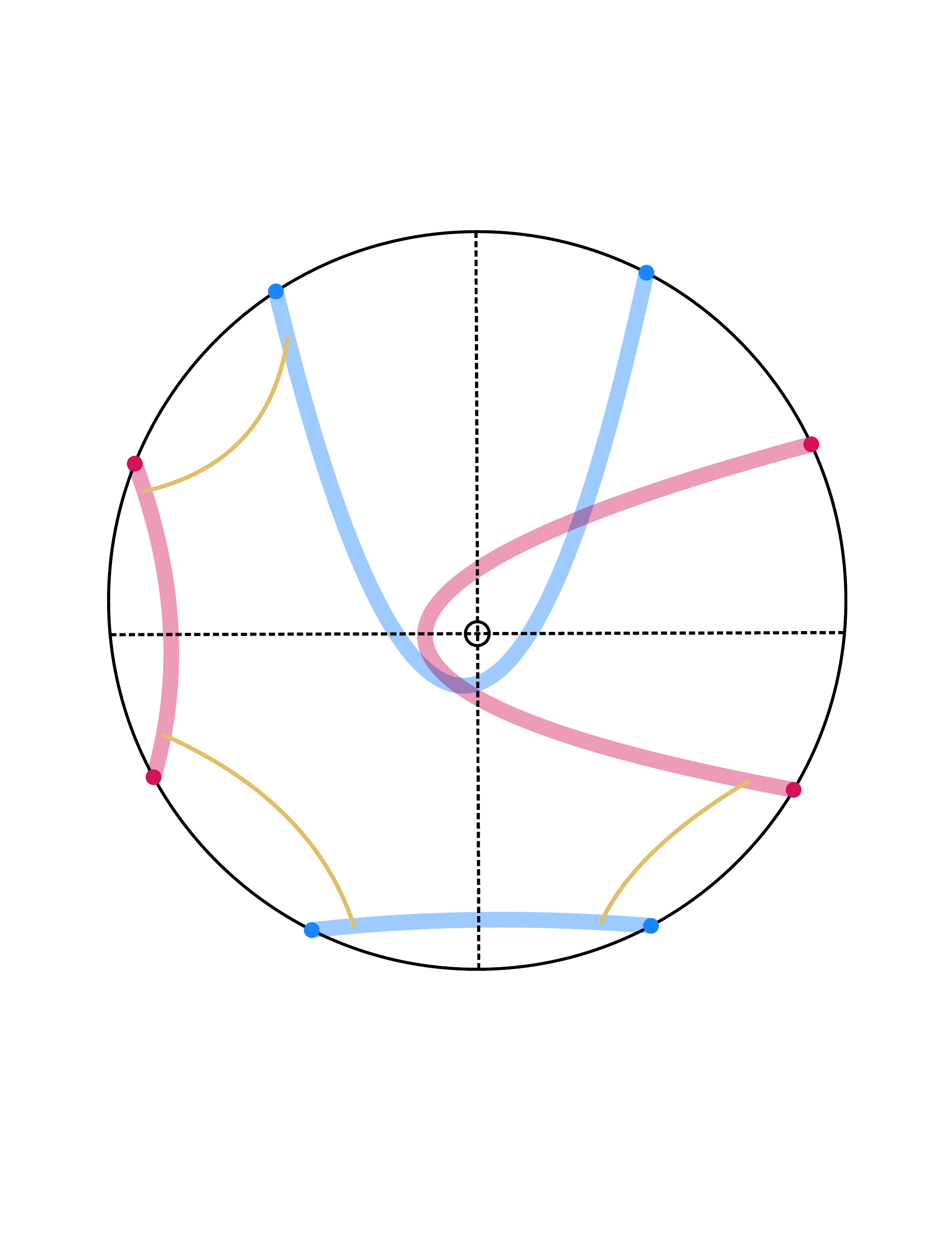}
		\caption{The three annuli in $D$, represented by arcs drawn in yellow.}
		\label{Annulus4}
\end{figure}

		Lastly, suppose we have the case where $D$ contains two chords of maximum possible length, as shown in Figure~\ref{Genus2D5}. Observe that each of the regions in this case are of types \textit{I, II, III, IV}, or \textit{IV}. By Proposition~\ref{NonIsotopicD}, these regions are not pocket regions and so $S_1$ and $S_2$ are not isotopic, completing the proof.
	\end{proof}

\begin{Prop}\label{DistinctG2Surfaces}
	Every punctured partitioned chord diagrams, up to isotopy of the chords, corresponding to a closed, connected, essential, orientable surface of genus 2 are non-isotopic.
\end{Prop}
	\begin{proof}
		We may assume that $S_1$ and $S_2$ have isotopic $S_C$'s. We deal with the remaining cases of $D$ containing at most one chord of maximum possible length since Lemma~\ref{Genus2Case1} proved the case for $D_1$ and $D_2$ each containing a maximum possible length chord. We need only consider three cases, shown in Figure~\ref{Genus2PPCD}, that these are the remaining possibilities of $D$ for each of the six possibilities for $D_1$ and $D_2$.

\begin{figure}[h!]
\centering
	\begin{minipage}[b]{0.31\linewidth}
		\centering
		\includegraphics[viewport = 0 170 610 700, scale = 0.2, clip]{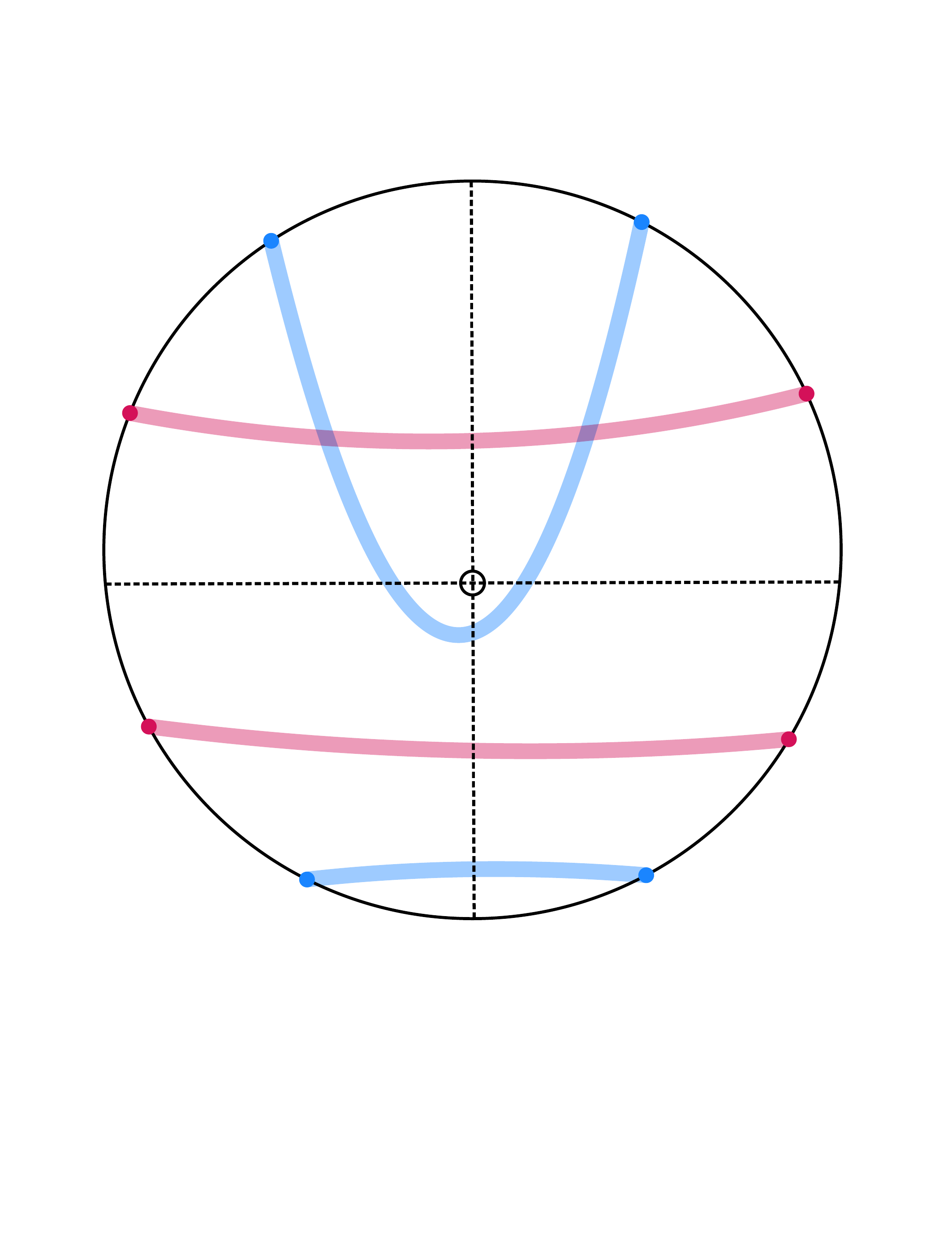}
	\end{minipage}
	\begin{minipage}[b]{0.31\linewidth}
		\centering
		\includegraphics[viewport = 0 170 610 700, scale = 0.2, clip]{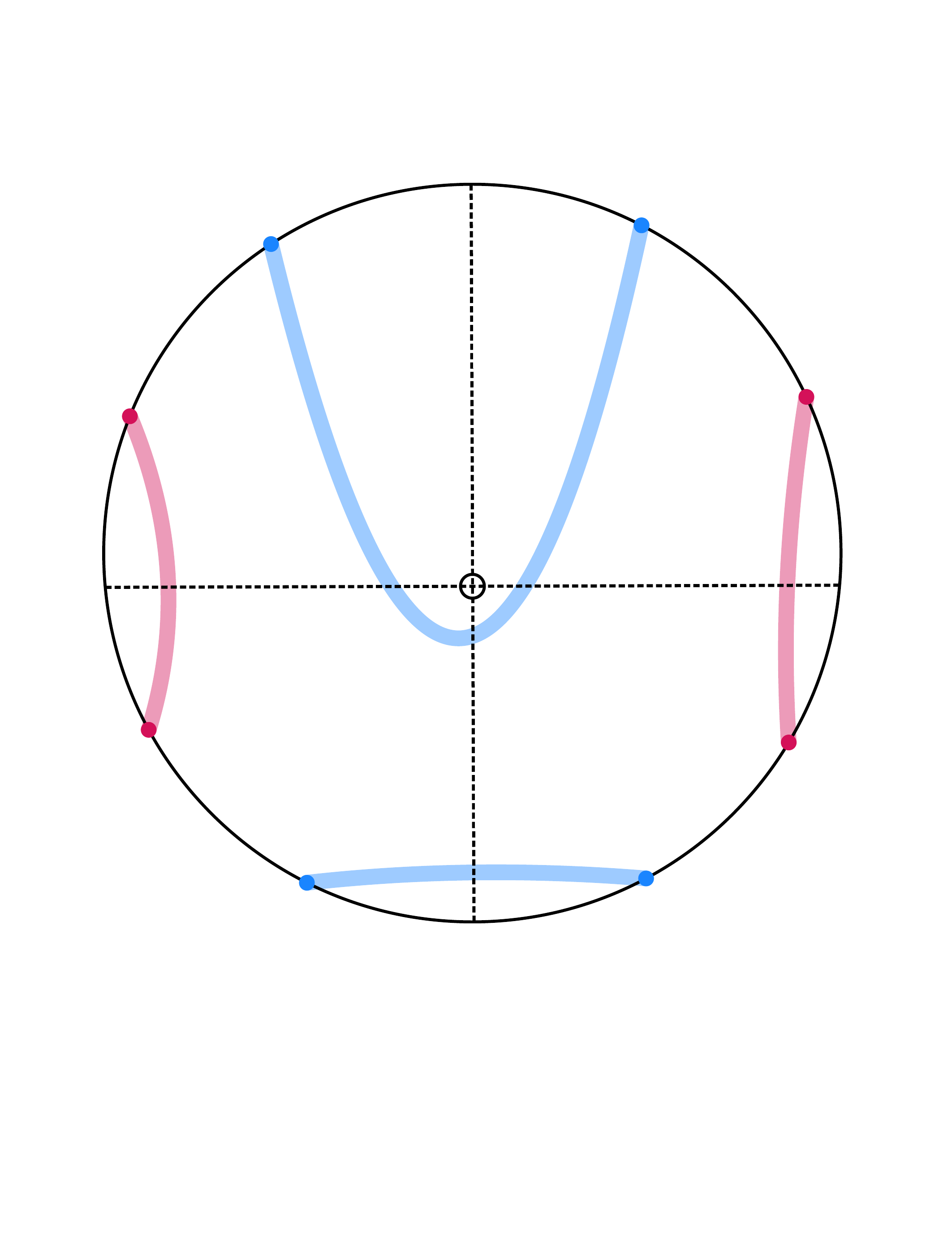}
	\end{minipage}
	\begin{minipage}[b]{0.31\linewidth}
		\centering
		\includegraphics[viewport = 0 170 610 700, scale = 0.2, clip]{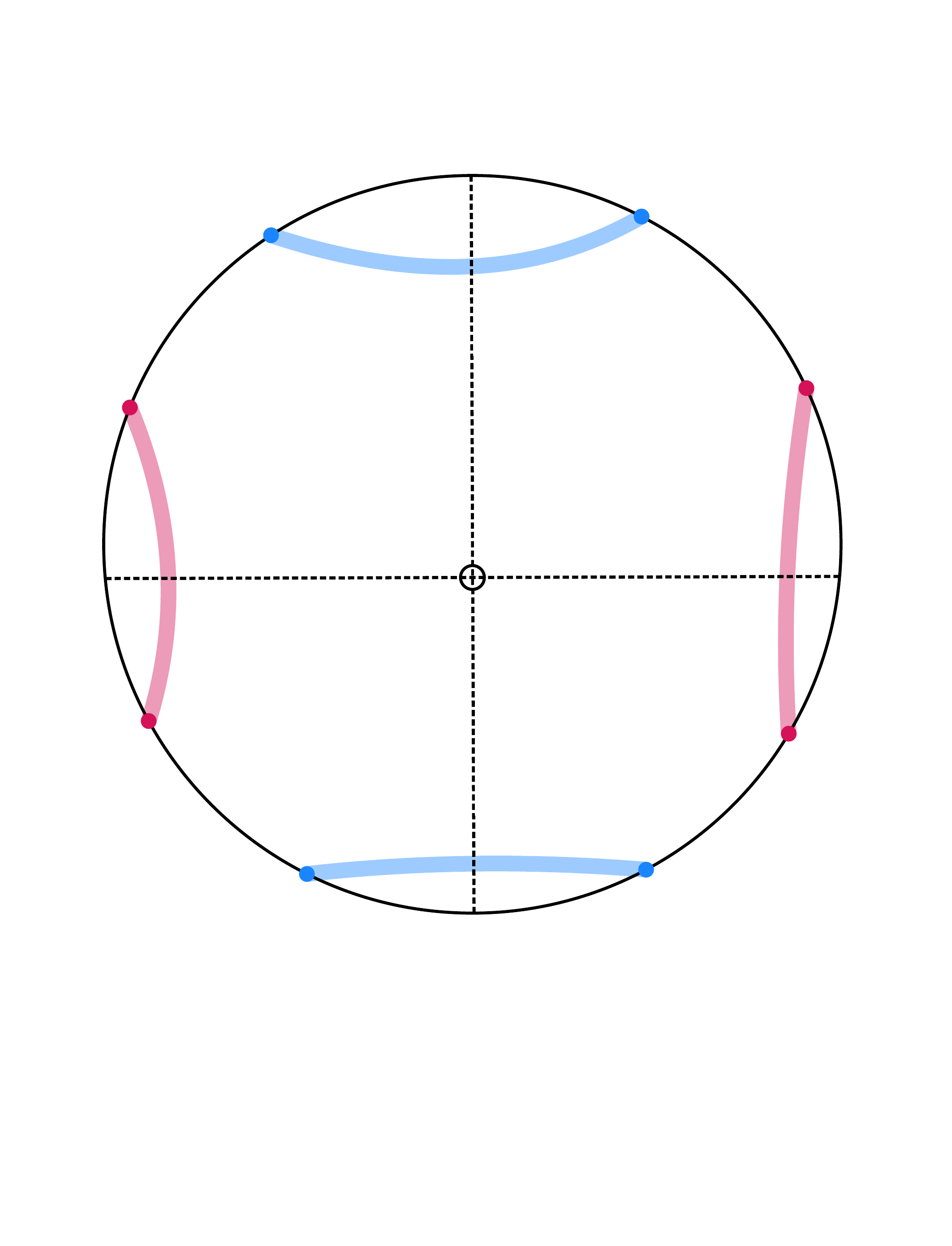}
	\end{minipage}
	\caption{The four possibilities of $D$ for $D_1$ and $D_2$ corresponding to genus 2 surfaces.}
	\label{Genus2PPCD}
\end{figure}

	The first case, shown in Figure~\ref{Genus2PPCD} (left), may be isotoped to the chord diagram shown in Figure~\ref{Genus2D2A} below. This corresponds to an isotopy of $S_2$ to have the innermost $S_C$ and resolving any intersections. We now proceed with these three cases.

\begin{figure}[h!]
	\centering
	\includegraphics[viewport = 0 160 610 700, scale = 0.28, clip]{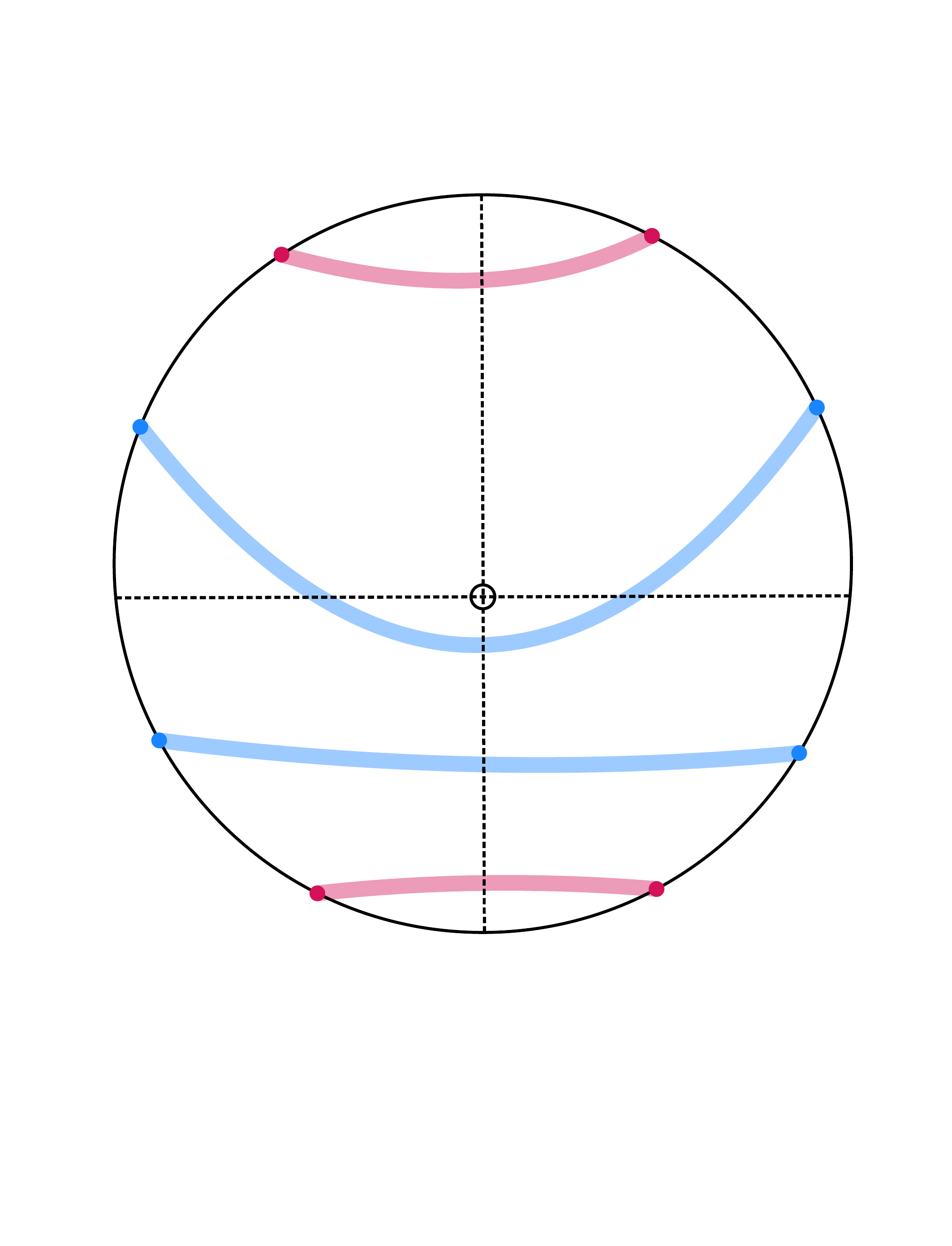}
	\caption{Resulting $D$ after isotopy.}
	\label{Genus2D2A}
\end{figure}

		In all of these cases, no chords intersect. Thus, there is only one region to consider: the region $R$ between the $S_C$'s of $S_1$ and $S_2$. Following Lemma~\ref{Genus2Case1}, take four arcs with the following properties: one base point of the arc lies on a chord of $D_1$ and the other base point of the arc lies on a chord of $D_2$ where the arc is completely contained in a single region of $D$ and parallel to a segment of the boundary circle (see Figure~\ref{CasesOfAnnuli}). Observe that two of these cases, the chords of $D_1$ and $D_2$ are disjoint. The complement of these four annuli decomposes $R$ such that one component which is bounded into two disjoint annuli of $S_1$ and two disjoint annuli of $S_2$. As in the previous case, this implies that $R$ is not a pocket region. Therefore, by Proposition~\ref{Waldhausen}, since none of these regions are pocket regions, then $S_1$ and $S_2$ are non-isotopic. 
		
		In the final case, observe that a resulting region is of type \textit{I}. By Proposition~\ref{NonIsotopicD}, this is not a pocket region. Thus, since no regions are pocket regions, $S_1$ and $S_2$ are non-isotopic, completing the proof.
	\end{proof}

\begin{figure}[h!]
\centering
	\begin{minipage}[b]{0.31\linewidth}
		\centering
		\includegraphics[viewport = 0 160 610 680, scale = 0.2, clip]{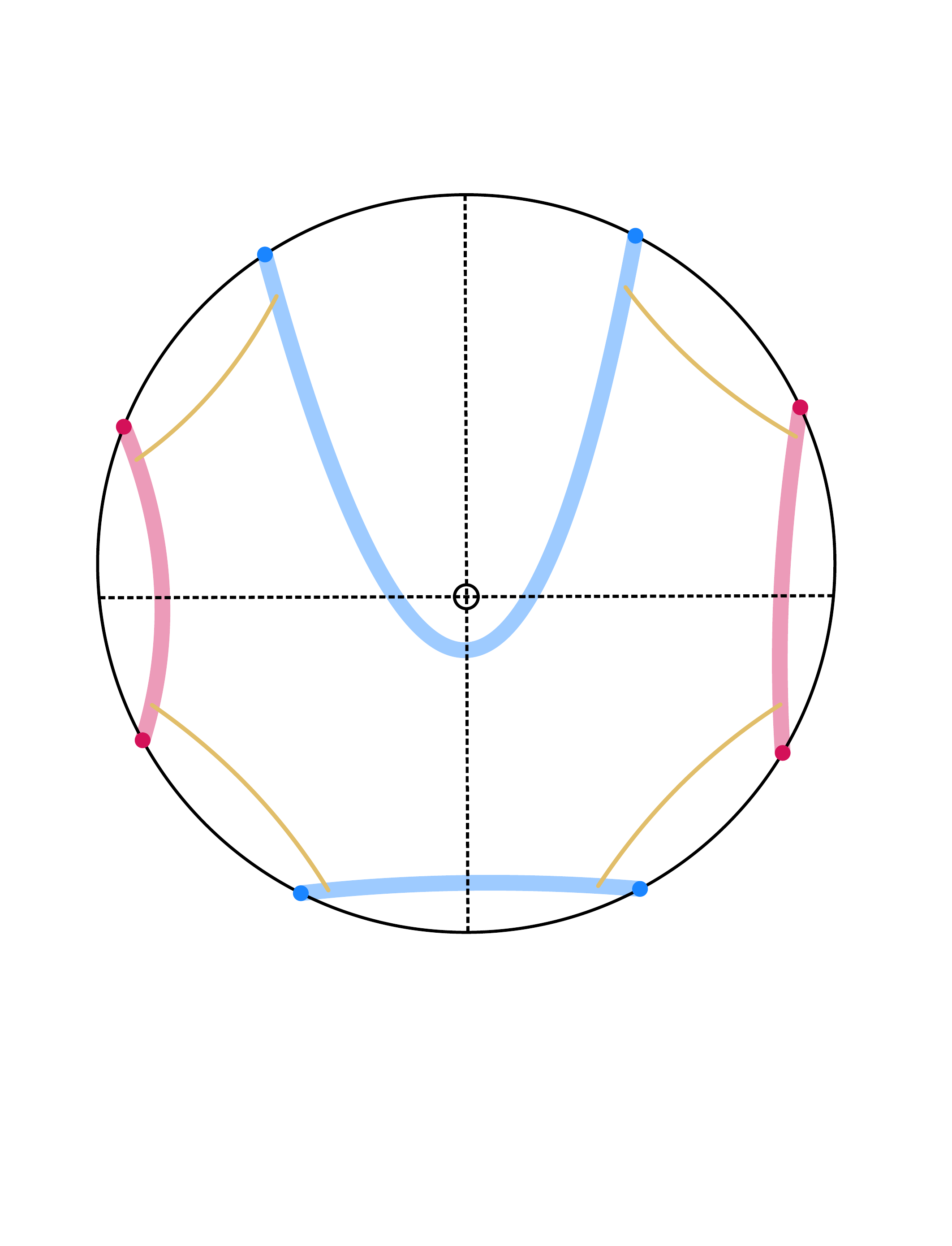}
	\end{minipage}
	\begin{minipage}[b]{0.31\linewidth}
		\centering
		\includegraphics[viewport = 0 165 610 700, scale = 0.2, clip]{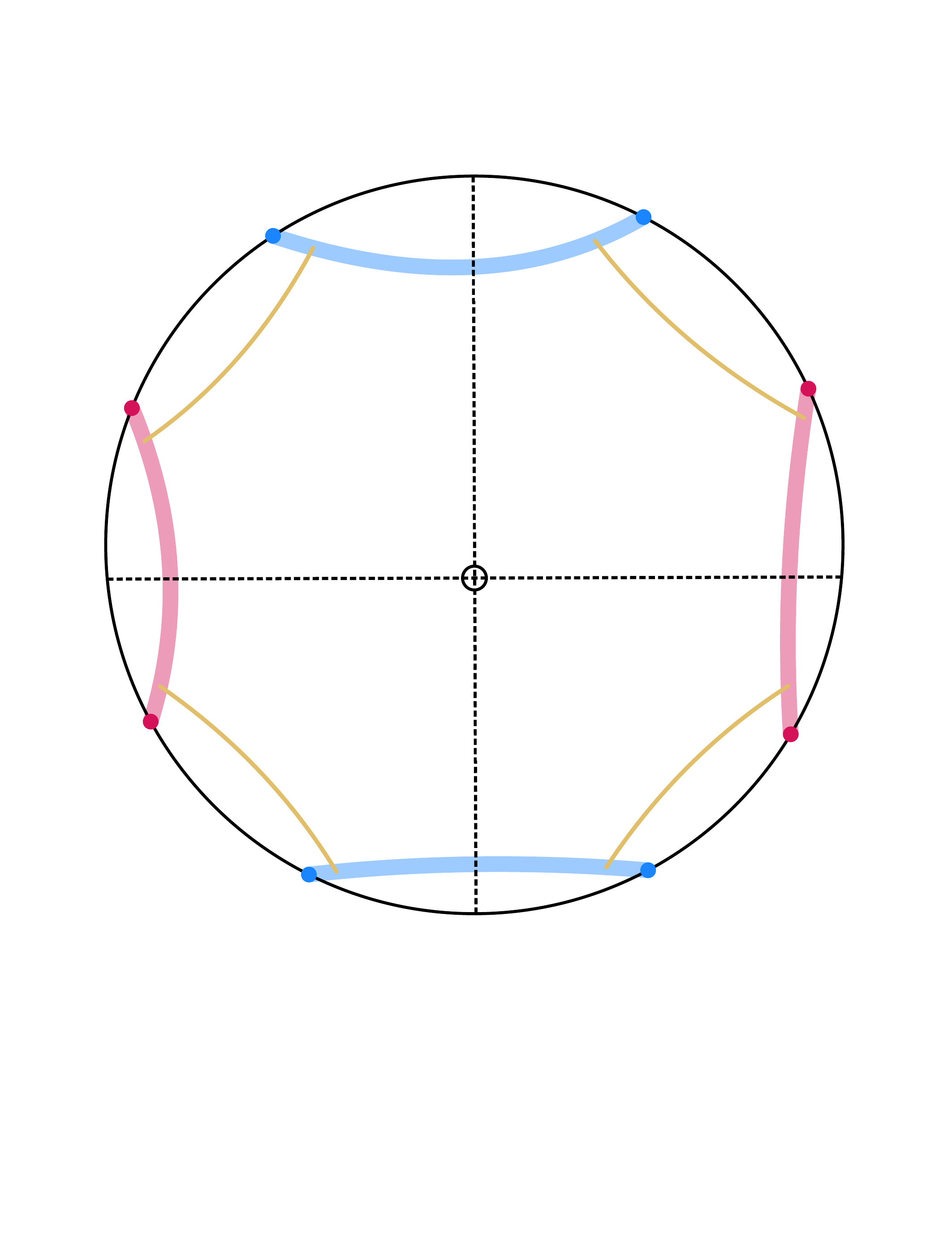}
	\end{minipage}
	\begin{minipage}[b]{0.31\linewidth}
		\centering
		\includegraphics[viewport = 0 160 610 680, scale = 0.2, clip]{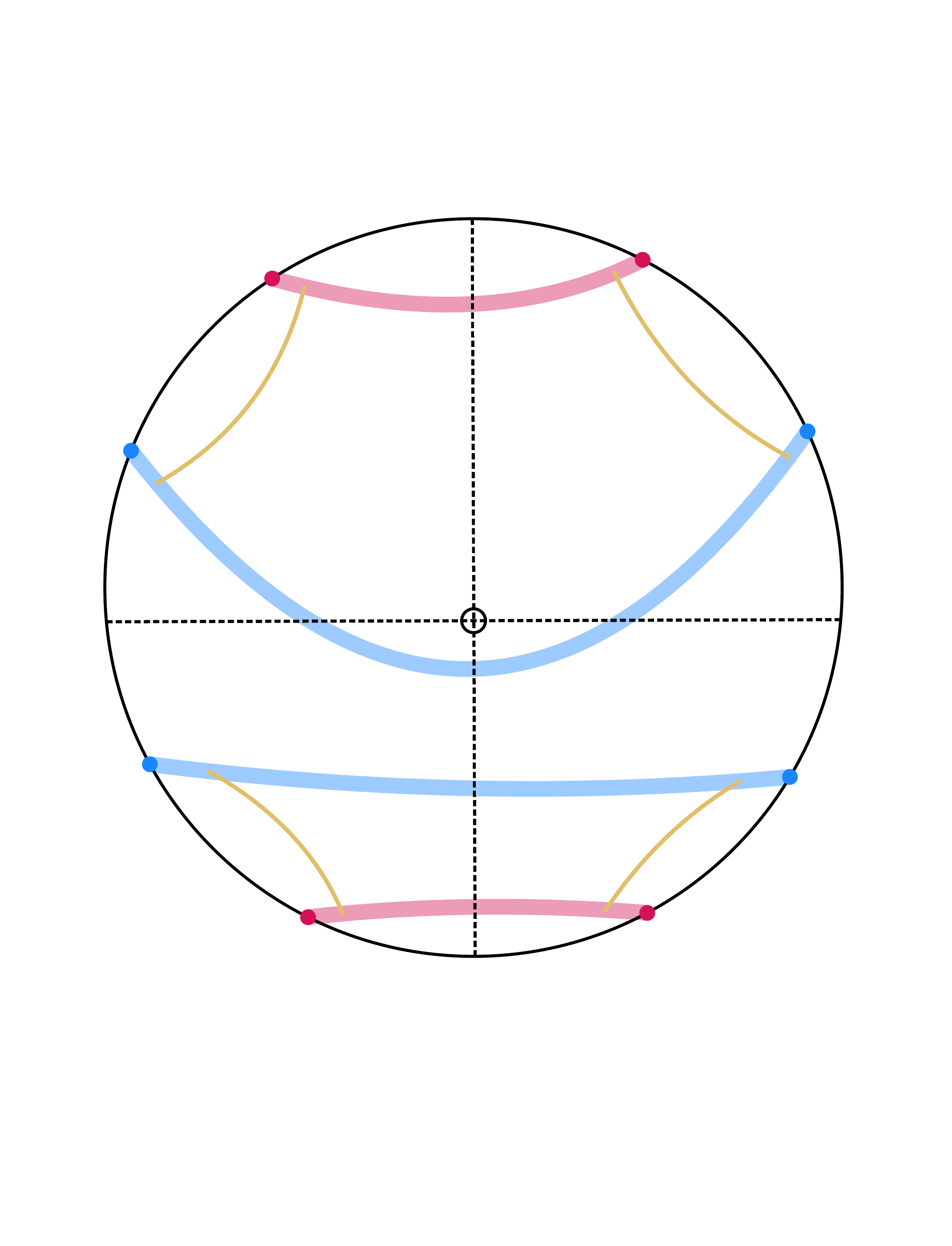}
	\end{minipage}
	\caption{The three cases of $D$ along with the annuli drawn in yellow.}
	\label{CasesOfAnnuli}
\end{figure}

We now prove Theorem~\ref{main}.

\begin{proof}[Proof of Theorem~\ref{main}]
	We only consider surfaces which are closed, connected, essential and orientable. Propositions~\ref{G2Case}, ~\ref{MaxPPCDsGeneral}, and~\ref{MaxPPCDs} shows that there are at most $12$ genus $2$ surfaces and $8\phi(g - 1)$ genus greater than $2$ surfaces, respectively. By Propositions~\ref{DistinctG2Surfaces} and~\ref{NonIsotopicD}, these surfaces are non-isotopic, proving the result. 
\end{proof}

\textit{Remark}. The results in Theorem~\ref{main} were made possible from experimental data obtained for the $(4,3,3,3)$ pretzel knot via the computer program in \cite{dunfield2020counting}. The surface counts in the complement of this knot were given by the generating function $$B_M = \frac{-2x^4 + 8x^3 - 10x^2 + 12x}{x^4 - 4x^3 + 6x^2 - 4x + 1}.$$ The sequence, $a_g$, of the number of closed, connected, essential, orientable essential surfaces of genus $g$ starting at $g = 2$, coming from $B_M$ were as follows: $$a_g = [12,8,16,16,32,16,48,32,48,32,80,32,96,48,64,64,128,48,144,64,...],$$ which helped motivate Theorem~\ref{main}.

%
%
%
\section{Acknowledgements}
I would like to thank my advisor Nathan Dunfield for this interesting question and for the guidance and stimulating discussions throughout this project. I would also like to thank Chaeryn Lee for helpful discussions, reading an early draft of this paper, and giving helpful advice to improve it. This work was partially supported under US National Science Foundation grant DMS-1811156.

\bibliographystyle{amsalpha}
\bibliography{bibliography}

\providecommand{\bysame}{\leavevmode\hbox to3em{\hrulefill}\thinspace}
\providecommand{\MR}{\relax\ifhmode\unskip\space\fi MR }
\providecommand{\MRhref}[2]{%
  \href{http://www.ams.org/mathscinet-getitem?mr=#1}{#2}
}
\providecommand{\href}[2]{#2}
\begin{thebibliography}{DGR20}

\bibitem[AHT02]{AHW2002}
Ian Agol, Joel Hass, and William Thurston, \emph{The computational complexity
  of knot genus and spanning area}, Transactions of the American Mathematical
  Society \textbf{358} (2002).

\bibitem[DGR20]{dunfield2020counting}
Nathan~M. Dunfield, Stavros Garoufalidis, and J.~Hyam Rubinstein,
  \emph{Counting essential surfaces in 3-manifolds}, preprint (2020),
  \texttt{https://arxiv.org/abs/2007.10053}.

\bibitem[DL05]{JADeLoera}
Jes{\'u}s~A. De~Loera, \emph{The many aspects of counting lattice points in
  polytopes}, Mathematische Semesterberichte \textbf{52} (2005), no.~2,
  175--195.

\bibitem[FO84]{FLOYD1984117}
William Floyd and Ulrich Oertel, \emph{Incompressible surfaces via branched
  surfaces}, Topology \textbf{23} (1984), no.~1, 117--125.

\bibitem[HT85]{Hatcher1985IncompressibleSI}
Allen Hatcher and William~P. Thurston, \emph{Incompressible surfaces in
  2-bridge knot complements}, Inventiones mathematicae \textbf{79} (1985),
  225--246.

\bibitem[HTT17]{MR3658176}
Joel Hass, Abigail Thompson, and Anastasiia Tsvietkova, \emph{The number of
  surfaces of fixed genus in an alternating link complement}, Int. Math. Res.
  Not. IMRN (2017), no.~6, 1611--1622. \MR{3658176}

\bibitem[JO84]{JACO1984195}
William Jaco and Ulrich Oertel, \emph{An algorithm to decide if a 3-manifold is
  a haken manifold}, Topology \textbf{23} (1984), no.~2, 195--209.

\bibitem[KM12]{MR2916295}
Jeremy Kahn and Vladimir Markovi\'{c}, \emph{Counting essential surfaces in a
  closed hyperbolic three-manifold}, Geom. Topol. \textbf{16} (2012), no.~1,
  601--624. \MR{2916295}

\bibitem[Lee21]{lee2021essential}
Chaeryn Lee, \emph{Essential surfaces in the exterior of k13n586}, preprint
  (2021), \texttt{https://arxiv.org/abs/2110.10231}.

\bibitem[Lyo71]{Lyon1971}
Herbert~C. Lyon, \emph{Incompressible surfaces in knot spaces}, Transactions of
  the American Mathematical Society \textbf{157} (1971), 53--62.

\bibitem[Oer84a]{oertel1984}
Ulrich Oertel, \emph{Closed incompressible surfaces in complements of star
  links}, Pacific Journal of Mathematics \textbf{111} (1984), no.~1, 209--230.

\bibitem[Oer84b]{Oertel1984385}
Ulrich Oertel, \emph{Incompressible branched surfaces}, Inventiones
  Mathematicae \textbf{75} (1984), 385--410.

\bibitem[Thu82]{bams1183548782}
William~P. Thurston, \emph{{Three dimensional manifolds, Kleinian groups and
  hyperbolic geometry}}, Bulletin (New Series) of the American Mathematical
  Society \textbf{6} (1982), no.~3, 357 -- 381.

\bibitem[Tol78]{Tollefson1978INVOLUTIONS}
Jeffrey~L. Tollefson, \emph{Involutions of seifert fiber spaces}, Pacific
  Journal of Mathematics \textbf{74} (1978), 519--529.

\bibitem[Tol95]{ojm1200786486}
\bysame, \emph{{Isotopy classes of incompressible surfaces in irreducible
  3-manifolds}}, Osaka Journal of Mathematics \textbf{32} (1995), no.~4, 1087
  -- 1111.

\bibitem[Wal68]{MR224099}
Friedhelm Waldhausen, \emph{On irreducible {$3$}-manifolds which are
  sufficiently large}, Ann. of Math. (2) \textbf{87} (1968), 56--88.
  \MR{224099}

\end{thebibliography}

\end{document}